\documentclass{amsart}

\usepackage{amsmath}
\usepackage{amssymb, amscd, graphicx}
\usepackage{pst-all}

\psset{unit=1mm}
\psset{fillcolor=white}
\psset{dotsep=1.5pt}
\psset{dash=1.5pt 1.5pt}
\psset{linewidth=0.8pt}
\psset{arrowsize=3.5pt}
\psset{doublesep=1pt}
\psset{coilwidth=1.2mm}
\psset{coilheight=2}
\psset{coilaspect=20}
\psset{coilarm=2}
\newpsstyle{double}{linewidth=0.5pt, doubleline=true}
\newpsstyle{etc}{linestyle=dotted}
\newpsstyle{exist}{linestyle=dashed}
\newpsstyle{thin}{linewidth=0.5pt}
\newpsstyle{thinexist}{linewidth=0.5pt, linestyle=dashed}

\numberwithin{equation}{section}

\theoremstyle{plain} 
\newtheorem{prop}[equation]{Proposition}
\newtheorem{coro}[equation]{Corollary}
\newtheorem{lemm}[equation]{Lemma}

\theoremstyle{definition}

\newtheorem{defi}[equation]{Definition}
\newtheorem{nota}[equation]{Notation}
\newtheorem{rema}[equation]{Remark}

\newcommand\BOX[1]{\begin{minipage}{105mm}{#1}\end{minipage}}

\newcounter{ITEM}
\newcommand\ITEM[1]{\setcounter{ITEM}{#1}\leavevmode\hbox{\rm(\roman{ITEM})}}

\renewcommand\aa{a}
\newcommand\AAA{\mathcal{A}}

\newcommand\Cat{\mathcal{C}\hspace{-0.3ex}a\hspace{-0.1ex}t}
\newcommand\cc{c}
\newcommand\CCC{\mathcal{C}}
\newcommand\CCCi{\CCC^{\!\scriptscriptstyle\times}}

\newcommand\dd{d}
\newcommand\Der[2]{#1_{\HS{-0.2}/_{\HS{-0.2}#2}}}
\renewcommand\div{\prec}
\newcommand\Div{\mathrm{Div}}
\newcommand\dive{\preccurlyeq\nobreak}

\newcommand\diveS{\mathrel{\preccurlyeq_{\!{}_{\SSS}}}}
\newcommand\divS{\mathrel{\prec_{\!{}_{\SSS}}}}

\newcommand\divve[1]{\mathrel{\preccurlyeq_{#1}}}

\newcommand\ee{e}
\newcommand\eqir{\mathbin{=^{\!\scriptscriptstyle\times}}}
\newcommand\eqirS{\mathrel{=^{\!\scriptscriptstyle\times}_{\!\SSS}}}
\newcommand\etc{\pdots}
\newcommand\ew{\varepsilon}

\newcommand\ff{f}
\newcommand\FF{F}
\newcommand\fft{\hat{f}}

\renewcommand\ge{\geqslant}
\renewcommand\gg{g}
\newcommand\GG{G}
\newcommand\GGG{\mathcal{G}}

\newcommand\hh{h}
\newcommand\HH{H}
\newcommand\HHH{\mathcal{H}}
\newcommand\hht{\hat{h}}
\newcommand\HHu{\underline{H}}
\newcommand\HS[1]{\hspace{#1ex}}

\newcommand\id[1]{1_{\!#1}}
\newcommand\Id[1]{\boldsymbol{1}_{\!#1}}
\newcommand\ii{i}
\newcommand\II{I}
\newcommand\III{\mathcal{I}}
\newcommand\IIIS{\III_{\SSS}}
\newcommand\inv{^{-1}}

\newcommand\jj{j}
\newcommand\JJ{J}
\newcommand\JJJ{\mathcal{J}}
\newcommand\JJJS{\JJJ_{\SSS}}

\newcommand\kk{k}
\newcommand\KK{K}

\renewcommand\le{\leqslant}

\newcommand\LGG[2]{\Vert#2\Vert_{#1}}
\newcommand\LT[2]{\Vert#1\Vert_{#2}}

\newcommand\MM{M}

\newcommand\NN{N}
\newcommand\noteqir{\mathbin{\not=^{\!\scriptscriptstyle\times}}}

\newcommand\Obj{\mathcal{O}\HS{-0.15}b\HS{-0.25}j}
\newcommand\OP{\,{\scriptstyle\bullet}\,}
\newcommand\Ord{\mathbf{Ord}}

\newcommand\pc{\HS{0.2}\vert\HS{0.2}}
\newcommand\pdots{\hspace{0.2ex}{\cdot}{\cdot}{\cdot}\hspace{0.2ex}}
\newcommand\pp{p}
\newcommand\Pref[1]{\mathrel{{\le}_{#1}}\nobreak}\newcommand\PRESp[2]{\langle#1\,\vert\, #2\rangle^{\scriptscriptstyle\!+}\!}

\newcommand\qq{q}
\newcommand\quot{\!\mathord{/}\!}

\newcommand\resp{\mbox{\it resp}}
\newcommand\rr{r}
\newcommand\RRR{\mathcal{R}}

\newcommand\seq[1]{#1}
\newcommand\Seq[2]{#1^{[#2]}}
\newcommand\seqq[2]{#1\pc#2}
\newcommand\seqqq[3]{#1\pc#2\pc#3}
\newcommand\seqqqq[4]{#1\pc#2\pc#3\pc#4}
\newcommand\seqqqqq[5]{#1\pc#2\pc#3\pc#4\pc#5}
\newcommand\seqqqqqq[6]{#1\pc#2\pc#3\pc#4\pc#5\pc#6}

\newcommand\SSS{\mathcal{S}}
\newcommand\SSSg{\underline\SSS}
\newcommand\SSSi{\SSS^{\scriptscriptstyle\times}}
\newcommand\SSSs{\SSS^{\scriptstyle\sharp}}

\newcommand\tta{\mathtt{a}}
\newcommand\ttb{\mathtt{b}}

\newcommand\TTu{\underline{T}}

\newcommand\ud{\HS{0.2}\hbox{-}}

\def\VR(#1,#2){\vrule width0pt height#1mm depth#2mm}

\newcommand\wdots{, ...\HS{0.2},}
\newcommand\ww{w}

\newcommand\xx{x}
\newcommand\XXX{\mathcal{X}}

\newcommand\yy{y}

\title{Garside families and Garside germs} 

\author{Patrick DEHORNOY}
\address{Laboratoire de Math\'ematiques Nicolas Oresme,
CNRS UMR 6139, Universit\'e de Caen, 14032 Caen, France}
\email{patrick.dehornoy@unicaen.fr}
\urladdr{www.math.unicaen.fr/\!\hbox{$\sim$}dehornoy}

\author{Fran\c cois DIGNE}
\address{Laboratoire Ami\'enois de Math\'ematique
Fondamentale et Appliqu\'ee, CNRS UMR 7352, Universit\'e de Picardie
Jules-Verne, 80039 Amiens, France}
\email{digne@u-picardie.fr}
\urladdr{www.mathinfo.u-picardie.fr/digne/}

\author{Jean MICHEL}
\address{Institut Math\'ematique de Jussieu, CNRS UMR 7586, Universit\'e Denis 
Diderot Paris 7,
175, rue du Chevaleret, 75013 Paris, France}
\email{jmichel@math.jussieu.fr}
\urladdr{www.math.jussieu.fr/\!\hbox{$\sim$}jmichel/}

\keywords{Garside family, Garside monoid, Garside group, germ, greedy normal 
form, Coxeter group, braid monoid}

\subjclass{20M05, 18B40, 20F10, 20F36}

\thanks{Work partially supported by the ANR grant TheoGar ANR-08-BLAN-0269-02}

\begin{document}

\begin{abstract}
Garside families have recently emerged as a relevant context for extending 
results involving Garside monoids and groups, which themselves extend the 
classical theory of (generalized) braid groups. Here we establish various 
characterizations of Garside families, that is, equivalently, various criteria 
for establishing the existence of normal decompositions of a certain type.
\end{abstract}

\maketitle

In 1969, F.A.\,Garside~\cite{Gar} solved the word and conjugacy problems of 
Artin's braid groups by using convenient monoids. This approach was pursued 
\cite{Adj, ElM, Thu, Eps, Cha} and extended in several steps, first to 
Artin-Tits groups of spherical type~\cite{BrS, Dlg}, then to a larger family 
of groups now known as Garside groups~\cite{Dfx, Dgk, Dgo}. More recently, it 
was realized that going to a categorical context allows for capturing further 
examples~\cite{DiM, BesD, Dht, McC}, and a coherent theory now emerges around 
a central unifying notion called a Garside family. The aim of this paper is to 
present the main basic results of this approach. A more comprehensive text, 
including examples and many further developments, will be found 
in~\cite{Garside}. Algorithmic issues are addressed in~\cite{Dig}.

The philosophy of Garside's theory as developed in the past decades is that, 
in some cases, a group can be realized as a group of fractions for a monoid 
and that the divisibility relations of the latter provide  a lot  of 
information about the group. The key technical ingredient in the approach is a 
certain distinguished decomposition for the elements of the monoid and the 
group in terms of some fixed (finite) family, usually called the greedy normal 
form. Our current approach consists in analyzing the abstract mechanism 
underlying the greedy normal form and developing it in the general context of 
what we call Garside families. The leading principle is that, with Garside 
families, one should retrieve all results about Garside monoids and groups at 
no extra cost. 

In the current paper, we concentrate on one fundamental question, namely 
characterizing Garside families. As the latter are defined to be those 
families that guarantee the existence of the normal form, this exactly amounts 
to establishing various (necessary and sufficient) criteria for this 
existence. Two types of characterizations will be established here: extrinsic 
characterizations consist in recognizing whether a subfamily of a given 
category is a Garside family, whereas intrinsic ones consist in recognizing 
whether an abstract family (more precisely, a precategory) generates a 
category in which it embeds as a Garside family. 

Beyond the results themselves, one of our goals is to show that the new 
framework, which properly extends those previously considered in literature, 
works efficiently and provides arguments that are both simple and natural. In 
particular, appealing to a categorical framework introduces no additional 
complexity and helps in many places to develop a better intuition of the 
situation. Among the specific features in the current approach are the facts 
that it is compatible with the existence of nontrivial invertible elements, 
and requires no Noetherianity assumption. 

The paper is organized as follows. In Section~\ref{S:Cat}, we quickly list the 
basic notions about categories and the derived divisibility relations needed 
for further developments. In Section~\ref{S:Normal}, we introduce the central 
notion of an $\SSS$-normal decomposition for an element of a category, and 
define a Garside family to be a family~$\SSS$ such that every element of the 
ambient category admits at least one $\SSS$-normal decomposition. Next, in 
Section~\ref{S:RecGar}, we establish what we called various extrinsic 
characterizations of Garside families, mainly based on closure properties and 
on properties of so-called head functions. Then, in Section~\ref{S:Germs}, we 
turn to intrinsic definitions and develop the notion of a germ, which, as the 
name suggests, is a sort of partial structure from which a category can be 
constructed. In Section~\ref{S:RecGerm}, we establish various 
characterizations of Garside germs, which are those germs that embed as 
Garside families in the category they generate. Finally, in 
Section~\ref{S:Appli},  we  give  an  interesting  application of the previous 
results,  namely the construction of the classical and dual braid monoids 
associated with a Coxeter or a complex reflection group.

%%%%%%%%
\section{Categories and divisibility}
\label{S:Cat}

In this preliminary section, we fix some terminology. As mentioned in the 
introduction, it is convenient to work in a category framework that we 
describe here.

%%%%%
\subsection{The category framework}
\label{SS:Cat}

The general context is that of categories, which are seen here just as monoids 
with a conditional, not necessarily everywhere defined product. A 
\emph{precategory} is a family of elements with two objects called ``source'' 
and ``target'' attached to each element, and a \emph{category} is a 
precategory equipped with a partial binary product such that $\ff\gg$ exists 
if and only if the target of~$\ff$ coincides with the source of~$\gg$. In 
addition, the product is assumed to be associative whenever defined and, for 
each object~$\xx$, there exists an identity-element~$\id\xx$ attached 
to~$\xx$, so that $\id\xx \gg = \gg = \gg\id\yy$ holds for every~$\gg$ with 
source~$\xx$ and target~$\yy$. A monoid is the special case of a category when 
there is only one object, so that the product is always defined. It is 
convenient to represent an element~$\gg$ with source~$\xx$ and target~$\yy$ 
using an arrow as in \begin{picture}(13,5)
\put(0,0){$\xx$}
\put(11,0){$\yy$}
\pcline{->}(3,0.5)(10,0.5)
\taput{$\gg$}
\end{picture}.

\begin{defi}
A category~$\CCC$ is said to be \emph{left-cancellative} (\resp.\ 
\emph{right-cancellative}) if $\ff \gg = \ff \gg'$ (\resp.\ $\gg\ff = \gg' 
\ff$) implies $\gg = \gg'$ in~$\CCC$. It is called \emph{cancellative} is it 
is both left- and right-cancellative.
\end{defi} 

 (In other words, in the usual language of categories~\cite{Mac}, a category 
is left-cancellative if all morphisms are epimorphisms, and it is 
right-cancellative if all morphisms are monomorphisms.)  In the sequel, we 
shall always work with categories that are (at least) left-cancellative. This 
implies in particular that there exists a unique, well defined notion of 
invertible element.

\begin{lemm}
\label{L:Inv}
If $\CCC$ is a left-cancellative category, an element~$\gg$ of~$\CCC$ 
with source~$\xx$ and target~$\yy$ has a left-inverse, that is, there 
exists~$\ff$ satisfying $\ff\gg = \id\yy$, if and only if it has a 
right-inverse, that is, there exists~$\ff$ satisfying $\gg \ff = \id\xx$.
\end{lemm}

\begin{proof}
Assume $\ff \gg = \id\yy$. Right-multiplying by~$\ff$, we deduce $\ff \gg \ff 
= \ff$, which, by left-cancellativity, implies $\gg \ff = \id\xx$. Similarly, 
$\gg \ff = \id\xx$ implies $\gg \ff \gg = \gg$, whence $\ff \gg = \id\yy$.
\end{proof}

If $\CCC$ is a left-cancellative category, we shall say that an element~$\gg$ 
of~$\CCC$ is \emph{invertible} if it admits a left- and right-inverse, 
naturally denoted by~$\gg\inv$. Note that the product of two invertible 
elements is always invertible, that the inverse of an invertible element is 
invertible, and that an identity-element~$\id\xx$ is invertible (and 
equal to its inverse). Thus the invertible elements of a left-cancellative 
category~$\CCC$ make a subgroupoid of~$\CCC$. 

\begin{nota}
For $\CCC$ a left-cancellative category, we denote by~$\CCCi$ the family of 
all invertible elements of~$\CCC$ and, for every subfamily~$\SSS$ of~$\CCC$, 
we put $\SSSs = \SSS\CCCi \cup \CCCi$.
\end{nota}

In the above notation, as everywhere in the paper, if $\XXX_1, \XXX_2$ are 
subfamilies of a category~$\CCC$, we denote by~$\XXX_1 \XXX_2$ the family of 
all elements of~$\CCC$ that can be expressed, in at least one way, as~$\gg_1 
\gg_2$ with $\gg_\ii \in \XXX_\ii$, $\ii = 1,2$. So, for instance, $\SSS\CCCi$ 
is the family of all elements obtained by right-multiplying an element 
of~$\SSS$ by an invertible element.  In such a context, we naturally write 
$\XXX^\rr$ for $\XXX \pdots \XXX$, $\rr$~factors. 

%%%%
\subsection{The divisibility relations}
\label{SS:Div}

\begin{defi}
Assume that $\CCC$ is a left-cancellative category. For $\ff, \gg$ in~$\CCC$, 
we say that $\ff$ is a \emph{left-divisor} of~$\gg$, or, equivalently, that 
$\gg$ is a \emph{right-multiple} of~$\ff$, written $\ff \dive \gg$, if there 
exists~$\gg'$ in~$\CCC$ satisfying $\ff \gg' = \gg$. Symmetrically, we say 
that $\ff$ is a \emph{right-divisor} of~$\gg$, or, equivalently, that $\gg$ is 
a \emph{left-multiple} of~$\ff$, if there exists~$\gg'$ in~$\CCC$ satisfying 
$\gg = \gg' \ff$. 
\end{defi}

\noindent
\parbox{\textwidth}{\rightskip40mm
\VR(4,0) \quad In terms of arrows and diagrams, $\ff$ being a left-divisor 
of~$\gg$ corresponds to the existence of an arrow~$\gg'$ making the diagram on 
the right commutative.\qquad\VR(0,2)
\hfill
\rlap
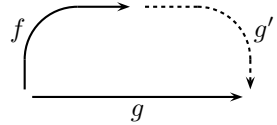
{%
\begin{picture}(32,0)(-7,-1)
\pcline(0,1)(0,5)
\psarc(7,5){7}{90}{180}
\pcline{->}(7,12)(14,12)
\put(-2,8){$\ff$}
\pcline{->}(1,0)(29,0)
\tbput{$\gg$}
\pcline[style=exist](16,12)(23,12)
\psarc[style=exist](23,5){7}{0}{90}
\pcline[style=exist]{->}(30,5)(30,1)
\put(30.5,8){$\gg'$}
\end{picture}
}}

\begin{nota}
For $\CCC$ a left-cancellative category and $\ff, \gg$ in~$\CCC$, we write $\ff \eqir \gg$ if there exists an invertible element~$\ee$ satisfying $\ff\ee = \gg$, and $\ff \div \gg$ for the conjunction of~$\ff \dive \gg$ and $\gg \not\dive \ff$. 
\end{nota}

\begin{lemm}
\label{L:DivOrder}
If $\CCC$ is a left-cancellative category, the relation~$\dive$ is a partial preordering on~$\CCC$, and the conjunction of $\ff \dive \gg$ and $\gg \dive \ff$ is equivalent to~$\ff \eqir \gg$.
\end{lemm}

\begin{proof}
For every~$\gg$ with target~$\yy$, we have $\gg = \gg \id\yy$, so $\gg \dive 
\gg$ always holds. Next, if we have both $\ff \gg' = \gg$ and $\gg \hh' = 
\hh$, we deduce $\ff (\gg' \hh') = (\ff \gg') \hh' = \gg \hh' = \hh$, so the  
conjunction of~$\ff \dive \gg$ and $\gg \dive \hh$ implies $\ff \dive \hh$. So 
$\dive$ is reflexive and transitive.

Assume $\ff \dive \gg$ and $\gg \dive \ff$. Then there exist~$\ff', \gg'$ 
satisfying $\ff \gg' = \gg$ and $\gg \ff' = \ff$. We deduce $\ff (\gg' \ff') = 
(\ff \gg') \ff' = \gg \ff' = \ff$, whence $\gg' \ff' = \id\yy$ where $\yy$ is 
the target of~$\ff$. It follows that $\ff'$ and $\gg'$ must be invertible, 
which 
implies $\ff \eqir \gg$. Conversely, $\ff \ee = \gg$ with $\ee \in \CCCi$ 
implies $\ff = \gg \ee\inv$, so $\ff \eqir \gg$ implies $\ff \dive \gg$ 
and~$\gg \dive \ff$.
\end{proof}

It follows that the left-divisibility relation of a category~$\CCC$ is a 
partial ordering on~$\CCC$ if and only if the relation~$\eqir$ is equality, 
that is, if and only if, the only invertible elements of~$\CCC$ are the 
identity-elements. In this case, we shall say that $\CCC$ has no 
\emph{nontrivial} invertible element.

%%%%
\subsection{Paths and free categories}
\label{SS:Paths}

We now fix some terminology and notation for paths, which are the natural 
counterparts of sequences (or words) in the context of a monoid. The notion 
does not depend on the product of the category  but only on the sources and 
targets of the considered elements,  and it will be useful to define it in the 
more general context of a precategory, which is like a category but with no 
product and identity-elements. 

\begin{defi}
\label{D:Path}
Assume that $\XXX$ is a precategory. For $\qq \ge 1$, a \emph{path in~$\XXX$}, 
or \emph{$\XXX$-path}, of length~$\qq$ is a sequence $(\gg_1 \wdots \gg_\qq)$ 
of elements of~$\XXX$ such that, for $1 < \kk \le \qq$, the target 
of~$\gg_{\kk-1}$ coincides with the source of~$\gg_\kk$. In this case, the 
source (\resp.\ target) of the path is defined to be the source of~$\gg_1$ 
(\resp.\ the target of~$\gg_\qq$). In addition, for every object~$\xx$, one 
defines an \emph{empty path}~$\ew_\xx$ associated with~$\xx$, whose source and 
target are~$\xx$ and whose length is zero. The family of all 
$\XXX$-paths of 
length~$\qq$ (\resp.\ all $\XXX$-paths) is denoted by~$\XXX^{[\qq]}$ (\resp.\ 
by~$\XXX^*$). If $\ww_1, \ww_2$ are $\XXX$-paths and the target of~$\gg_1$ 
coincides with the source of~$\gg_2$, the \emph{concatenation} of~$\ww_1$ 
and~$\ww_2$ is denoted by~$\ww_1 \pc \ww_2$.
\end{defi}

In the case of a set, that is, a precategory with one object, the condition 
about source and target vanishes, and an $\XXX$-path is simply a finite 
sequence of elements of~$\XXX$, or, equivalently, a word in the 
alphabet~$\XXX$.

\begin{nota}
By definition, an $\XXX$-path $(\gg_1 \wdots \gg_\qq)$ is the concatenation of 
the $\qq$ length one paths $(\gg_1) \wdots (\gg_\qq)$. Identifying the 
length one path~$(\gg)$ with its unique entry~$\gg$, we then find $(\gg_1 
\wdots \gg_\qq) = \seqqq{\gg_1}\etc{\gg_\qq}$, a simplified notation used in 
the sequel. 
\end{nota}

\begin{lemm}
\label{L:Free}
For every precategory~$\XXX$, the family~$\XXX^*$ equipped with concatenation 
and the empty paths~$\ew_\xx$ for~$\xx$ in~$\Obj(\XXX)$ is a category, still 
denoted by~$\XXX^*$. Every category~$\CCC$  including~$\XXX$ and  generated 
by~$\XXX$ is a homomorphic image of~$\XXX^*$: there exists a surjective 
functor of~$\XXX^*$ onto~$\CCC$ that extends the identity on~$\XXX$.
\end{lemm}

The proof is standard. Owing to Lemma~\ref{L:Free}, the category~$\XXX^*$ is 
called the \emph{free} category based on~$\XXX$ (or the free monoid in the 
case of a set). If $\CCC$ is a category and $\XXX$ is a subfamily of~$\CCC$, 
the unique functor of~$\XXX^*$ to~$\CCC$ that extends the identity maps 
on~$\XXX$ and~$\Obj(\XXX)$ is the evaluation map that associates to each 
$\XXX$-path $\seqqq{\gg_1}\etc{\gg_\qq}$ the product $\gg_1 \pdots \gg_\qq$ as 
computed using the product of~$\CCC$. In this case, we say that the path 
$\seqqq{\gg_1}\etc{\gg_\qq}$ is a \emph{decomposition} of the element $\gg_1 
\pdots \gg_\qq$ of~$\CCC$.

 If $\XXX$ is a precategory and $\RRR$ is a family of pairs of $\XXX$-paths of 
the form $\{\ww, \ww'\}$ with $\ww, \ww'$ sharing the same source and 
the same 
target, one denotes by~$\PRESp\XXX\RRR$ the category~$\XXX^* \quot \equiv$ 
where $\equiv$ is the smallest congruence on~$\XXX^*$ that includes~$\RRR$, 
that is, the smallest equivalence relation compatible with the product. If a 
pair $\{\ww, \ww'\}$ lies in~$\RRR$, the evaluations of the paths~$\ww$ 
and~$\ww'$ in the quotient-category $\PRESp\XXX\RRR$ coincide, and, therefore, 
it is customary to write $\ww = \ww'$ instead of $\{\ww, \ww'\}$ in this 
context, and to call such a pair a \emph{relation} of the presented 
category~$\PRESp\XXX\RRR$.  

%%%%%%%%%
\section{$\SSS$-normal decompositions}
\label{S:Normal}

We now arrive at a central topic in the current approach, namely the notion of 
an $\SSS$-normal decomposition for an element of a category.

%%%%
\subsection{The notion of an $\SSS$-greedy path}
\label{SS:Greedy}

The first step is the notion of an $\SSS$-greedy path, which captures the 
intuition that each entry in the path contains as much of~$\SSS$ as is 
possible.

\begin{defi}
Assume that  $\SSS$ is a subfamily of a left-cancellative category~$\CCC$, 
that is, $\SSS$ is included in~$\CCC$. A 
$\CCC$-path~$\seqqq{\gg_1}\etc{\gg_\qq}$ is called \emph{$\SSS$-greedy} if, 
for every~$\ii < \qq$, we have
\begin{equation}
\label{E:Greedy}
\forall\hh{\in}\SSS \ \forall\ff{\in}\CCC \ (\hh \dive \ff\gg_\ii \gg_{\ii+1} 
\Rightarrow \hh \dive \ff\gg_\ii).
\end{equation}
\end{defi}

Note that $\SSS$-greediness is a local property in that a path 
$\seqqq{\gg_1}\etc{\gg_\qq}$ is $\SSS$-greedy if and only if every length two 
subpath $\seqq{\gg_\ii}{\gg_{\ii+1}}$ is $\SSS$-greedy.

\noindent
\parbox{\textwidth}{\rightskip40mm
\VR(4,0) \quad In terms of diagrams, the fact that $\seqq{\gg_1}{\gg_2}$ is 
$\SSS$-greedy means that every diagram as aside splits: whenever $\hh$ 
left-divides~$\ff\gg_1\gg_2$, it left-divides $\ff\gg_1$, so there 
exists~$\ff'$ satisfying $\hh \ff' = \ff \gg_1$. Then the assumption $\hh \gg' 
= \ff \gg_1 \gg_2$ implies $\hh \ff' \gg_2 =\nobreak \hh \gg'$, whence $\ff' 
\gg_2 = \gg'$ by left-cancelling~$\hh$.\VR(0,2)
\hfill
\rlap
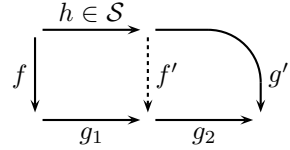
{%
\begin{picture}(33,0)(-6,-3)
\pcline{->}(1,0)(14,0)
\tbput{$\gg_1$}
\pcline{->}(16,0)(29,0)
\tbput{$\gg_2$}
\pcline{->}(1,12)(14,12)
\taput{$\hh \in \SSS$}
\pcline{->}(0,11)(0,1)
\tlput{$\ff$}
\pcline(16,12)(23,12)
\psarc(23,5){7}{0}{90}
\pcline{->}(30,5)(30,1)
\put(31,5){$\gg'$}
\pcline[style=exist]{->}(15,11)(15,1)
\trput{$\ff'$}
\end{picture}
}}

Before proceeding, we establish two technical results about greedy paths. The 
first result is that the strength of greediness is not changed by the possible 
existence of nontrivial invertible elements.

\begin{lemm}
\label{L:GreedyInv}
For every subfamily~$\SSS$ of a left-cancellative category~$\CCC$, being 
$\SSS$-greedy and $\SSSs$-greedy are equivalent properties. 
\end{lemm}

\begin{proof}
By definition, a path that is $\SSSs$-greedy is $\XXX$-greedy for every~$\XXX$ 
included in~$\SSSs$. As $\SSS$ is included in~$\SSSs$, being $\SSSs$-greedy 
implies being $\SSS$-greedy.

Conversely, assume that $\seqq{\gg_1}{\gg_2}$ is $\SSS$-greedy, and we have 
$\hh \dive \ff \gg_1 \gg_2$ with $\hh \in \SSSs$. Two cases are possible. 
Assume first $\hh \in \SSS \CCCi$, say $\hh = \hh' \ee$ with $\hh' \in \SSS$ 
and $\ee \in \CCCi$. Then we have  $\hh' \ee \dive \ff \gg_1 \gg_2$,  hence 
$\hh' \dive \ff \gg_1$ since $\seqq{\gg_1}{\gg_2}$ is $\SSS$-greedy, that is, 
there exists~$\ff'$ satisfying $\hh' \ff' = \ff \gg_1$. We deduce $\hh \ee\inv 
\ff' = \ff \gg_1$, hence $\hh \dive \ff \gg_1$. Assume now $\hh \in \CCCi$. 
Then we can write $\hh \dive \hh \hh\inv$, whence $\hh \dive \ff 
\gg_1$ again. In every case, $\hh \dive \ff \gg_1$ holds, and 
$\seqq{\gg_1}{\gg_2}$ is $\SSSs$-greedy.
\end{proof}

The second result involves the connection between the current notion of 
greediness and the one in the literature. If $\seqq{\gg_1}{\gg_2}$ is 
$\SSS$-greedy, then, in particular, every element of~$\SSS$ that 
left-divides~$\gg_1\gg_2$ left-divides~$\gg_1$, which is the notion of a 
greediness considered for instance in~\cite{Eps} or~\cite{Dgk}. The current 
definition is stronger as it involves the additional factor~$\ff$ 
in~\eqref{E:Greedy}. However, this implication becomes an equivalence, that 
is, one recovers the usual notion of greediness, when the reference 
family~$\SSS$ satisfies some conditions.

\noindent
\parbox{\textwidth}{%
\begin{defi}
\label{D:RCClosed}
\rightskip40mm
\VR(6,0)A subfamily~$\XXX$ of a left-cancellative category~$\CCC$ is called 
\emph{closed under right-complement} if
\begin{equation}
\label{E:RCClosed}
\forall\ff, \gg {\in} \XXX\ \forall\hh{\in} \CCC \ (\ff, \gg \dive \hh 
\Rightarrow \exists \ff'\!,\gg' {\in} \XXX\, (\ff\gg' = \gg\ff' \dive 
\hh)).\hspace{40mm}
\end{equation}
is satisfied.\VR(0,2)
\end{defi}
\hfill
\rlap
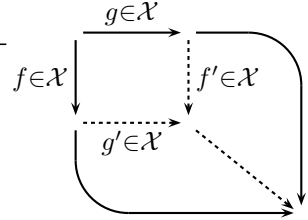
{%
\begin{picture}(0,0)(31,-3)
\pcline{->}(1,24)(14,24)
\taput{$\gg {\in}\XXX$}
\pcline{->}(0,23)(0,13)
\tlput{$\ff{\in}\XXX$}
\pcline(0,11)(0,7)
\psarc(7,7){7}{180}{270}
\pcline{->}(7,0)(29,0)
\pcline(16,24)(23,24)
\psarc(23,17){7}{0}{90}
\pcline{->}(30,17)(30,1)
\pcline[style=exist]{->}(1,12)(14,12)
\tbput{$\gg'{\in}\XXX$}
\pcline[style=exist]{->}(15,23)(15,13)
\trput{$\ff'{\in}\XXX$}
\pcline[style=exist]{->}(16,11)(29,0.5)
\end{picture}
}}

\begin{lemm}
\label{L:GreedyClosed}
Assume that $\CCC$ is a left-cancellative category, and $\SSS$ is a subfamily 
of~$\CCC$ such that $\SSSs$ generates~$\CCC$ and is closed under 
right-complement. 
Then $\seqq{\gg_1}{\gg_2}$ is $\SSS$-greedy if and only if one has
\begin{equation}
\label{E:Greedy2}
\forall\hh{\in}\SSS\ (\hh \dive \gg_1 \gg_2 \Rightarrow \hh \dive \gg_1).
\end{equation}
\end{lemm}

\begin{proof}
As already observed, by definition, \eqref{E:Greedy2} holds whenever 
$\seqq{\gg_1}{\gg_2}$ is $\SSS$-greedy. Conversely, assume that 
\eqref{E:Greedy2} holds, and we have $\hh \dive \ff \gg_1 \gg_2$. By 
assumption, $\SSSs$ generates~$\CCC$, so we can write $\ff = \ff_1 \pdots 
\ff_\pp$ with $\ff_1 \wdots \ff_\pp$ in~$\SSSs$. Then $\ff_1$ and $\hh$ belong 
to~$\SSSs$ and, by assumption, both left-divide $\ff_1 \pdots \ff_\pp \gg_1 
\gg_2$. The assumption that $\SSSs$ is closed under right-complement implies 
the existence of~$\ff'_1$ and~$\hh_1$ in~$\SSSs$ satisfying $\ff_1 \hh_1 = \hh 
\ff'_1 \dive \ff_1 \pdots \ff_\pp \gg_1 \gg_2$, see 
Figure~\ref{F:GreedyClosed}. Left-cancelling~$\ff_1$, we deduce $\hh_1 \dive 
\ff_2 \pdots \ff_\pp \gg_1 \gg_2$ and, arguing similarly, we deduce the 
existence of~of~$\ff'_2$ and~$\hh_2$ in~$\SSSs$ satisfying $\ff_2 \hh_2 = 
\hh_1 \ff'_2 \dive \ff_2 \pdots \ff_\pp \gg_1 \gg_2$, and so on. After 
$\pp$~steps, we obtain $\ff'_\pp$ and $\hh_\pp$ in~$\SSSs$ satisfying $\ff_\pp 
\hh_\pp = \hh_{\pp-1} \ff'_\pp$ and $\hh_\pp \dive \gg_1 \gg_2$.
Repeating the proof of 
Lemma~\ref{L:GreedyInv}, we see that \eqref{E:Greedy2} implies 
$\forall\hh{\in}\SSSs\, (\hh \dive \gg_1 \gg_2 \Rightarrow \hh \dive \gg_1)$ 
and we deduce $\hh_\pp \dive \gg_1$, whence $\hh \dive \ff_1 \pdots \ff_\pp 
\gg_1$ using the commutativity of the diagram. Hence $\seqq{\gg_1}{\gg_2}$ is 
$\SSS$-greedy.
\end{proof}

\begin{figure}[htb]
\begin{picture}(75,32)(0,0)
\pcline{->}(1,24)(14,24)
\taput{$\ff_1$}
\pcline{->}(16,24)(29,24)
\taput{$\ff_2$}
\pcline[style=etc](31,24)(44,24)
\pcline{->}(46,24)(59,24)
\taput{$\ff_\pp$}
\psline(61,24)(68,24)
\psarc(68,17){7}{0}{90}
\psline{->}(75,17)(75,13)
\put(76,17){$\gg_1$}

\pcline[style=exist]{->}(1,12)(14,12)
\tbput{$\ff'_1$}
\pcline[style=exist]{->}(16,12)(29,12)
\tbput{$\ff'_2$}
\pcline[style=etc](31,12)(44,12)
\pcline[style=exist]{->}(46,12)(59,12)
\tbput{$\ff'_\pp$}
\pcline[style=exist]{->}(61,12)(74,12)

\psline(0,11)(0,7)
\psarc(7,7){7}{180}{270}
\psline{->}(7,0)(74,0)
\put(-2,0){$\fft$}
\psline[style=exist](15,11)(15,7)
\psarc[style=exist](22,7){7}{180}{270}
\psline[style=exist](30,11)(30,7)
\psarc[style=exist](37,7){7}{180}{270}
\psline[style=exist](45,11)(45,7)
\psarc[style=exist](52,7){7}{180}{270}
\psline[style=exist](60,11)(60,7)
\psarc[style=exist](67,7){7}{180}{270}
\pcline{->}(75,11)(75,1)
\trput{$\gg_2$}

\pcline{->}(0,23)(0,13)
\tlput{$\hh$}
\pcline[style=exist]{->}(15,23)(15,13)
\trput{$\hh_1$}
\pcline[style=exist]{->}(30,23)(30,13)
\trput{$\hh_2$}
\pcline[style=exist]{->}(45,23)(45,13)
\trput{$\hh_{\pp-1}$}
\pcline[style=exist]{->}(60,23)(60,13)
\trput{$\hh_\pp$}

\psline[style=thin](0,27)(0,30)(60,30)(60,27)
\put(29,32){$\ff$}
\end{picture}
\caption[]{\sf\smaller Proof of Lemma~\ref{L:GreedyClosed}: whenever $\SSSs$ 
generates the ambient category and is closed under right-complement, each 
relation $\hh \dive \ff \gg_1 \gg_2$ leads to a relation of the form $\hh_\pp 
\dive \gg_1 \gg_2$ and, therefore, the restricted form of greediness implies 
the general form.}
\label{F:GreedyClosed}
\end{figure}
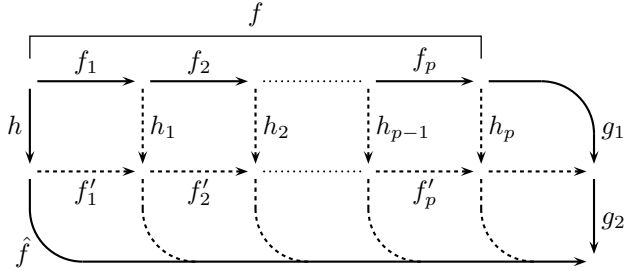

%%%%
\subsection{The notion of an $\SSS$-normal path}
\label{SS:Normal}

Building on the notion of an $\SSS$-greedy path, we can now introduce the 
distinguished decompositions we shall be interested in.

\begin{defi}
Assume that $\SSS$ is a subfamily of a left-cancellative category~$\CCC$. A 
$\CCC$-path~$\seqqq{\gg_1}\etc{\gg_\qq}$ is called \emph{$\SSS$-normal} if it 
is $\SSS$-greedy and, in addition, the entries $\gg_1 \wdots \gg_\qq$ lie 
in~$\SSSs$.
\end{defi}

Like greediness, $\SSS$-normality is a local property: a path 
$\seqqq{\gg_1}\etc{\gg_\qq}$ is $\SSS$-normal if and only if every length two 
subpath $\seqq{\gg_\ii}{\gg_{\ii+1}}$ is $\SSS$-normal. In diagrams, it will 
be convenient to indicate that a path $\seqq{\gg_1}{\gg_2}$ is $\SSS$-normal 
by drawing a small arc connecting the ends of the arrows, as in
\begin{picture}(30,5)(0,0)
\pcline{->}(1,0.5)(14,0.5)
\taput{$\gg_1$}
\pcline{->}(16,0.5)(29,0.5)
\taput{$\gg_2$}
\psarc[style=thin](14.5,0.5){3}{0}{180}
\end{picture}.
We shall naturally say that a path $\seqqq{\gg_1}\etc{\gg_\qq}$ is an 
$\SSS$-normal decomposition for an element~$\gg$ of the ambient category if 
$\seqqq{\gg_1}\etc{\gg_\qq}$ is a decomposition of~$\gg$ that is $\SSS$-normal.

For further reference, we first observe that, in an $\SSS$-normal path, 
invertible elements can always be added at the end and that, conversely, an 
invertible entry cannot be followed by a non-invertible one.

\begin{lemm}
\label{L:InvFirst}
Assume that $\SSS$ is a subfamily of a left-cancellative category~$\CCC$. 

\ITEM1 If $\gg_2$ is invertible, then every path of the form 
$\seqq{\gg_1}{\gg_2}$ is $\SSS$-greedy.

\ITEM2 If $\gg_1$ is invertible and $\gg_2$ lies in~$\SSSs$, then a path of 
the form $\seqq{\gg_1}{\gg_2}$ is $\SSS$-greedy (if and) only if  $\gg_2$  is 
invertible.
\end{lemm}

\begin{proof}
\ITEM1 Assume that $\gg_2$ is invertible. Then, for every~$\hh$ in~$\SSS$, the 
relation $\hh \dive \ff \gg_1 \gg_2$ implies $\hh \dive \ff \gg_1 \gg_2 
\gg_2\inv$, whence $\hh \dive \ff \gg_1$, and $\seqq{\gg_1}{\gg_2}$ is 
$\SSS$-greedy.

\ITEM2 Assume that $\gg_1$ is invertible, $\gg_2$ lies in~$\SSSs$, and 
$\seqq{\gg_1}{\gg_2}$ is $\SSS$-greedy. By Lemma~\ref{L:GreedyInv}, 
$\seqq{\gg_1}{\gg_2}$ is $\SSSs$-greedy and we have $\gg_2 = \gg_1\inv \gg_1 
\gg_2$ with $\gg_2 \in \SSSs$, whence $\gg_2 \dive \gg_1\inv \gg_1$. The 
latter relation implies that $\gg_2$ is invertible.
\end{proof}

It follows that, in an $\SSS$-normal path, the non-invertible entries always 
occur first, and the invertible entries always occur at the end.

We now address the uniqueness of $\SSS$-normal decompositions. The good news 
is that (some form of) uniqueness comes for free.

\begin{defi}
\label{D:Deformation}
(See Figure~\ref{F:Deformation}.) Assume that $\CCC$ is a left-cancellative 
category. A $\CCC$-path $\seqqq{\ff_1}\etc{\ff_\pp}$ is said to be a 
\emph{deformation by invertible elements}, or \emph{$\CCCi$-deformation}, of 
another path $\seqqq{\gg_1}\etc{\gg_\qq}$ if there exist $\ee_0 \wdots  
\ee_\rr$ in~$\CCCi$, $\rr = \max(\pp, \qq)$, such that $\ee_0$ and $\ee_\rr$  
are identity-elements and $\ee_{\ii-1} \gg_\ii = \ff_\ii \ee_\ii$ holds for $1 
 \le \ii \le \rr$, where, for $\pp \not=\qq$, the shorter path is expanded by 
identity-elements.
\end{defi}

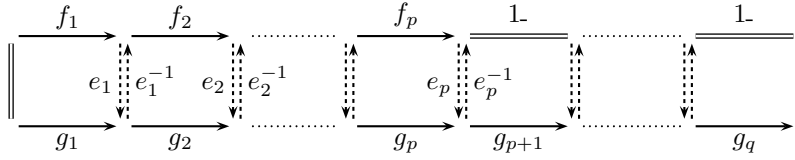
\begin{figure}[htb]
\begin{picture}(105,15)(0,-1)
\pcline{->}(1,0)(14,0)
\tbput{$\gg_1$}
\pcline{->}(1,12)(14,12)
\taput{$\ff_1$}
\pcline{->}(16,0)(29,0)
\tbput{$\gg_2$}
\pcline{->}(16,12)(29,12)
\taput{$\ff_2$}
\pcline[style=etc](32,0)(43,0)
\pcline[style=etc](32,12)(43,12)
\pcline{->}(46,0)(59,0)
\tbput{$\gg_\pp$}
\pcline{->}(46,12)(59,12)
\taput{$\ff_\pp$}
\pcline{->}(61,0)(74,0)
\tbput{$\gg_{\pp+1}$}
\pcline[style=double](61,12)(74,12)
\taput{$\id\ud$}
\pcline[style=etc](76,0)(89,0)
\pcline[style=etc](76,12)(89,12)
\pcline{->}(91,0)(104,0)
\tbput{$\gg_\qq$}
\pcline[style=double](91,12)(104,12)
\taput{$\id\ud$}
\psline[style=double,style=thin](0,1)(0,11)
\pcline[style=exist]{<-}(14.5,1)(14.5,11)
\tlput{$\ee_1$}
\pcline[style=exist]{->}(15.5,1)(15.5,11)
\trput{$\ee_1\inv$}
\pcline[style=exist]{<-}(29.5,1)(29.5,11)
\tlput{$\ee_2$}
\pcline[style=exist]{->}(30.5,1)(30.5,11)
\trput{$\ee_2\inv$}
\pcline[style=exist]{<-}(44.5,1)(44.5,11)
\pcline[style=exist]{->}(45.5,1)(45.5,11)
\pcline[style=exist]{<-}(59.5,1)(59.5,11)
\tlput{$\ee_\pp$}
\pcline[style=exist]{->}(60.5,1)(60.5,11)
\trput{$\ee_\pp\inv$}
\pcline[style=exist]{<-}(74.5,1)(74.5,11)
\pcline[style=exist]{->}(75.5,1)(75.5,11)
\pcline[style=exist]{<-}(89.5,1)(89.5,11)
\pcline[style=exist]{->}(90.5,1)(90.5,11)
\psline[style=double](105,1)(105,11)
\end{picture}
\caption[]{\sf\smaller Deformation by invertible elements: invertible elements 
connect the corresponding entries; if one path is shorter (here we are in the 
case $\pp < \qq$), it is extended by identity-elements.}
\label{F:Deformation}
\end{figure}

If the ambient category~$\CCC$ contains no nontrivial invertible element, then 
being a $\CCCi$-deformation just means coinciding up to adding final 
identity-elements. The uniqueness result for $\SSS$-normal decompositions is 
as follows:

\begin{prop}
\label{P:Unique}
Assume that $\SSS$ is a subfamily of a left-cancellative category~$\CCC$. Then 
any two $\SSS$-normal decompositions of an element of~$\CCC$ (if any) are 
$\CCCi$-deformat\-ions of one another.
\end{prop}

The proof appeals to an auxiliary result.

\begin{lemm}
\label{L:Grouping}
Assume that $\SSS$ is a subfamily of a left-cancellative category~$\CCC$. If 
$\seqqq{\gg_1}\etc{\gg_\qq}$ is an $\SSS$-greedy $\CCC$-path, then 
$\seqq{\gg_1}{\gg_2\pdots\gg_\qq}$ is $\SSS$-greedy as well.
\end{lemm}

\begin{proof}
Assume $\hh \dive \ff \gg_1 (\gg_2 \pdots \gg_\qq)$, where $\hh$ lies 
in~$\SSS$. Using associativity, we write $\hh \dive (\ff \gg_1 \pdots 
\gg_{\qq-2}) \gg_{\qq-1} \gg_\qq$, and the assumption that 
$\seqq{\gg_{\qq-1}}{\gg_\qq}$ is $\SSS$-greedy implies $\hh \dive (\ff \gg_1 
\pdots \gg_{\qq-2}) \gg_{\qq-1}$, which is also $\hh \dive (\ff \gg_1 \pdots 
\gg_{\qq-3}) \gg_{\qq-2} \gg_{\qq-1}$. The assumption that 
$\seqq{\gg_{\qq-2}}{\gg_{\qq-1}}$ is $\SSS$-greedy implies now $\hh \dive (\ff 
\gg_1 \pdots \gg_{\qq-3}) \gg_{\qq-2}$, and so on. Finally, we obtain $\hh 
\dive (\ff) \gg_1 \gg_2$, and the assumption that $\seqq{\gg_1}{\gg_2}$ is 
$\SSS$-greedy implies $\hh \dive \ff \gg_1$. So 
$\seqq{\gg_1}{\gg_2\pdots\gg_\qq}$ is $\SSS$-greedy.
\end{proof}

\begin{proof}[Proof of Proposition~\ref{P:Unique}]
Assume that $\seqqq{\ff_1}\etc{\ff_\pp}$ and $\seqqq{\gg_1}\etc{\gg_\qq}$ are 
$\SSS$-normal decompositions of an element~$\gg$ of~$\CCC$. Let $\yy$ be the 
target of~$\gg$. At the expense of adding factors~$\id\yy$ at the end of the 
shorter path, we may assume $\pp = \qq$: by Lemma~\ref{L:InvFirst}, 
adding identity-elements at the end of an $\SSS$-greedy path yields an 
$\SSS$-greedy path.

Let $\ee_0 = \id\xx$, where $\xx$ is the source of~$\gg$. By 
Lemma~\ref{L:Grouping}, the path $\seqq{\gg_1}{\gg_2 \pdots \gg_\qq}$ is 
$\SSS$-greedy, hence, by Lemma~\ref{L:GreedyInv}, $\SSSs$-greedy. Now $\ff_1$ 
belongs to~$\SSSs$ and left-divides $\ff_1 \pdots \ff_\pp$, which is $\ee_0 
\gg_1 \pdots \gg_\qq$, so it must left-divide~$\gg_1$. In other words, there 
exists~$\ee_1$ satisfying $\ff_1 \ee_1 = \ee_0 \gg_1$. Left-cancelling~$\ff_1$ 
in $\ff_1 \pdots \ff_\pp = (\ff_1\ee_1) \gg_2 \pdots \gg_\pp$, we deduce 
$\ff_2 \pdots \ff_\pp = \ee_1 \gg_1 \pdots \gg_\qq$. Now $\ff_2$ belongs 
to~$\SSSs$ and, by Lemma~\ref{L:Grouping} again, $\seqq{\gg_2}{\gg_3 \pdots 
\gg_\qq}$ is $\SSS$-greedy, so we deduce the existence of~$\ee_2$ satisfying 
$\ff_2 \ee_2 = \ee_1 \gg_1$, and so on, giving the existence of~$\ee_\ii$ 
satisfying $\ff_\ii \ee_\ii = \ee_{\ii-1} \gg_\ii$ for every $1 \le \ii \le 
\qq$.

Exchanging the roles of $\seqqq{\ff_1}\etc{\ff_\pp}$ and 
$\seqqq{\gg_1}\etc{\gg_\qq}$ and arguing symmetrically from $\ee'_0 = \id\xx$, 
we obtain the existence of~$\ee'_\ii$ satisfying $\gg_\ii \ee'_\ii = 
\ee'_{\ii-1} \ff_\ii$ for every $1 \le \ii \le \pp$. We deduce, for 
every~$\ii$, the equalities $(\ff_1 \pdots \ff_\ii) \ee_\ii = \gg_1 \pdots 
\gg_\ii$ and $(\gg_1 \pdots \gg_\ii) \ee'_\ii = \ff_1 \pdots \ff_\ii$, which 
imply that $\ee'_\ii$ is the inverse of~$\ee_\ii$. Hence 
$\seqqq{\ff_1}\etc{\ff_\pp}$ is a $\CCCi$-deformation of 
$\seqqq{\gg_1}\etc{\gg_\qq}$.
\end{proof}

It can be checked that, conversely, every $\CCCi$-deformation of an 
$\SSS$-normal decomposition of an element~$\gg$ is again an $\SSS$-normal 
decomposition of~$\gg$ provided the relation $\CCCi \SSS \subseteq \SSSs$ is 
satisfied, but we shall not need that result here.

A direct application of Proposition~\ref{P:Unique} is that the number of 
non-invertible entries in an $\SSS$-normal decomposition of an element~$\gg$ 
does not depend on the choice of that decomposition: indeed, if $\gg$ and 
$\gg'$ are connected by an equality $\gg' \ee = \ee' \gg$ with $\ee, \ee'$ 
invertible, then $\gg$ is invertible if and only if $\gg'$ is invertible. So 
the following notion is non-ambiguous.

\begin{defi}
Assume that $\SSS$ is a subfamily of a left-cancellative category~$\CCC$. The 
\emph{$\SSS$-length} of an element~$\gg$ of~$\SSS$, written $\LGG\SSS\gg$, is 
the number of non-invertible entries in an $\SSS$-normal decomposition 
of~$\gg$ (if any).
\end{defi}

Of course, the number~$\LGG\SSS\gg$ is defined only if $\gg$ admits at least 
one $\SSS$-normal decomposition.

\begin{lemm}
\label{L:Length}
Assume that $\SSS$ is a subfamily of a left-cancellative category~$\CCC$. Then 
every element~$\gg$  of~$\CCC$  that admits an $\SSS$-normal decomposition 
admits an $\SSS$-normal decomposition of length~$\LGG\SSS\gg$.
\end{lemm}

\begin{proof}
Assume that $\seqqq{\gg_1}\etc{\gg_\qq}$ is an $\SSS$-normal decomposition 
of~$\gg$. Let $\rr = \LGG\SSS\gg$. By definition, $\rr$ is the number of 
non-invertible entries in $\seqqq{\gg_1}\etc{\gg_\qq}$, so $\rr \le \qq$ 
holds. Assume $\rr < \qq$. By Lemma~\ref{L:InvFirst}, the entries $\gg_1 
\wdots \gg_\rr$ are non-invertible, whereas $\gg_{\rr+1} \wdots \gg_\qq$ are 
invertible. But, then, $\gg_\rr \pdots \gg_\qq$ is an element of~$\SSSs$, and 
therefore $\seqqqq{\gg_1}\etc{\gg_{\rr-1}}{\gg_\rr \pdots \gg_\qq}$ is an 
$\SSS$-normal decomposition of~$\gg$ that has length~$\rr$.
\end{proof}

We shall prove

\begin{prop}
\label{P:Length}
Assume that $\SSS$ is a subfamily of a left-cancellative category~$\CCC$. Then 
$\LGG\SSS\gg \le \rr$ holds for every element~$\gg$ in~$(\SSSs)^\rr$ that 
admits an $\SSS$-normal decomposition.
\end{prop}

We begin with a new auxiliary result about greedy paths. 

\begin{lemm}
\label{L:Grouping2}
Assume that $\SSS$ is a subfamily of a left-cancellative category~$\CCC$. If 
$\seqqq{\gg_1}\etc{\gg_\qq}$ is $\SSS$-greedy and $\rr < \qq$ holds, then 
$\seqq{\gg_1\pdots \gg_\rr}{\gg_{\rr+1}\pdots\gg_\qq}$ is $\SSS^\rr$-greedy.
\end{lemm}

\begin{proof}
Assume $\hh \dive \ff \gg_1 \pdots \gg_\qq$ with $\hh \in \SSS^\rr$, say 
$\hh = \hh_1 \pdots \hh_\rr$ with $\hh_1 \wdots \hh_\rr \in \SSS$. By  
Lemma~\ref{L:Grouping}, $\seqq{\gg_1}{\gg_2 \pdots \gg_\qq}$ is $\SSS$-greedy, 
 hence the  assumption  implies $\hh_1 \dive \ff \gg_1$, so there 
exists~$\ff'_1$  satisfying $\hh_1 \ff'_1 = \ff \gg_1$. Using 
left-cancellativity, we deduce $\hh_2 \pdots \hh_\rr \dive \ff'_1 \gg_2 \pdots 
\gg_\qq$. By Lemma~\ref{L:Grouping} again, $\seqq{\gg_2}{\gg_3 \pdots 
\gg_\qq}$ is $\SSS$-greedy, and we deduce $\hh_2 \dive \ff'_1 \gg_2$, so there 
exists~$\ff'_2$  satisfying $\hh_2 \ff'_2 = \ff'_1 \gg_2$. Repeating the 
argument, we find after $\rr$~steps $\ff'_\rr$ satisfying $\hh_\rr \ff'_\rr = 
\ff'_{\rr-1} \gg_\rr$, and we deduce $\hh_1 \pdots \hh_\rr \ff'_\rr = \ff 
\gg_1 \pdots \gg_\rr$, whence $\hh \dive \ff \gg_1 \pdots \gg_\rr$. So 
$\seqq{\gg_1 \pdots \gg_\rr}{\gg_{\rr+1} \pdots\gg_\qq}$ is $\SSS^\rr$-greedy. 
\end{proof}

\begin{proof}[Proof of Proposition~\ref{P:Length}]
Assume that $\seqqq{\gg_1}\etc{\gg_\qq}$ is an $\SSS$-normal decomposition for 
an element~$\gg$ that belongs to~$(\SSSs)^\rr$. If $\qq \le \rr$ holds, then 
$\LGG\SSS\gg \le \rr$ is trivial. Assume now $\qq > \rr$. By 
Lemma~\ref{L:GreedyInv}, the path $\seqqq{\gg_1}\etc{\gg_\qq}$ is 
$\SSSs$-greedy so, by Lemma~\ref{L:Grouping2} (applied with~$\SSSs$), the 
sequence $\seqq{\gg_1 \pdots \gg_\rr}{\gg_{\rr+1} \pdots \gg_\qq}$ is 
$(\SSSs)^\rr$-greedy. As $\gg$ belongs to~$(\SSSs)^\rr$ and it left-divides 
$\gg_1 \pdots \gg_\qq$, we deduce that $\gg$ left-divides $\gg_1 \pdots 
\gg_\rr$, say $\gg \ff =\nobreak \gg_1 \pdots \gg_\rr$. As we have $\gg = 
\gg_1 \pdots 
\gg_\qq$, by left-cancelling $\gg_1 \pdots \gg_\rr$, we deduce 
$\gg_{\rr+1} \pdots \gg_\qq \ff = \id\yy$, where $\yy$ is the target of~$\ff$. It 
follows that  $\gg_{\rr+1} \wdots \gg_\qq$  all must be invertible. So the 
non-invertible entries are among $\gg_1 \wdots \gg_\rr$, and $\LGG\SSS\gg \le 
\rr$ follows.
\end{proof}

%%%%
\subsection{Garside families}
\label{SS:Gar}

After uniqueness, we now address the existence of $\SSS$-normal 
decompositions. If every element of the ambient category admits an 
$\SSS$-normal decomposition, then, clearly, the family~$\SSSs$ must generate 
the category. However this necessary condition is not sufficient in general, 
and we shall introduce Garside families precisely as those families that 
guarantee the existence of normal decompositions.

\begin{defi}
\label{D:Gar}
A \emph{Garside family} in a left-cancellative category~$\CCC$ is a 
subfamily~$\SSS$ of~$\CCC$ such that every element of~$\CCC$ admits an 
$\SSS$-normal decomposition.
\end{defi}

Garside families exist in every category: indeed, if $\CCC$ is any 
left-cancellative category, then $\CCC$ is a Garside family in itself since, 
for every~$\gg$ in~$\CCC$, the length one path~$\seq\gg$ is a $\CCC$-normal 
decomposition of~$\gg$. On the other hand, it may happen that $\CCC$ is the 
only Garside family in~$\CCC$: for instance, this is the case for the monoid 
$\PRESp{\tta, \ttb}{\tta\ttb = \ttb\tta, \tta^2 = \ttb^2}$ as, 
anticipating on Definition~\ref{D:RMClosed} and Proposition~\ref{P:RecGarII}, 
one can see using induction on~$\rr \ge 0$ that every subset that is closed 
under right-comultiple and contains~$1, \tta$, and~$\ttb$ 
contains~$\tta^{\rr+1}$ and~$\tta^\rr\ttb$ for every~$\rr$, hence coincides 
with the whole monoid.  

The connection with the notion of a Garside monoid as developed in~\cite{Dfx, 
Dgk} is easily described:

\begin{prop}
Assume that $\MM$ is a Garside monoid with Garside element~$\Delta$. Then the 
family~$\Div(\Delta)$ of all divisors of~$\Delta$ is a Garside family in~$\MM$.
\end{prop}

\begin{proof}
Let $\gg$ be an element of $\MM \setminus \{1\}$. Let $\gg_1$ be the left-gcd 
of~$\gg$ and~$\Delta$. Then $\gg_1$ belongs to~$\Div(\Delta)$, and $\gg_1$ 
left-divides~$\gg$, say $\gg = \gg_1 \gg'$. If $\gg_1$ is not~$1$, we repeat 
the argument, finding a decomposition $\gg' = \gg_2 \gg''$ with 
$\gg_2$ the left-gcd of~$\gg'$ and~$\Delta$, and so on. The assumption that $\MM$ 
is 
atomic guarantees that the construction will stop after finitely many steps 
and one finds a decomposition $\gg = \gg_1 \pdots \gg_\qq$ for~$\gg$ in terms 
of divisors of~$\Delta$. 

So it remains to see that the sequence $\seqqq{\gg_1}\etc{\gg_\qq}$ is 
$\Div(\Delta)$-greedy. Now, we claim that the family~$\Div(\Delta)$ satisfies 
the conditions of Lemma~\ref{L:GreedyClosed}. Indeed, by definition of a 
Garside element, $\Div(\Delta)$ generates~$\MM$. On the other hand, assume 
$\ff, \gg \dive \Delta$ and $\ff, \gg \dive \hh$. Then $\ff$ and $\gg$ 
left-divide the left-gcd~$\hh_1$ of~$\hh$ and~$\Delta$, so we have $\ff \gg' = 
\gg \ff' \dive \hh_1$ for some~$\ff', \gg'$. By construction, $\hh_1$ belongs 
to~$\Div(\Delta)$, and so do $\ff'$ and $\gg'$, since all right-divisors 
of~$\Delta$ also are left-divisors. Hence $\Div(\Delta)$ is closed under 
right-complement. Now, assume $\hh \dive \gg_1 \pdots \gg_\qq$ with $\hh \in 
\Div(\Delta)$. Then $\hh$ left-divides the left-gcd of~$\gg_1 \pdots 
\gg_\qq$ 
and~$\Delta$, which is precisely~$\gg_1$. By Lemma~\ref{L:GreedyClosed}, we 
deduce that $\seqq{\gg_1}{\gg_2 \pdots \gg_\qq}$ is $\Div(\Delta)$-greedy and, 
\emph{a fortiori}, so is $\seqq{\gg_1}{\gg_2}$. Applying the same argument 
to~$\gg', \gg''$... inductively shows that $\seqq{\gg_\ii}{\gg_{\ii+1}}$ is 
$\Div(\Delta)$-greedy as well. 
So the sequence $\seqqq{\gg_1}\etc{\gg_\qq}$ is $\Div(\Delta)$-normal, every 
element of~$\MM$ admits a $\Div(\Delta)$-normal decomposition, and 
$\Div(\Delta)$ is a Garside family in~$\MM$.
\end{proof}

%%%%
\section{Recognizing Garside families}
\label{S:RecGar}

Definition~\ref{D:Gar} gives no practical criterion for recognizing Garside 
families. Our aim for now on will be to establish various characterizations of 
such families, which amounts to establishing various necessary and sufficient 
conditions guaranteeing the existence of normal decompositions. 

%%%%%%%%
\subsection{Recognizing Garside families I: incremental approach}
\label{SS:Increm}

The first characterization relies on the possibility of constructing 
$\SSS$-normal decompositions using an induction.

\begin{prop}
\label{P:RecGarI}
A subfamily~$\SSS$ of a left-cancellative category~$\CCC$ is a Garside family 
if and only if 
\begin{equation}
\label{E:RecGarI}
\BOX{\VR(3,1)$\SSSs$ generates~$\CCC$ and every element of $(\SSSs)^2$ admits 
an $\SSS$-normal decomposition.}
\end{equation} 
\end{prop}

As already noted, if $\SSS$ is a Garside family in~$\CCC$, then, by 
definition, every element of~$\CCC$ admits a decomposition in which every 
entry lies in~$\SSSs$, so $\SSSs$ must generate~$\CCC$, and every element 
admits an $\SSS$-normal decomposition hence, in particular, so does every 
element of~$(\SSSs)^2$. So every Garside family necessarily 
satisfies~\eqref{E:RecGarI}, and the point is to prove that, conversely, every 
family satisfying~\eqref{E:RecGarI} is a Garside family. We begin with 
preparatory results.

\begin{lemm}[``domino rule'']
\label{L:Domino}
\rightskip40mm
\begin{picture}(0,0)(-47,12)
\psarc[style=thin](15,0){3}{180}{360}
\psarc[style=thin](15,12){3.5}{180}{270}
\psarc[style=thinexist](15,12){3}{0}{180}
\pcline{->}(1,0)(14,0)
\tbput{$\gg_1$}
\pcline{->}(16,0)(29,0)
\tbput{$\gg_2$}
\pcline{->}(1,12)(14,12)
\taput{$\gg'_1$}
\pcline{->}(16,12)(29,12)
\taput{$\gg'_2$}
\pcline{->}(0,11)(0,1)
\trput{$\ff_0$}
\pcline{->}(15,11)(15,1)
\trput{$\ff_1$}
\pcline{->}(30,11)(30,1)
\tlput{$\ff_2$}
\end{picture}
Assume that $\SSS$ is a subfamily of a left-cancellative category~$\CCC$, and 
we have a commutative diagram with edges in~$\CCC$ as on the right. If 
$\seqq{\gg_1}{\gg_2}$ and $\seqq{\gg'_1}{\ff_1}$ are $\SSS$-greedy, then 
$\seqq{\gg'_1}{\gg'_2}$ is $\SSS$-greedy as well.
\end{lemm}

\begin{proof}
Assume $\hh \in \SSS$ and $\hh \dive \ff \gg'_1 \gg'_2$. As the diagram is 
commutative, we have $\hh \dive \ff \ff_0 \gg_1 \gg_2$. As 
$\seqq{\gg_1}{\gg_2}$ is $\SSS$-greedy , we deduce $\hh \dive \ff \ff_0 
\gg_1$, hence $\hh \dive \ff \gg'_1 \ff_1$. Now, as $\seqq{\gg'_1}{\ff_1}$ is 
$\SSS$-greedy, we deduce $\hh \dive \ff \gg'_1$. Therefore 
$\seqq{\gg'_1}{\gg'_2}$ is $\SSS$-greedy.
\end{proof}

\begin{lemm}
\label{L:LeftMult}
Assume that $\SSS$ is a subfamily of a left-cancellative category~$\CCC$ that 
satisfies~\eqref{E:RecGarI} and $\gg$ is an element of~$\CCC$ that admits an 
$\SSS$-normal decomposition of length~$\qq$. Then, for every~$\ff$ in~$\SSSs$, 
the element~$\ff \gg$ admits an $\SSS$-normal decomposition of length $\qq+1$ 
whenever it is defined. Moreover, we have $\LGG\SSS\gg \le \LGG\SSS{\ff\gg} 
\le \LGG\SSS\gg + 1$.
\end{lemm}

\begin{proof}
(See Figure~\ref{F:LeftMult}.) Assume that $\seqqq{\gg_1}\etc{\gg_\qq}$ is an 
$\SSS$-normal decomposition of~$\gg$. Put $\ff_0 = \ff$. By~\eqref{E:RecGarI} 
and Proposition~\ref{P:Length}, the element~$\ff_0 \gg_1$ of~$(\SSSs)^2$ 
admits an $\SSS$-normal decomposition of length two, say 
$\seqq{\gg'_1}{\ff_1}$. Then, similarly, the element~$\ff_1 \gg_2$ 
of~$(\SSSs)^2$ admits an $\SSS$-normal decomposition of length two, say 
$\seqq{\gg'_2}{\ff_2}$, and so on. Finally, $\ff_{\qq-1} \gg_\qq$ admits an 
$\SSS$-normal decomposition of length two, say $\seqq{\gg'_\qq}{\ff_\qq}$. By 
construction, $\seqqqq{\gg'_1}\etc{\gg'_\qq}{\ff_\qq}$ is a decomposition 
of~$\ff \gg$, and its entries lie in~$\SSSs$. Moreover, for $1 \le \ii < \qq$, 
the paths $\seqq{\gg_\ii}{\gg_{\ii+1}}$ and $\seqq{\gg'_\ii}{\ff_\ii}$ are 
$\SSS$-greedy, so the domino rule (Lemma~\ref{L:Domino}) implies that 
$\seqq{\gg'_\ii}{\gg'_{\ii+1}}$ is $\SSS$-greedy as well. Thus 
$\seqqqq{\gg'_1}\etc{\gg'_\qq}{\ff_\qq}$ is $\SSS$-greedy, hence $\SSS$-normal.

As for the $\SSS$-length, we can assume $\qq = \LGG\SSS\gg$, which is always 
possible by Lemma~\ref{L:Length}. Then the above argument provides an 
$\SSS$-normal decomposition of length~$\qq + 1$ for~$\ff\gg$, implying 
$\LGG\SSS{\ff\gg} \le \LGG\SSS\gg + 1$. If $\qq = 0$ holds, that is, if $\gg$ 
is invertible, then $\LGG\SSS\gg \le \LGG\SSS{\ff\gg}$ is trivial. Otherwise, 
we have $\qq \ge 1$ and, by assumption, the entries~$\gg_1 \wdots \gg_\qq$ are 
not invertible. This implies that $\gg'_1 \wdots \gg'_\qq$ are not invertible 
either: indeed, by Lemma~\ref{L:InvFirst}, the assumption that 
$\seqq{\gg'_\ii}{\ff_\ii}$ is $\SSS$-normal implies that $\ff_\ii$ is 
invertible whenever $\gg'_\ii$ is, so that $\ff_{\ii-1}\gg_\ii$, hence 
$\gg_\ii$, must be invertible, contrary to the assumption. Hence $\LGG\SSS\gg 
\le \LGG\SSS{\ff\gg}$ is always satisfied.
\end{proof}

\begin{figure}[htb]
\begin{picture}(45,18)(0,-2)
\pcline{->}(-14,0)(-1,0)
\tbput{$\ff$}
\pcline{->}(1,0)(14,0)
\tbput{$\gg_1$}
\pcline{->}(16,0)(29,0)
\tbput{$\gg_2$}
\psline[style=etc](33,0)(41,0)
\pcline{->}(46,0)(59,0)
\tbput{$\gg_\qq$}

\pcline{->}(1,12)(14,12)
\taput{$\gg'_1$}
\pcline{->}(16,12)(29,12)
\taput{$\gg'_2$}
\psline[style=etc](33,12)(41,12)
\pcline{->}(46,12)(59,12)
\taput{$\gg'_\qq$}

\pcline{->}(0,11)(0,1)
\trput{$\ff_0$}
\pcline{->}(15,11)(15,1)
\trput{$\ff_1$}
\pcline{->}(30,11)(30,1)
\trput{$\ff_2$}
\pcline{->}(45,11)(45,1)
\trput{$\ff_{\qq-1}$}
\pcline{->}(60,11)(60,1)
\trput{$\ff_\qq$}

\psline[style=double](-15,1)(-15,5)
\psarc[style=double](-8,5){7}{90}{180}
\psline[style=double](-8,12)(-1,12)

\psarc[style=thin](14.5,0){3}{180}{360}
\psarc[style=thin](29.5,0){3}{180}{360}
\psarc[style=thin](44.5,0){3}{180}{360}
\psarc[style=thinexist](14.5,12){3}{0}{180}
\psarc[style=thinexist](29.5,12){3}{0}{180}
\psarc[style=thinexist](44.5,12){3}{0}{180}
\psarc[style=thin](15,12){3}{180}{270}
\psarc[style=thin](30,12){3}{180}{270}
\psarc[style=thin](60,12){3}{180}{270}
\end{picture}
\caption[]{\sf\smaller Proof of Lemma~\ref{L:LeftMult}: starting from~$\ff$ 
in~$\SSSs$ and an $\SSS$-normal decomposition of~$\gg$, we inductively build 
an $\SSS$-normal decomposition of~$\ff\gg$.}
\label{F:LeftMult}
\end{figure}

We can now complete the argument.

\begin{proof}[Proof of Proposition~\ref{P:RecGarI}]
We already noted that, if $\SSS$ is a Garside family in~$\CCC$, then 
\eqref{E:RecGarI} is satisfied. 

Conversely, assume that $\SSS$ satisfies~\eqref{E:RecGarI}. We prove using 
induction on~$\rr$ that every element~$\gg$ of~$(\SSSs)^\rr$ admits an 
$\SSS$-normal decomposition. For $\rr = 0$, that is, if $\gg$ has the 
form~$\id\xx$, the empty path~$\ew_\xx$ is an $\SSS$-normal decomposition 
of~$\gg$ and, for $\rr = 1$, the sequence~$(\gg)$ is an $\SSS$-normal 
decomposition of~$\gg$. For $\rr \ge 2$, we write $\gg = \ff \gg'$ with $\ff$ 
in~$\SSSs$ and $\gg'$ in~$(\SSSs)^{\rr-1}$, and we apply 
Lemma~\ref{L:LeftMult}. So every element of~$\CCC$ that belongs 
to~$(\SSSs)^\rr$ admits an $\SSS$-normal decomposition. As, by assumption, 
$\SSSs$ generates~$\CCC$, every element of~$\CCC$ is eligible. Hence $\SSS$ is 
a Garside family in~$\CCC$.
\end{proof}

%%%%
\subsection{Recognizing Garside families II: closure properties}
\label{SS:RecGarII}

We turn to further characterizations of Garside families involving closure 
properties, such as closure under right-complement 
(Definition~\ref{D:RCClosed}). Here we introduce another similar notion.

\noindent
\parbox{\textwidth}{%
\begin{defi}
\label{D:RMClosed}
\rightskip40mm
\VR(6,0)A subfamily~$\XXX$ of a left-cancellative category~$\CCC$ is called 
\emph{closed under right-comultiple} if
\begin{equation}
\label{E:RMClosed}
\forall\ff, \gg {\in} \XXX\ \forall\hh{\in} \CCC \ (\ff, \gg \dive \hh 
\Rightarrow \exists \hh' {\in} \XXX\, (\ff, \gg \dive \hh' \dive 
\hh)).\hspace{40mm}
\end{equation}
is satisfied.\VR(0,4)
\end{defi}
\hfill
\rlap
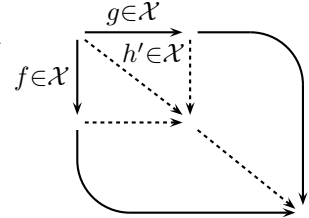
{%
\begin{picture}(0,0)(30,-3)
\pcline{->}(1,24)(14,24)
\taput{$\gg {\in}\XXX$}
\pcline{->}(0,23)(0,13)
\tlput{$\ff{\in}\XXX$}
\pcline(0,11)(0,7)
\psarc(7,7){7}{180}{270}
\pcline{->}(7,0)(29,0)
\pcline(16,24)(23,24)
\psarc(23,17){7}{0}{90}
\pcline{->}(30,17)(30,1)
\pcline[style=exist]{->}(1,12)(14,12)
\pcline[style=exist]{->}(15,23)(15,13)
\pcline[style=exist]{->}(16,11)(29,0.5)
\pcline[style=exist]{->}(1,23)(14,13)
\put(6,20){$\hh'{\in}\XXX$}
\end{picture}
}}

So, in words, a family~$\XXX$ is closed under right-comultiple if and only if 
every common right-multiple of two elements of~$\XXX$ must be a right-multiple 
of some common right-multiple that lies in~$\XXX$. On the other hand, we shall 
naturally say that a family~$\XXX$ is closed \emph{under right-divisor} if 
every right-divisor of an element of~$\XXX$ belongs to~$\XXX$. For further 
reference, we immediately note an easy connection between closure properties.

\begin{lemm}
\label{L:RMRC}
A subfamily~$\XXX$ of a left-cancellative category that is closed under 
right-comultiple and right-divisor is closed under right-complement.
\end{lemm}

\begin{proof}
Assume $\ff, \gg \dive \hh$ with $\ff, \gg \in \XXX$. As $\XXX$ is closed 
under right-comultiple, there exists~$\hh'$ in~$\XXX$ satisfying $\ff, \gg 
\dive \hh' \dive \hh$. By definition, this means that there exist~$\ff', \gg'$ 
satisfying $\ff \gg' = \gg \ff' = \hh'$. Now, as $\XXX$ is closed under 
right-divisor, the assumption that $\hh$ lies in~$\XXX$ implies that $\ff'$ 
and $\gg'$ also do. But this means that $\XXX$ is closed under 
right-complement.
\end{proof}

We shall need one more notion.

\begin{defi}
Assume that $\SSS$ is a subfamily of a left-cancellative category~$\CCC$. 
For~$\gg$ in~$\CCC$, we say that $\gg_1$ is an \emph{$\SSS$-head} of~$\gg$ if 
$\gg_1$ belongs to~$\SSS$, it left-divides~$\gg$, and every element of~$\SSS$ 
that left-divides~$\gg$ left-divides~$\gg_1$.
\end{defi}

Here are the expected characterizations of Garside families.

\begin{prop}
\label{P:RecGarII}
A subfamily~$\SSS$ of a left-cancellative category~$\CCC$ is a Garside family 
if and only if it satisfies one of the following equivalent conditions:
\begin{gather}
\label{E:RecGarII1}
\BOX{$\SSSs$ generates~$\CCC$, it is closed under right-divisor, and every 
non-invertible element of $\CCC$ admits an $\SSS$-head;}\\
\label{E:RecGarII2}
\BOX{$\SSSs$ generates~$\CCC$, it is closed under right-complement, and every 
non-invertible element of $(\SSSs)^2$ admits an $\SSS$-head;}\\
\label{E:RecGarII3}
\BOX{$\SSSs$ generates~$\CCC$, it is closed under right-comultiple and 
right-divisor, and every non-invertible element of $(\SSSs)^2$ admits a 
$\div$-maximal left-divisor in~$\SSS$.}
\end{gather} 
\end{prop}

The difference between the final conditions in~\eqref{E:RecGarII2} 
and~\eqref{E:RecGarII3} is that, in~\eqref{E:RecGarII3}, we do not demand that 
every element of~$\SSS$ that left-divides the considered element~$\gg$ 
left-divides the maximal left-divisor, but only that no proper multiple of the 
latter left-divides~$\gg$, a weaker condition.

Contrary to Proposition~\ref{P:RecGarI}, it is not obvious that the conditions 
of Proposition~\ref{P:RecGarII} necessarily hold for every Garside family, so 
a proof is needed in both directions. As usual we shall split the argument 
into several steps. The first one directly follows from our definitions.

\begin{lemm}
\label{L:Head}
Assume that $\SSS$ is a subfamily of a left-cancellative category~$\CCC$ and 
$\gg_1$ lies in~$\SSS$.

\ITEM1 If $\seqq{\gg_1}{\gg_2}$ is $\SSS$-greedy, then $\gg_1$ is an 
$\SSS$-head for~$\gg_1 \gg_2$.

\ITEM2 Conversely, if $\SSSs$ generates~$\CCC$ and is closed under 
right-complement, then $\gg_1$ being an $\SSS$-head of~$\gg_1 \gg_2$ implies 
that $\seqq{\gg_1}{\gg_2}$ is $\SSS$-greedy.
\end{lemm}

\begin{proof}
\ITEM1 We have $\gg_1 \dive \gg_1 \gg_2$. Assume $\hh \in 
\SSS$ and $\hh \dive \gg_1 \gg_2$. As $\seqq{\gg_1}{\gg_2}$ is $\SSS$-greedy, 
we deduce $\hh \dive \gg_1$, which, by definition, means that $\gg_1$ is an 
$\SSS$-head for~$\gg_1 \gg_2$.

\ITEM2 Assume that $\gg_1$ is an $\SSS$-head of~$\gg_1 \gg_2$. By definition, 
every element of~$\SSS$ that left-divides~$\gg_1 \gg_2$ left-divides~$\gg_1$. 
Now, by Lemma~\ref{L:GreedyClosed}, this implies that $\seqq{\gg_1}{\gg_2}$ is 
$\SSS$-greedy whenever $\SSSs$ generates~$\CCC$ and is closed under 
right-complement.
\end{proof}

\begin{lemm}
\label{L:Closed}
Assume that $\SSS$ is a Garside family in a left-cancellative category~$\CCC$. 
Then $\SSSs$ is closed under right-divisor, right-comultiple, and 
right-complement.  Moreover, we have $\CCCi \SSSs \subseteq \SSSs$.
\end{lemm}

\begin{proof}
Assume that $\gg$ lies in~$\SSSs$ and $\ff$ right-divides~$\gg$.  
Then we have $\LGG\SSS\gg \le 
1$, and Lemma~\ref{L:LeftMult} inductively implies $\LGG\SSS\ff \le 
\LGG\SSS\gg$, whence $\LGG\SSS\ff \le 1$, which in turn implies $\ff \in 
\SSSs$. So $\SSSs$ is closed under right-divisor.

Next, assume that $\ff, \gg$ lie in~$\SSSs$ and $\hh$ is a common 
right-multiple of~$\ff$ and~$\gg$. Let $\seqqq{\hh_1}\etc{\hh_\rr}$ be an 
$\SSS$-normal decomposition of~$\hh$, which exists as $\SSS$ is a Garside 
family. Then $\ff$ is an element of~$\SSSs$ that left-divides~$\hh_1 \pdots 
\hh_\rr$ and $\seqqq{\hh_1}\etc{\hh_\rr}$, hence $\seqq{\hh_1}{\hh_2 \pdots 
\hh_\rr}$, are $\SSSs$-greedy. It follows that $\ff$ left-divides~$\hh_1$.  
Similarly  $\gg$ left-divides~$\hh_1$. Hence $\hh_1$ is a common 
right-multiple of~$\ff$ and~$\gg$ that left-divides~$\hh$ and lies in~$\SSSs$. 
Hence $\SSSs$ is closed under right-comultiple.

Then, Lemma~\ref{L:RMRC} implies that $\SSSs$ is closed under right-complement 
since it is closed under right-comultiple and right-divisor.

 Finally, assume $\gg \in \CCCi \SSSs$, say $\gg = \ee \hh$ with $\ee \in 
\CCCi$ and $\hh \in \SSSs$. Then we have $\hh = \ee\inv \gg$, so $\gg$ 
right-divides an element of~$\SSSs$, hence it belongs to~$\SSSs$ by the first 
result above.
\end{proof}

\begin{lemm}
\label{L:Strict}
Assume that $\SSS$ is a subfamily of a left-cancellative category~$\CCC$ that 
satisfies $\CCCi \SSSs \subseteq \SSSs$. Then  every element  that admits an 
$\SSS$-normal decomposition admits one in which all entries except possibly 
the last one lie in~$\SSS \setminus\nobreak \CCCi$.
\end{lemm}

\begin{proof}
Assume  first  that $\gg$ is non-invertible and $\seqqq{\gg_1}\etc{\gg_\qq}$ 
is an $\SSS$-normal decomposition of~$\gg$. The assumption that $\gg$ is 
non-invertible implies that $\LGG\SSS\gg$ is at least one, and, by 
Lemma~\ref{L:Length}, we can assume that $\qq$ is the $\SSS$-length of~$\gg$, 
implying that none of $\gg_1 \wdots \gg_\qq$ is invertible.

As $\gg_1$ is a non-invertible element of~$\SSSs$, we can write $\gg_1 = 
\gg'_1 \ee_1$ with~$\gg'_1$ in~$\SSS$ and $\ee_1$ in~$\CCCi$. Moreover, the 
assumption that $\gg_1$ is not invertible implies that $\gg'_1$ is not 
invertible either. Next, $\ee_1 \gg_2$ belongs to~$\CCCi \SSSs$, hence, by 
assumption, to~$\SSSs$. So we can write $\ee_1 \gg_2 = \gg'_2 \ee_2$ 
with~$\gg'_2$ in~$\SSS$ and $\ee_2$ in~$\CCCi$ and, again, the assumption that 
$\gg_2$ is non-invertible implies that $\gg'_2$ is non-invertible. We continue 
in the same way until~$\gg_\qq$, finding $\gg'_\qq$ in~$\SSS \setminus \CCCi$ 
and $\ee_\qq$ in~$\CCCi$ that satisfy $\ee_{\qq-1} \gg_\qq = \gg'_\qq 
\ee_\qq$. Then $\seqqqq{\gg'_1}\etc{\gg'_{\qq-1}}{\gg'_\qq\ee_\qq}$ is a 
decomposition of~$\gg$ whose non-terminal entries are non-invertible elements 
of~$\SSS$. 

It remains to see that the latter path is $\SSS$-greedy. But this follows from 
the domino rule (Lemma~\ref{L:Domino}) since, for every~$\ii$, the paths 
$\seqq{\gg_\ii}{\gg_{\ii+1}}$ and $\seqq{\gg'_\ii}{\ee_\ii}$ are 
$\SSS$-greedy. Hence $\seqqqq{\gg'_1}\etc{\gg'_{\qq-1}}{\gg'_\qq\ee_\qq}$ is 
an $\SSS$-normal decomposition of~$\gg$ with the expected properties.

 On the other hand, if $\ee$ is invertible, then, by definition, the length 
one path $\seq\ee$ is an $\SSS$-normal decomposition of~$\ee$ vacuously 
satisfying the condition of the statement.
\end{proof}

We can now complete one direction in the implications of 
Proposition~\ref{P:RecGarII}.

\begin{proof}[Proof of Proposition~\ref{P:RecGarII} ($\Rightarrow$)]
Assume that $\SSS$ is a Garside family in~$\CCC$. First, as already noted, 
$\SSSs$ must generate~$\CCC$. Next, by Lemma~\ref{L:Closed}, $\SSSs$ is closed 
under right-complement, right-comultiple, and right-divisor. Finally, let 
$\gg$ be a non-invertible element of~$(\SSSs)^2$. Then the $\SSS$-length 
of~$\gg$ is not zero. 

Assume first $\LGG\SSS\gg = 1$. Then $\gg$ belongs to~$\SSSs \setminus \CCCi$, 
so it can be written as $\gg_1 \ee$ with $\gg_1 \in \SSS$ and $\ee \in \CCCi$, 
in which case $\seqq{\gg_1}\ee$ is an $\SSS$-normal decomposition of~$\gg$, 
and, by Lemma~\ref{L:Head}, $\gg_1$ is an $\SSS$-head of~$\gg$. 

Assume  next  $\LGG\SSS\gg \ge 2$. By Lemmas~\ref{L:Closed} 
and~\ref{L:Strict}, $\gg$ admits an $\SSS$-normal 
decomposition~$\seqqq{\gg_1}\etc{\gg_\qq}$ such that $\gg_1\wdots \gg_{\qq-1}$ 
lie in~$\SSS$. Then $\seqq{\gg_1}{\gg_2 \pdots \gg_\qq}$ is $\SSS$-greedy and, 
by Lemma~\ref{L:Head}, $\gg_1$ is an $\SSS$-head of~$\gg$.

Hence \eqref{E:RecGarII1} and \eqref{E:RecGarII2} are satisfied, and 
so 
is~\eqref{E:RecGarII3} as an $\SSS$-head is \emph{a fortiori} a 
$\div$-maximal left-divisor lying in~$\SSS$.
\end{proof}

 Before going to the second half of the proof of Proposition~\ref{P:RecGarII}, 
we add two more observations about heads for further reference, namely a 
characterization of those elements that admit an $\SSS$-head and a connection 
between $\SSS$- and $\SSSs$-heads.

\begin{lemm}
\label{L:HeadExist}
Assume that $\SSS$ is a subfamily of a left-cancellative category~$\CCC$. 

\ITEM1 If $\SSS$ is a Garside family, an element of~$\CCC$ admits an 
$\SSS$-head if and only if it lies in~$\SSS\CCC$.

\ITEM2 Every $\SSS$-head is an $\SSSs$-head.
\end{lemm}

\begin{proof}
\ITEM1 If $\gg$ admits an $\SSS$-head, then, by definition, $\gg$ is 
left-divisible by an element of~$\SSS$, so it belongs to~$\SSS\CCC$. 
Conversely, assume that $\gg$ lies in~$\SSS\CCC$, say $\gg = \gg_1 \gg'$ with 
$\gg_1 \in \SSS$. If $\gg$ is not invertible, $\gg$ admits an $\SSS$-head by 
Proposition~\ref{P:RecGarII}. Otherwise, $\gg'$ must be invertible, so $\gg 
\eqir \gg_1$ holds, and every element of~$\SSS$ that left-divides~$\gg$ also 
left-divides~$\gg_1$. Then $\gg_1$ is an $\SSS$-head of~$\gg$.

\ITEM2 Assume that $\gg_1$ is an $\SSS$-head of~$\gg$. First $\gg_1$ lies 
in~$\SSS$, hence \emph{a fortiori} in~$\SSSs$. Next, assume that $\hh$ belongs 
to~$\SSSs$ and $\hh \dive \gg$ holds. If $\hh$ is invertible, then $\hh \dive 
\gg_1$ necessarily holds. Otherwise, write $\hh = \hh' \ee$ with $\hh'$ 
in~$\SSS$ and $\ee$ in~$\CCCi$. Then $\hh \dive \gg$ implies $\hh' \dive \gg$, 
whence $\hh' \dive \gg_1$ since $\gg_1$ is an $\SSS$-head of~$\gg$, hence $\hh 
\dive \gg_1$. So $\gg_1$ is an $\SSSs$-head of~$\gg$.
\end{proof}

We now establish several preliminary results in view of the converse 
implication in Proposition~\ref{P:RecGarII}.

\begin{lemm}
\label{L:RCRD}
Assume that $\SSS$ is a subfamily of a left-cancellative category~$\CCC$  such 
that $\SSSs$  generates~$\CCC$ and is closed under right-complement. Then 
$\SSSs$ is closed under  right-divisor.
\end{lemm}

\begin{proof}
The argument is similar to that used for Lemma~\ref{L:GreedyClosed}, see 
Figure~\ref{F:RCRD}. So assume that $\hh$ belongs to~$\SSSs$ and $\gg$ 
right-divides~$\hh$, that is, we have $\hh = \ff \gg$. As $\SSSs$ 
generates~$\CCC$, we can write $\ff = \ff_1 \pdots \ff_\pp$ with $\ff_1 \wdots 
\ff_\pp \in \SSSs$. The assumption that $\SSSs$ is closed under 
right-complement and the fact that $\hh$ and $\ff_1$ left-divide~$\ff_1 \pdots 
\ff_\pp \gg$ imply the existence of~$\ff'_1, \hh_1$ in~$\SSSs$ satisfying $\hh 
\ff'_1 = \ff_1 \hh_1 \dive \ff_1 \pdots \ff_\pp \gg$. By 
left-cancelling~$\ff_1$, we deduce that $\hh_1$ and $\ff_2$ left-divide~$\ff_2 
\pdots \ff_\pp \gg$, whence the existence of~$\ff'_2, \hh_2$ satisfying $\hh_1 
\ff'_2 = \ff_2 \hh_2 \dive \ff_2 \pdots \ff_\pp \gg$, and so on until  
$\hh_{\pp-1} \ff'_\pp = \ff_\pp \hh_\pp \dive  \ff_\pp \gg$, which implies 
$\hh_\pp \dive \gg$,  say $\hh_\pp \ee = \gg$. By construction, $\ff'_1 \pdots 
\ff'_\pp \ee$ is an identity-element, hence all of the factors must 
be invertible. Now $\hh_\pp$ belongs to~$\SSSs$, hence $\gg$, which is 
$\hh_\pp \ee$, belongs to~$\SSSs \CCCi$, which is $\SSSs$. So $\SSSs$ is 
closed under right-divisor.
\end{proof}

\begin{figure}[htb]
\begin{picture}(75,32)(0,0)
\pcline{->}(1,24)(14,24)
\taput{$\ff_1$}
\pcline{->}(16,24)(29,24)
\taput{$\ff_2$}
\pcline[style=etc](31,24)(44,24)
\pcline{->}(46,24)(59,24)
\taput{$\ff_\pp$}
\psline(61,24)(68,24)
\psarc(68,17){7}{0}{90}
\pcline{->}(75,17)(75,1)
\trput{$\gg$}

\pcline[style=exist]{->}(1,12)(14,12)
\tbput{$\ff'_1$}
\pcline[style=exist]{->}(16,12)(29,12)
\tbput{$\ff'_2$}
\pcline[style=etc](31,12)(44,12)
\pcline[style=exist]{->}(46,12)(59,12)
\tbput{$\ff'_\pp$}
\pcline[style=exist]{->}(61,12)(74,1)
\put(68,7.5){$\ee$}

\psline[style=double](0,11)(0,7)
\psarc[style=double](7,7){7}{180}{270}
\psline[style=double](7,0)(74,0)
\psline[style=exist](15,11)(15,7)
\psarc[style=exist](22,7){7}{180}{250}
\psline[style=exist](30,11)(30,7)
\psarc[style=exist](37,7){7}{180}{250}
\psline[style=exist](45,11)(45,7)
\psarc[style=exist](52,7){7}{180}{250}

\pcline{->}(0,23)(0,13)
\tlput{$\hh$}
\pcline[style=exist]{->}(15,23)(15,13)
\trput{$\hh_1$}
\pcline[style=exist]{->}(30,23)(30,13)
\trput{$\hh_2$}
\pcline[style=exist]{->}(45,23)(45,13)
\trput{$\hh_{\pp-1}$}
\pcline[style=exist]{->}(60,23)(60,13)
\trput{$\hh_\pp$}

\psline[style=thin](0,27)(0,30)(60,30)(60,27)
\put(29,32){$\ff$}
\end{picture}
\caption[]{\sf\smaller Proof of Lemma~\ref{L:RCRD}: using the closure under 
right-complement, one fills the diagram, and concludes that $\gg$ must lie 
in~$\SSSs$, since $\hh_\pp$ lies in~$\SSSs$ and $\ee$ is invertible.}
\label{F:RCRD}
\end{figure}

\begin{lemm}
\label{L:HeadNormal}
Assume that $\SSS$ is a subfamily of a left-cancellative category~$\CCC$ that 
generates~$\CCC$ and is closed under right-complement. Then every element 
of~$(\SSSs)^2$ that admits an $\SSS$-head admits an $\SSS$-normal 
decomposition.
\end{lemm}

\begin{proof}
Assume that $\gg$ is $\hh_1 \hh_2$ with $\hh_1, \hh_2 \in \SSSs$ and $\gg$ 
admits an $\SSS$-head, say~$\gg_1$. By assumption, we have $\hh_1 \dive \gg$ 
with $\hh_1 \in \SSSs$.  By Lemma~\ref{L:HeadExist}, $\gg_1$ is an 
$\SSSs$-head of~$\gg$, implying $\hh_1 \dive \gg_1$, say  $\gg_1 = \hh_1 \ff$. 
Then we have $\gg = \hh_1 \ff \gg_2 = \hh_1 \hh_2$, whence $\hh_2 = \ff 
\gg_2$. It follows that $\gg_2$ right-divides~$\hh_2$, an element of~$\SSSs$. 
By Lemma~\ref{L:RCRD}, $\SSSs$ must be closed under right-divisor, so $\gg_2$ 
lies in~$\SSSs$. Hence $\seqq{\gg_1}{\gg_2}$ is a decomposition of~$\gg$ whose 
entries lie in~$\SSSs$.

Finally, by Lemma~\ref{L:Head}, the assumption that $\gg_1$ is an $\SSS$-head 
of~$\gg_1 \gg_2$ implies that $\seqq{\gg_1}{\gg_2}$ is $\SSS$-greedy, hence it 
is an $\SSS$-normal decomposition of~$\gg$.
\end{proof}

\begin{lemm}
\label{L:CMHead}
Assume that $\SSS$ is a subfamily of a left-cancellative category~$\CCC$ such 
that $\SSSs$ is closed under right-comultiple. Then a $\div$-maximal 
left-divisor in~$\SSS$ is an $\SSS$-head.
\end{lemm}

\begin{proof}
Assume that $\gg_1$ is a $\div$-maximal left-divisor of~$\gg$ lying 
in~$\SSS$. Write $\gg = \gg_1 \gg_2$. Let $\hh$ be an arbitrary left-divisor 
of~$\gg$ lying in~$\SSS$. We wish to prove $\hh \dive \gg_1$. Now, $\gg$ is a 
common right-multiple of~$\hh$ and~$\gg_1$, which both lie in~$\SSS$, hence 
in~$\SSSs$. As the latter is closer under right-comultiple, there must exist a 
common right-multiple~$\hh'$ of~$\hh$ and~$\gg_1$ that lies in~$\SSSs$ and 
left-divides~$\gg$. As we have $\gg_1 \dive \hh'$, the assumption that $\gg_1$ 
is a $\div$-maximal left-divisor of~$\gg$ implies $\gg_1 \eqir \hh'$, whence 
$\hh \dive \hh' \dive \gg_1$, and, finally, $\hh \dive \gg_1$. Hence $\gg_1$ 
is an $\SSS$-head of~$\gg$. 
\end{proof}

\begin{proof}[Proof of Proposition~\ref{P:RecGarII} ($\Leftarrow$)]
Assume first that $\SSS$ satisfies~\eqref{E:RecGarII1}.
We claim that $\SSS$ is also closed under right-comultiple and 
right-complement. Indeed, assume that $\ff, \gg$ belong to~$\SSSs$ and $\ff, 
\gg \dive \hh$ hold. Assume first $\hh \notin \CCCi$, and let $\hh_1$ be an 
$\SSS$-head of~$\hh$. As $\ff$ belongs to~$\SSSs$ and it left-divides~$\hh$, 
we must have $\ff \dive \hh_1$, hence $\ff \gg' = \hh_1$ for some~$\gg'$, and, 
similarly, $\gg \ff' = \hh_1$ for some~$\ff'$. As $\SSSs$ is closed under 
right-divisor, $\ff'$ and $\gg'$ must belong to~$\SSSs$. If $\hh$ is 
invertible, then $\hh$ belongs to~$\SSSs$ and we can take $\hh_1 = \hh$, $\gg' 
= \ff\inv \hh_1$, and $\ff' = \gg\inv \hh_1$. In all cases, $\hh_1, \ff', 
\gg'$ witness that $\SSSs$ is closed under right-comultiple and 
right-complement. Hence, \eqref{E:RecGarII1} implies~\eqref{E:RecGarII2}.

Next, assume that $\SSS$ satisfies~\eqref{E:RecGarII2}. Then, by assumption, 
every element of~$(\SSSs)^2$ admits an $\SSS$-head, hence, by 
Lemma~\ref{L:HeadNormal}, an $\SSS$-normal decomposition. Hence, by 
Proposition~\ref{P:RecGarI}, $\SSS$ is a Garside family in~$\CCC$. 

Finally, assume that $\SSS$ satisfies~\eqref{E:RecGarII3}. First, by 
Lemma~\ref{L:RMRC}, the assumption that $\SSSs$ is closed under 
right-comultiple and right-divisor implies that it is closed under 
right-complement. Then, by Lemma~\ref{L:CMHead}, the existence of a 
$\div$-maximal left-divisor in~$\SSS$ for every element~$\gg$ of~$(\SSSs)^2$ 
implies the existence of an $\SSS$-head whenever at least one such divisor 
lies in~$\SSS$, which is guaranteed when $\gg$ is non-invertible. So 
\eqref{E:RecGarII3} implies~\eqref{E:RecGarII2}, and $\SSS$ must be a Garside 
family again.
\end{proof}

%%%%
\subsection{Special contexts}
\label{SS:Special}

All results established so far are valid in an arbitrary left-cancellative 
category. When the ambient category happens to satisfy additional properties, 
some of the conditions involved in the characterizations of Garside families 
may be automatically satisfied or take simpler forms.

\begin{defi}
\label{D:Noeth}
A category~$\CCC$ is called \emph{right-Noetherian} if right-divisibility 
in~$\CCC$ is a well-founded relation, that is, every nonempty subfamily has a 
least element.
\end{defi}

Above, by a least element we mean an element that right-divides every element of the considered subfamily. We recall that $\ff \div \gg$ stands for the conjunction of $\ff \dive \gg$ and $\ff \not\dive \gg$, that is, by Lemma~\ref{L:DivOrder}, the conjunction of $\ff \dive \gg$ and $\ff \noteqir \gg$.

\begin{lemm}
\label{L:Noeth}
A left-cancellative category~$\CCC$ is right-Noetherian if and only if there 
is no bounded $\div$-increasing sequence in~$\CCC$.
\end{lemm}

\begin{proof}
Assume that $\CCC$ is right-Noetherian and we have $\gg_1 \dive \gg_2 \dive 
\pdots \dive \gg$ in~$\CCC$. For each~$\ii$, write $\gg_\ii \ff_\ii = 
\gg_{\ii+1}$ and $\gg_\ii \hh_\ii = \gg$. Then we have $\gg_\ii \hh_\ii = 
\gg_{\ii+1} \hh_{\ii+1} = \gg_\ii \ff_\ii \hh_{\ii+1}$, whence $\hh_\ii = 
\ff_\ii \hh_{\ii+1}$ by left-cancelling $\gg_\ii$. So the sequence $(\hh_1, 
\hh_2, ...)$ is decreasing for right-divisibility. Let $\hh_\NN$ be a least 
element in $\{\hh_\ii \mid \ii \ge 1\}$. Necessarily $\ff_\ii$ is invertible 
for $\ii \ge \NN$, which implies  $\gg_{\ii+1} \eqir \gg_\ii$.  So $(\gg_1, 
\gg_2, ...)$ cannot be $\div$-increasing.

Conversely, assume that $\CCC$ contains no bounded $\div$-increasing sequence 
in~$\CCC$, and let $(\ff_1, \ff_2, ...)$ be a decreasing sequence with respect 
to right-divisibility. Write $\ff_\ii = \gg'_\ii \ff_{\ii+1}$. Set $\gg = 
\ff_1$, and $\gg_\ii = \gg'_1 \pdots \gg'_\ii$. Then we have $\gg_1 \dive 
\gg_2 \dive \pdots \dive \gg$. The assumption that $\CCC$ contains no bounded 
$\div$-increasing sequence implies that $\gg'_\ii$ is invertible for~$\ii$ 
large enough. It follows that right-divisibility admits no infinite descending 
sequence. By the Axiom of Dependent Choices, this implies that 
right-divisibility is a well-founded relation, that is, that $\CCC$ is 
right-Noetherian.
\end{proof}

\begin{prop}
\label{R:RecGarNoeth}
A subfamily~$\SSS$ of a left-cancellative category that is right-Noetherian is 
a Garside family if and only if 
\begin{equation}
\label{E:RecGarNoeth}
\BOX{$\SSSs$ generates~$\CCC$ and it is closed under right-comultiple and 
right-divisor.}
\end{equation}
\end{prop}

\begin{proof}
Owing to Proposition~\ref{P:RecGarII} and~\eqref{E:RecGarII3}, it is enough to 
show that every non-invertible element~$\gg$ in~$\CCC$ (or only 
in~$(\SSSs)^2)$ admits a $\div$-maximal left-divisor lying in~$\SSS$. First, 
as $\SSSs$ generates~$\CCC$ and $\gg$ is non-invertible, the latter is 
left-divisible by some non-invertible element of~$\SSSs$, hence by some 
(non-invertible) element of~$\SSS$, say~$\gg_1$. If $\gg_1$ is not 
$\div$-maximal among left-divisors of~$\gg$ lying in~$\SSS$, we can find 
$\gg_2$ satisfying $\gg_1 \div \gg_2 \dive \gg$. If $\gg_2$ is not 
$\div$-maximal, we find $\gg_3$ satisfying $\gg_2 \div \gg_3 \dive \gg$, and 
so on. By Lemma~\ref{L:Noeth}, the construction cannot be repeated infinitely 
many times, which means that some $\gg_\NN$ must be $\div$-maximal among 
left-divisors of~$\gg$ lying in~$\SSS$. Then, every non-invertible element of 
$(\SSSs)^2$ admits a $\div$-maximal left-divisor in~$\SSS$.  So $\SSS$ 
satisfies \eqref{E:RecGarII3} hence, by Proposition~\ref{P:RecGarII}, it  is a 
Garside family in~$\CCC$.
\end{proof}

A further specialization leads to categories that admit least common 
right-multiples. 

\begin{defi}
Assume that $\CCC$ is a left-cancellative category. We say that $\hh$ is a 
\emph{least common right-multiple}, or \emph{right-lcm}, of~$\ff$ and~$\gg$ if 
$\hh$ is a right-multiple of~$\ff$ and~$\gg$ and every element~$\hh'$ that is 
a right-multiple of~$\ff$ and~$\gg$ is a right-multiple of~$\hh$. A 
subfamily~$\SSS$ of~$\CCC$ is said to be \emph{closed under right-lcm} 
in~$\CCC$ if every right-lcm of two elements of~$\SSS$ lies in~$\SSS$. 
\end{defi}

In other words, a right-lcm is a least common upper bound with 
respect to 
left-divisibility.

\begin{prop}
\label{R:RecGarLcm}
Assume that $\CCC$ is a left-cancellative category that is right-Noetherian 
and such that any two elements of~$\CCC$ that admit a common right-multiple 
admit a right-lcm. Then a subfamily~$\SSS$ of~$\CCC$ is a Garside family if 
and only if 
\begin{equation}
\label{E:RecGarLcm}
\BOX{$\SSSs$ generates~$\CCC$ and it is closed under right-lcm and 
right-divisor.}
\end{equation}
\end{prop}

\begin{proof}
Assume that $\SSS$ is a Garside family in~$\CCC$. Then, by 
Proposition~\ref{P:RecGarII}, $\SSSs$ is closed under right-comultiple. Assume 
that $\ff, \gg$ are elements of~$\SSSs$ that admit a common 
right-multiple. Then 
$\ff$ and~$\gg$ admit a right-lcm, say~$\hh$. As $\SSSs$ is closed under 
right-comultiple, there exists~$\hh'$ in~$\SSSs$ satisfying $\ff, \gg \dive 
\hh' \dive \hh$. By definition of a right-lcm, $\hh \dive \hh'$ must hold, 
whence $\hh' \eqir \hh$. So there exists an invertible element~$\ee$ 
satisfying 
$\hh = \hh' \ee$, whence $\hh \in \SSSs \CCCi = \SSSs$. This shows that 
$\SSSs$ is closed under right-lcm.

Conversely, assume that $\SSS$ satisfies~\eqref{E:RecGarLcm} and we have $\ff, 
\gg \dive \hh$ with $\ff, \gg \in \SSSs$. As $\ff, \gg$ admit a common 
right-multiple, they admit a right-lcm, say~$\hh'$, and, by assumption, $\hh'$ 
belongs to~$\SSSs$. By definition of a right-lcm, $\hh' \dive \hh$ holds, so 
$\SSSs$ is closed under right-comultiple. Then,  by 
Proposition~\ref{R:RecGarNoeth},  $\SSS$ is a Garside family in~$\SSS$.
\end{proof}

Specializing even more, we consider the case of categories that are both 
right- and left-Noetherian, the latter meaning that left-divisibility is well 
founded. In that case, standard arguments show that every non-invertible  
element  is a product of \emph{atoms}, defined to be those non-invertible 
elements that cannot be expressed as the product of at least two 
non-invertible elements. Then, if $\CCC$ is a Noetherian category containing 
no nontrivial invertible element, a subfamily~$\SSS$ of~$\CCC$ 
generates~$\CCC$ if and only if $\SSS$ contains all atoms of~$\CCC$. Then 
Proposition~\ref{R:RecGarLcm} directly implies

\begin{prop}
\label{R:RecGarAtom}
Assume that $\CCC$ is a left-cancellative category that is Noetherian, such 
that any two elements that admit a common right-multiple admit a right-lcm, 
and contains no nontrivial invertible element. Then a subfamily~$\SSS$ 
of~$\CCC$ is a Garside family if and only if 
\begin{equation}
\label{E:RecGarAtom}
\BOX{$\SSS$ contains the atoms of~$\CCC$ and it is closed under right-lcm and 
right-divisor.}
\end{equation}
\end{prop}

\begin{coro}
Under the assumptions of Proposition~\ref{R:RecGarAtom}, there exists a 
smallest Garside family in~$\CCC$.
\end{coro}

\begin{proof}
 Under the considered assumptions, an intersection of Garside families is a 
Garside family since under these assumptions a Garside family is defined by 
closure properties plus the fact that it contains all atoms.  In 
particular, the intersection of all Garside families of~$\CCC$ is a smallest 
Garside family in~$\CCC$.
\end{proof}

%%%%
\subsection{Recognizing Garside families III: head functions}
\label{SS:Head}

We return to the general case, and establish further characterizations of 
Garside families, this time in terms of what is called a head 
function.

\begin{defi}
\label{D:HLaw}
Assume that $\CCC$ is a left-cancellative category. A partial map~$\HH$ 
of~$\CCC$ to itself is said to obey the \emph{$\HHH$-law} if 
\begin{equation}
\label{E:HLaw}
\HH(\ff\gg) \eqir \HH(\ff\HH(\gg))
\end{equation}
holds whenever both terms are defined. If \eqref{E:HLaw} holds with equality 
instead of~$\eqir$, we say that $\HH$ obeys the \emph{sharp} $\HHH$-law. 
\end{defi}

 In the same context, in addition to the $\HHH$-law, we shall also consider 
the following conditions, supposed to hold for all~$\ff, \gg$ in the domain of 
the considered map~$\HH$:
\begin{equation}
\label{E:RecGarIII0}
\mbox{\ITEM1~$\HH(\gg) \dive \gg$, 
\quad \ITEM2~$\ff \dive \gg \Rightarrow \HH(\ff) \dive \HH(\gg)$, 
\quad \ITEM3~$\gg \in \SSSs \Rightarrow \HH(\gg) \eqir \gg$.}
\end{equation}

\begin{prop}
\label{P:RecGarIII}
A subfamily~$\SSS$ of a left-cancellative category~$\CCC$ is a Garside family 
if and only if it satisfies one of the following equivalent 
conditions:
\begin{gather}
\label{E:RecGarIII1}
\BOX{$\SSSs$ generates~$\CCC$ and there exists $\HH : \CCC \setminus \CCCi \to 
\SSS$ satisfying the $\HHH$-law and~\eqref{E:RecGarIII0};}\\
\label{E:RecGarIII2}
\BOX{$\SSSs$ generates~$\CCC$, $\CCCi\SSSs \subseteq \SSSs$ holds, and there 
exists $\HH : \SSS\CCC \to \SSS$ satisfying the sharp $\HHH$-law 
and~\eqref{E:RecGarIII0}.}
\end{gather}
\end{prop}

\begin{proof}
Let us first observe that, if $\SSSs$ generates~$\CCC$ and $\CCCi\SSSs 
\subseteq \SSSs$ holds, then $\CCC \setminus \CCCi$ is included in~$\SSS\CCC$: 
indeed, in this case, if $\gg$ is non-invertible, it must be left-divisible by 
an element of the form~$\ee\hh$ with $\ee \in \CCCi$ and $\hh \in 
\SSS\setminus \CCCi$, hence by an element of~$\SSS\setminus \CCCi$ owing to 
$\CCCi\SSSs \subseteq \SSSs$. Hence \eqref{E:RecGarIII2} 
implies~\eqref{E:RecGarIII1}.

Assume that $\SSS$ is a Garside family in~$\CCC$. For~$\gg$ in~$\SSS\CCC$, 
define~$\HH(\gg)$ to be an $\SSS$-head of~$\gg$, which exists  by 
Lemma~\ref{L:HeadExist}.  Then, by definition, $\HH(\gg) \dive \gg$ always 
holds. Next, if $\ff \dive \gg$ holds, $\HH(\ff)$ is an element of~$\SSS$ that 
left-divides~$\ff$, hence~$\gg$, so $\HH(\ff) \dive \HH(\gg)$ must hold. Then, 
if $\gg$ belongs to~$\SSSs \cap \SSS\CCC$, we have $\gg \eqir \gg'$ for 
some~$\gg'$ in~$\SSS$, and we deduce $\gg' \dive \HH(\gg) \dive \gg$, whence 
$\HH(\gg) \eqir \gg$.  So $\HH$ satisfies~\eqref{E:RecGarIII0}.  Finally, 
assume that $\ff\gg$ exists and $\gg$ lies in~$\SSS\CCC$. If $\gg$ is 
non-invertible, then $\HH(\gg)$ is non-invertible as well since $\LGG\SSS\gg 
\ge 1$ holds, meaning that $\gg$ is left-divisible by at least one 
non-invertible element of~$\SSS$, so both $\HH(\ff\gg)$ and $\HH(\ff\HH(\gg))$ 
are defined in this case. On the other hand, if $\gg$ is invertible, then 
$\HH(\gg) \eqir \gg$ holds, and $\ff\gg$ belongs to~$\SSS\CCC$ if and only if 
$\ff\HH(\gg)$ does. In all cases, by~\eqref{E:RecGarIII0}\ITEM1, we 
have $\HH(\gg) \dive \gg$, 
whence $\ff\HH(\gg) \dive \ff\gg$, and, by~\eqref{E:RecGarIII0}\ITEM2, 
$\HH(\ff\HH(\gg)) \dive 
\HH(\ff\gg)$. On the other hand, write $\gg = \HH(\gg) \gg'$. By 
Lemma~\ref{L:Head}, the path $\seqq{\HH(\gg)}{\gg'}$ is $\SSS$-greedy. Then 
the relation $\HH(\ff\gg) \dive \ff\gg$, which holds 
by~\eqref{E:RecGarIII0}\ITEM1, implies 
$\HH(\ff\gg) \dive \ff \HH(\gg)$, whence $\HH(\ff\gg) \dive \HH(\ff \HH(\gg))$ 
since $\HH(\ff\gg)$ lies in~$\SSS$. So we deduce $\HH(\ff\gg) \eqir 
\HH(\ff\HH(\gg))$, and $\HH$ satisfies the  $\HHH$-law.

Next, let $\SSS_0$ be an $\eqir$-selector on~$\SSS$, that is, a subfamily 
of~$\SSS$ that contains exactly one element in each $\eqir$-class (which 
exists by the Axiom of Choice). For~$\gg$ in~$\SSS\CCC$, define $\HH_0(\gg)$ 
to be the unique element of the $\eqir$-class of~$\HH(\gg)$ that lies 
in~$\SSS_0$. By construction, $\HH_0(\gg) \eqir \HH(\gg)$ holds for 
every~$\gg$ in~$\SSS\CCC$ and, therefore, the function~$\HH_0$  
satisfies~\eqref{E:RecGarIII0}  and the $\HHH$-law: for the latter, we have 
$\ff \HH(\gg) \eqir \ff\HH_0(\gg)$, whence, 
by~\eqref{E:RecGarIII0}\ITEM2, $\HH(\ff\HH(\gg)) \eqir 
\HH(\ff\HH_0(\gg))$, and, from there, 
$$\HH_0(\ff \gg) \eqir \HH(\ff\gg) \eqir \HH(\ff\HH(\gg)) \eqir \HH(\ff 
\HH_0(\gg)) \eqir \HH_0(\ff\HH_0(\gg)).$$
Now, by construction, $\HH_0(\ff\gg)$ and  $\HH_0(\ff\HH_0(\gg))$  belong 
to~$\SSS_0$, so, being $\eqir$-equivalent, they must be equal. Hence $\HH_0$ 
obeys the sharp $\HHH$-law, and \eqref{E:RecGarIII2} is satisfied. By the 
initial remark, \eqref{E:RecGarIII1} is satisfied too. 

Conversely, assume that $\SSS$ satisfies~\eqref{E:RecGarIII1}, with $\HH : 
\CCC \setminus \CCCi \to \SSS$ witnessing the expected conditions.  We shall 
prove that $\SSS$ satisfies~\eqref{E:RecGarII1}.  The first step is to show 
that, for each non-invertible~$\gg$, the element~$\HH(\gg)$ is an $\SSS$-head 
of~$\gg$. So let $\gg$ be a non-invertible element of~$\CCC$. By assumption, 
$\HH(\gg)$ belongs to~$\SSS$ and left-divides~$\gg$. Assume $\hh \dive \gg$ 
with $\hh \in \SSS$. By \eqref{E:RecGarIII0}\ITEM2 and~\ITEM3, we have $\hh 
\eqir \HH(\hh) \dive \HH(\gg)$, whence $\hh \dive \HH(\gg)$. Hence $\HH(\gg)$ 
is an $\SSS$-head of~$\gg$.

%The second step is to show that $\SSSs$ is closed under right-comultiple. So 
%assume $\ff, \gg \in \SSSs$ and $\ff, \gg \dive \hh$. As $\ff$ and $\gg$ are 
%$\eqir$-equivalent to elements of~$\SSS$, we deduce $\ff \dive \HH(\hh)$ and 
%$\gg \dive \HH(\hh)$, that is, $\HH(\hh)$ is a common right-multiple of~$\ff$ 
%and~$\gg$ that lies in~$\SSS$, hence in~$\SSSs$. 

The  second  step is to show that $\SSSs$ is closed under right-divisor. So 
assume that $\gg$ belongs to~$\SSSs$ and $\ff$ right-divides~$\gg$, say $\gg = 
\gg' \ff$. If $\ff$ is invertible, it belongs to~$\SSSs$ by definition. Assume 
now that $\ff$, and therefore~$\gg$, are not invertible. Then we have $\gg 
\eqir \hh$ for some~$\hh$ lying in~$\SSS$.  By~\eqref{E:RecGarIII0}\ITEM2  
and \ITEM3, we have $\HH(\gg) \eqir \HH(\hh)$, whence $\HH(\hh) \eqir \hh$. We 
deduce $\HH(\gg) \eqir \gg$, that is, $\HH(\gg' \ff) \eqir \gg' \ff$. Then the 
$\HHH$-law implies $\HH(\gg' \ff) \eqir \HH(\gg' \HH(\ff))$, whereas  
\eqref{E:RecGarIII0}\ITEM1  gives $\HH(\gg' \HH(\ff)) \dive \gg' \HH(\ff)$, so 
we find 
$$\gg' \ff \eqir \HH(\gg' \ff) \eqir \HH(\gg' \HH(\ff)) \dive \gg' \HH(\ff),$$
hence $\gg' \ff \dive \gg' \HH(\ff)$, and $\ff \dive \HH(\ff)$ by 
left-cancelling~$\gg'$. As $\HH(\ff) \dive \ff$ always holds  
by~\eqref{E:RecGarIII0}\ITEM1,  we deduce $\ff \eqir \HH(\ff)$. As, by 
definition, $\HH(\ff)$ lies in~$\SSS$, it follows that $\ff$ lies in~$\SSSs$. 
So $\SSSs$ is closed under right-divisor. 

It is now easy to conclude. 
%By Lemma~\ref{L:RMRC}, $\SSSs$ is closed under right-complement since it is 
%closed under right-comultiple and right-divisor. 
By assumption, $\SSSs$ generates~$\CCC$, we saw above that it is closed under 
right-divisor, and that every non-invertible element of~$\CCC$ admits an 
$\SSS$-head. So $\SSS$ satisfies~\eqref{E:RecGarII1} and, by 
Proposition~\ref{P:RecGarII}, it is a Garside family in~$\CCC$. 
\end{proof}

%%%%%%%%
\section{Germs}
\label{S:Germs}

So far, we have established extrinsic characterizations of Garside families, namely 
conditions that describe a Garside family in a given pre-existing category. We 
turn now to intrinsic characterizations, that is, we do not start from a 
pre-existing category but consider instead an abstract family~$\SSS$ equipped 
with a partial product and investigate necessary and sufficient conditions for 
such a structure, here called a germ, to generate a category in which $\SSS$ 
embeds as a Garside family. 

%%%%
\subsection{The notion of a germ}
\label{SS:Germ}

If $\SSS$ is a subfamily of a category~$\CCC$, then, for $\ff, \gg$ in~$\SSS$ 
such that $\ff \gg$ is defined, $\ff \gg$ may belong or not to~$\SSS$. 
Restricting to the case when the product belongs to~$\SSS$ gives a partial map from~$\Seq\SSS2$ to~$\SSS$. The resulting structure will be 
called a germ, and we shall be interested in the case when the whole category 
can be retrieved from the germ.

In the above situation, we shall denote by~$\OP_\SSS$, or simply $\OP$, the 
partial operation on~$\SSS$ induced by the product of the ambient category. It 
is easy to see that a partial operation of this type must obey some 
constraints.

\begin{lemm}
\label{L:Germ}
Assume that $\SSS$ is a subfamily of a category~$\CCC$  that contains all 
identity-elements of~$\CCC$.  The partial operation~$\OP$ of~$\Seq\SSS2$ 
to~$\SSS$ induced by the product of~$\CCC$ obeys the following rules:
\begin{gather}
\label{E:Germ1}
\BOX{If $\ff \OP \gg$ is defined, the source (\resp.\ target) of~$\ff \OP 
\gg$ is the source of~$\ff$ (\resp.\ the target of~$\gg$);}\\
\label{E:Germ2}
\BOX{$\id\xx \OP \ff = \ff = \ff \OP \id\yy$ hold for each~$\ff$ in~$\SSS(\xx, 
\yy)$,}\\
\label{E:Germ3}
\BOX{If $\ff \OP \gg$ and $\gg \OP \hh$ are defined, then $(\ff \OP \gg) \OP 
\hh$ is defined if and only if $\ff \OP (\gg \OP \hh)$ is, in which case they 
are equal.}
\end{gather}
Moreover, if $\SSS$ is closed under right-divisor in~$\CCC$, then $\OP$ 
satisfies
\begin{equation}
\label{E:LeftAss}
\BOX{If $(\ff \OP \gg) \OP \hh$ is defined, then so is $\gg \OP \hh$.}
\end{equation} 
\end{lemm}

\begin{proof}
Points~\eqref{E:Germ1} and~\eqref{E:Germ2} follow from the fact that $\OP$ is induced by the product of~$\CCC$. Next, \eqref{E:Germ3} follows from 
associativity in~$\CCC$: saying that $(\ff \OP \gg) \OP \hh$ exists means that 
$(\ff \gg) \hh$ belongs to~$\SSS$, hence so does $\ff (\gg \hh)$. As, by 
assumption, $\gg \OP \hh$ exists, this amounts to $\ff \OP( \gg \OP \hh)$ 
being defined.

For~\eqref{E:LeftAss}, the hypotheses imply that $(\ff \gg) \hh$, hence $\ff 
(\gg \hh)$, belongs to~$\SSS$. As $\SSS$ is closed under right-divisor, this 
implies that $\gg \hh$ belongs to~$\SSS$, hence that $\gg \OP \hh$ is 
defined.
\end{proof}

We shall therefore start from abstract families that obey the rules of 
Lemma~\ref{L:Germ}.  We recall from Definition~\ref{D:Path} that $\Seq\SSS\rr$ 
denotes the family of all length~$\rr$ paths in~$\SSS$. 

\begin{defi}
\label{D:Germ}
A \emph{germ} is a triple~$(\SSS, \Id\SSS, \OP)$ where $\SSS$ is a 
precategory, $\Id\SSS$ is a subfamily of~$\SSS$ containing an 
element~$\id\xx$ with source and target~$\xx$ for each object~$\xx$, and $\OP$ 
is a partial map from~$\Seq\SSS2$ to~$\SSS$ that satisfies~\eqref{E:Germ1}, 
\eqref{E:Germ2}, and~\eqref{E:Germ3}. If, moreover, \eqref{E:LeftAss} holds, 
the germ is said to be \emph{left-associative}. If $(\SSS, \Id\SSS, \OP)$ is a 
germ, we denote by~$\Cat(\SSS, \Id\SSS, \OP)$ the 
category~$\PRESp\SSS{\RRR_{\OP}}$, where $\RRR_{\OP}$ is the family of all 
relations~$\seqq\ff\gg = \seq{\ff \OP \gg}$ with $\ff, \gg$ in~$\SSS$ and $\ff 
\OP \gg$ defined. 
\end{defi}

In practice, we shall use the generic notation~$\SSSg$ for a germ with 
domain~$\SSS$. In the sequel, an equality of the form $\gg = \gg_1 \OP \gg_2$ always means ``$\gg_1 \OP \gg_2$ is defined and $\gg$ equals it''. 

For every subfamily~$\SSS$ of a category~$\CCC$, there exists an induced 
germ~$\SSSg$, and the corresponding relations~$\RRR_{\OP}$ are valid 
in~$\CCC$ by construction. Hence $\CCC$ is a quotient of~$\Cat(\SSSg)$. In 
most cases, even if $\SSS$ generates~$\CCC$, the partial product~$\OP$ does 
not determine the product of~$\CCC$, and $\CCC$ is a proper quotient 
of~$\Cat(\SSSg)$. For instance, if $\CCC$ contains no nontrivial 
invertible element and is generated by a family of atoms~$\AAA$, the induced 
partial product on~$\AAA$ consists of the trivial instances listed in~\eqref{E:Germ2} 
only, and the resulting category is a free category based on~$\AAA$. We shall 
now see that this cannot happen when $\SSS$ is a Garside family: in this case, 
the induced structure $\SSSg$ contains enough information to retrieve 
the initial category~$\CCC$. Here one has to be careful: if $\SSS$ is a 
general Garside family in~$\CCC$, then $\CCC$ is generated by~$\SSSs$, hence 
by~$\SSS \cup \CCCi$, but not necessarily by~$\SSS$ itself. To avoid problems, 
we shall restrict to particular Garside families.

\begin{defi}
A subfamily~$\SSS$ of a category~$\CCC$ is called \emph{full} if $\SSS$ contains all identity-elements of~$\CCC$ and it is closed under 
right-divisor in~$\CCC$. 
\end{defi}

It follows from Lemma~\ref{L:Closed} that, for every Garside family~$\SSS$, the family~$\SSSs$ is a full Garside family, which we know gives rise to the same greedy and normal 
paths. So considering full Garside families is not a proper restriction.  (On 
the other hand, for~$\SSS$ to be full is weaker than satisfying $\SSS = 
\SSSs$: one can exhibit a full Garside family~$\SSS$ such that $\SSS$ is 
properly included in~$\SSSs$.)  The precise result we shall prove is then

\begin{prop}
\label{P:GarGerm}
Assume that $\SSS$ is a full Garside family in a left-cancellative 
category~$\CCC$, and let $\SSSg$ be the induced germ. Then 
$\Cat(\SSSg)$ is isomorphic to~$\CCC$.
\end{prop}

By definition, the category~$\Cat(\SSSg)$ is specified by a presentation. In 
order to establish Proposition~\ref{P:GarGerm}, we shall show that every 
(full) Garside family provides a presentation of its ambient category.

\begin{lemm}
\label{L:Present}
Assume $\SSS$ is a generating subfamily of a left-cancell\-at\-ive 
category~$\CCC$ that 
is full and closed under right-comultiple. Then $\CCC$ admits the 
presentation~$\PRESp\SSS\RRR$ where $\RRR$ consists of all relations 
$\seqq\ff\gg = \seq\hh$ with $\ff, \gg, \hh$ in~$\SSS$ that are valid 
in~$\CCC$. 
\end{lemm}

\begin{proof}
First, by assumption, $\SSS$ generates~$\CCC$. Next, by 
definition, all relations of~$\RRR$ are valid in~$\CCC$. So the point is 
to show that every equality involving elements of~$\SSS$ that is valid 
in~$\CCC$ is a consequence of finitely many relations of~$\RRR$. 

So assume that $\seqqq{\ff_1}\etc{\ff_\pp}$ and $\seqqq{\gg_1}\etc{\gg_\qq}$ 
are two $\SSS$-paths with the same evaluation in~$\CCC$, that is, $\ff_1 
\pdots \ff_\pp = \gg_1 \pdots \gg_\qq$ holds in~$\CCC$. Then we inductively 
construct a rectangular grid as displayed in Figure~\ref{F:Present}.

 Put $\hh'_{\ii, 0} = \ff_{\ii+1} \pdots \ff_\pp$ for $0 \le \jj \le \pp$, and 
$\hh'_{0,\jj} = \gg_{\jj+1} \pdots \gg_\qq$ for $0 \le \jj \le \qq$.  First, 
we have $\ff_1 \hh'_{1,0} = \gg_1 \hh'_{0,1} = \hh'_{0,0}$ with $\ff_1, \gg_1$ 
in~$\SSS$. The assumption that $\SSS$ is closed under  right-comultiple  
implies the existence of~$\hh_{1, 1}$ in~$\SSS$ and $\ff_{1, 1}, \gg_{1, 1}, 
\hh'_{1,1}$ satisfying $\ff_1 \gg_{1, 1} = \gg_1 \ff_{1, 1} = \hh_{1, 1}$, 
$\hh'_{0,1} = \ff_{1,1} \hh'_{1,1}$, and $\hh'_{1,0} = \gg_{1,1} \hh'_{1,1}$. 
Moreover, the assumption that $\SSS$ is closed under right-divisor implies 
that $\ff_{1,1}$ and~$\gg_{1,1}$, which right-divide~$\hh_{1,1}$, belong 
to~$\SSS$. 

Now we repeat the same argument with~$\hh'_{0, 1} \wdots \hh'_{0, \qq-1}$, 
then $\hh'_{1,0} \wdots \hh'_{1, \qq-1}$, and so on until~$\hh'_{\pp-1, 
\qq-1}$: starting from the vertex~$(\ii-1, \jj-1)$, by induction hypothesis, 
we have  $\ff_{\ii, \jj-1}\hh'_{\ii, \jj-1} = \gg_{\ii-1, \jj} \hh'_{\ii-1, 
\jj} = \hh'_{\ii-1, \jj-1}$ with $\ff_{\ii, \jj-1}$ and~$\gg_{\ii-1, \jj}$  
in~$\SSS$. As $\SSS$ is closed under right-comultiple and 
right-divisor, there 
exist~$\ff_{\ii, \jj}, \gg_{\ii, \jj}, \hh_{\ii, \jj}$ in~$\SSS$ and 
$\hh'_{\ii, \jj}$ in~$\CCC$ satisfying $\ff_{\ii, \jj-1} \gg_{\ii, \jj} = 
\gg_{\ii-1,\jj} \ff_{\ii, \jj} = \hh_{\ii, \jj}$, $\hh'_{\ii-1,\jj} = 
\ff_{\ii, \jj} \hh'_{\ii, \jj}$, and $\hh'_{\ii, \jj-1} = \gg_{\ii, \jj} 
\hh'_{\ii, \jj}$. 

At this point, we see that the equality $\ff_1 \pdots \ff_\pp = \gg_1 \pdots 
\gg_\qq$ is the consequence of $2\pp\qq$ relations of the form $\seqq\ff\gg = 
\seq\hh$ with $\ff, \gg, \hh$ in~$\SSS$, namely the relations $\seqq{\ff_{\ii, 
\jj-1}}{\gg_{\ii, \jj}} = \seq{\hh_{\ii, \jj}}$ and $\seqq{\gg_{\ii-1, 
\jj}}{\ff_{\ii, \jj}} = \seq{\hh_{\ii, \jj}}$, plus the $\pp + \qq$ relations 
$\seqq{\ff_{\ii, \qq}}{\hh'_{\ii, \qq}} = \seq{\hh'_{\ii-1, \qq}}$ and 
$\seqq{\gg_{\pp, \jj}}{\hh'_{\pp, \jj}} = \seq{\hh'_{\pp, \jj-1}}$. By 
construction, all elements $\hh'_{\ii, \qq}$ and $\hh'_{\pp, \jj}$ are 
invertible, and so are all $\ff_{\ii, \qq}$ and $\gg_{\pp, \jj}$. As $\SSS$ 
generates~$\CCC$, every element of~$\CCCi$ is a finite product of elements 
of~$\SSS \cap \CCCi$ and, therefore, every relation of the form 
$\seqq{\ee_1}{\ee_2} = \seq\ee$ holding in~$\CCC$ follows from finitely many 
relations of this form with $\ee_1, \ee_2, \ee$ in~$\SSS \cap \CCCi$. This 
completes the argument.
\end{proof}

\begin{figure}[htb]
\begin{picture}(75,50)(0,0)
\pcline{->}(1,48)(14,48)
\taput{$\gg_1$}
\pcline{->}(17,48)(28,48)
\taput{$\gg_2$}
\psline[style=etc](32,48)(43,48)
\pcline{->}(46,48)(59,48)
\taput{$\gg_\qq$}
\psline[doubleline=true](61,48)(68,48)
\psarc[doubleline=true](68,41){7}{0}{90}
\pcline[doubleline=true,fillcolor=white](75,41)(75,1)
\trput{$\id\ud$}

\pcline{->}(0,47)(0,37)
\tlput{$\ff_1$}
\pcline{->}(1,47)(14,37)
\put(7,43){$\hh_{1,1}$}
\pcline{->}(15,47)(15,37)
\trput{$\ff_{1,1}$}
\pcline{->}(16,47)(29,37)
\put(22,43){$\hh_{1,2}$}
\pcline{->}(30,47)(30,37)
\trput{$\ff_{1,2}$}
\pcline{->}(45,47)(45,37)
\put(45.5,39){$\ff_{1,\qq-1}$}
\pcline{->}(46,47)(59,37)
\put(52,43){$\hh_{1,\qq}$}
\pcline{->}(60,47)(60,37)
\trput{$\ff_{1,\qq}$}

\pcline{->}(1,36)(14,36)
\tbput{$\gg_{1,1}$}
\pcline{->}(16,36)(29,36)
\tbput{$\gg_{1,2}$}
\psline[style=etc](32,36)(43,36)
\pcline{->}(46,36)(59,36)
\tbput{$\gg_{1,\qq}$}
\pcline[style=exist](61,36)(67,36)
\tbput{$\quad\hh'_{1,\qq}$}
\psarc[style=exist](67,29){7}{0}{90}
\psline[style=exist]{->}(74,29)(74,2)

\psline[style=etc](0,34)(0,26)
\psline[style=etc](15,34)(15,26)
\psline[style=etc](30,34)(30,26)
\psline[style=etc](45,34)(45,26)
\psline[style=etc](60,34)(60,26)

\pcline{->}(1,24)(14,24)
\taput{$\gg_{\pp-1,1}$}
\pcline{->}(16,24)(29,24)
\taput{$\gg_{\pp-1,2}$}
\psline[style=etc](32,24)(43,24)
\pcline{->}(46,24)(59,24)
\taput{$\gg_{\pp-1,\qq}$}
\pcline[style=exist](61,24)(67,24)
\taput{$\quad\hh'_{\pp-1,\qq}$}
\psarc[style=exist](67,17){7}{0}{90}
\psline[style=exist]{->}(74,17)(74,2)

\pcline{->}(0,23)(0,13)
\tlput{$\ff_\pp$}
\pcline{->}(1,23)(14,13)
\put(7,19){$\hh_{\pp,1}$}
\pcline{->}(15,23)(15,13)
\trput{$\ff_{\pp,1}$}
\pcline{->}(16,23)(29,13)
\put(22,19){$\hh_{\pp,2}$}
\pcline{->}(30,23)(30,13)
\trput{$\ff_{\pp,2}$}
\pcline{->}(45,23)(45,13)
\put(45.5,15){$\ff_{\pp,\qq-1}$}
\pcline{->}(46,23)(59,13)
\put(52,19){$\hh_{\pp,\qq}$}
\pcline{->}(60,23)(60,13)
\trput{$\ff_{\pp,\qq}$}

\pcline{->}(1,12)(14,12)
\tbput{$\gg_{\pp,1}$}
\pcline{->}(16,12)(29,12)
\tbput{$\gg_{\pp,2}$}
\psline[style=etc](32,12)(43,12)
\pcline{->}(46,12)(59,12)
\tbput{$\gg_{\pp,\qq}$}
\pcline[style=exist]{->}(61,11)(74,1)
\put(64,10){$\hh'_{\pp,\qq}$}

\psline[doubleline=true](0,11)(0,7)
\psarc[doubleline=true](7,7){7}{180}{270}
\psline[doubleline=true,fillcolor=white](7,0)(74,0)
\put(-3,2){$\id\ud$}
\psline[style=exist](15,11)(15,7)
\psarc[style=exist](22,8){7}{180}{270}
\put(17,4){$\hh'_{\pp, 1}$}
\psline[style=exist](30,11)(30,7)
\psarc[style=exist](37,8){7}{180}{270}
\put(32,4){$\hh'_{\pp, 2}$}
\psline[style=exist](45,11)(45,7)
\psarc[style=exist](52,8){7}{180}{270}
\put(47,4){$\hh'_{\pp, \qq-1}$}
\psline[style=exist]{->}(22,1)(73,1)
\end{picture}
\caption[]{\sf\smaller Factorization of the equality $\ff_1 \pdots \ff_\pp = 
\gg_1 \pdots \gg_\qq$ in terms of relations $\ff \gg = \hh$ with $\ff, \gg, 
\hh$ in~$\SSS$.}
\label{F:Present}
\end{figure}

We can now complete the proof of Proposition~\ref{P:GarGerm}.

\begin{proof}[Proof of Proposition~\ref{P:GarGerm}]
First, by Proposition~\ref{P:RecGarII}, the family~$\SSSs$ is closed under 
right-comultiple. An easy direct verification shows that, in any case, a 
subfamily~$\SSS$ is closed under right-comultiple if and only if $\SSSs$ is. 
So, here, $\SSS$ is closed under right-comultiple. By assumption, it is full, 
so, by Lemma~\ref{L:Present}, $\CCC$ is presented by the relations 
$\seqq\ff\gg = \seq\hh$ with $\ff, \gg, \hh$ in~$\SSS$. This means that $\CCC$ 
admits the presentation $\PRESp\SSS{\RRR_{\OP}}$. So $\CCC$ is isomorphic 
to~$\Cat(\SSSg)$, which, by definition, admits that presentation.
\end{proof}

So, in the case of a full Garside family, the induced germ contains all 
information needed to determine the ambient category, and 
it is therefore natural to investigate the germs that occur in this way. 

%%%%

\subsection{The embedding problem}
\label{SS:Embed}

From now on, we start from an abstract germ~$\SSSg$, and investigate the 
properties of the category~$\Cat(\SSSg)$ and of (the image of)~$\SSS$ 
in~$\Cat(\SSSg)$. The first question is whether $\SSSg$ embeds 
in~$\Cat(\SSSg)$. 
This need not be the case in general, but we shall see that left-associativity 
is a sufficient condition.

\begin{nota}
\label{N:eq}
For $\SSSg$ a germ, we denote by~$\equiv$ the congruence on~$\SSS^*$ generated 
by the relations of~$\RRR_{\OP}$, and by~$\iota$ the prefunctor of~$\SSS$ 
to~$\Cat(\SSSg)$ that is the identity on~$\Obj(\SSS)$ and maps~$\gg$ to the 
$\equiv$-class of~$(\gg)$.
\end{nota}

So, by definition, $\equiv$ is the equivalence relation 
on~$\SSS^*$ generated by all pairs
\begin{equation}
\label{E:Pair}
( \seqqqqqq{\gg_1}\etc{\gg_\ii}{\gg_{\ii+1}}\etc{\gg_\qq} \ , \ 
\seqqqqq{\gg_1}\etc{\gg_\ii \OP \gg_{\ii+1}}\etc{\gg_\qq} ),
\end{equation}
that is, the pairs in which two adjacent entries are replaced with their 
$\OP$-product, assuming that the latter exists. The category~$\Cat(\SSSg)$ is 
then ~$\SSS^*\quot\equiv$.

\begin{prop}
\label{P:Embed}
If $\SSSg$ is a left-associative germ, the map~$\iota$ of Notation~\ref{N:eq} 
is injective and the product of~$\Cat(\SSSg)$ extends the image of~$\OP$ 
under~$\iota$. Moreover, $\iota\SSS$ is closed under right-divisor 
in~$\Cat(\SSSg)$.
\end{prop}

\begin{proof}
We inductively define a partial map~$\Pi$ from~$\SSS^*$ to~$\SSS$ by
\begin{equation}
\label{E:Embed}
\Pi(\ew_\xx) = \id\xx \ \mbox{and} \ \Pi(\seq\gg \pc \ww) = \gg \OP \Pi(\ww) 
\mbox{ if $\gg$ lies in~$\SSS$ and $\gg \OP \Pi(\ww)$ is defined}.
\end{equation}
%To make notation simpler, we shall write $\Pi\seqqq{\gg_1}\etc{\gg_\qq}$ for 
%$\Pi(\seqqq{\gg_1}\etc{\gg_\qq})$. 
We claim that $\Pi$ induces a well defined partial map from~$\Cat(\SSSg)$ 
to~$\SSS$, more precisely that, if $\ww, \ww'$ are $\equiv$-equivalent 
elements of~$\SSS^*$, then $\Pi(\ww)$ exists if and only if $\Pi(\ww')$ does, 
and in this case they are equal. To prove this, we may assume that $\Pi(\ww)$ 
or $\Pi(\ww')$ is defined and that $(\ww, \ww')$ is of the 
type~\eqref{E:Pair}. Let $\gg = \Pi(\seqqq{\gg_{\ii+2}}\etc{\gg_\pp})$. The 
assumption that $\Pi(\ww)$ or $\Pi(\ww')$ is defined implies that $\gg$ is 
defined. Then \eqref{E:Embed} gives $\Pi(\ww) = 
\Pi(\seqqqq{\gg_1}\etc{\gg_{\ii-1}}\hh)$ whenever $\Pi(\ww)$ is defined, and 
$\Pi(\ww') =\nobreak \Pi(\seqqqq{\gg_1}\etc{\gg_{\ii-1}}{\hh'})$ whenever 
$\Pi(\ww')$ is defined, with $\hh = \Pi(\seqqq{\gg_\ii}{\gg_{\ii+1}}\gg)$ and 
$\hh' = \Pi(\seqq{\gg_\ii \OP \gg_{\ii+1}}\gg)$, that is,
\begin{equation*}
\hh = \gg_\ii \OP (\gg_{\ii+1} \OP \gg) \quad \mbox{and} \quad \hh' = (\gg_\ii 
\OP \gg_{\ii+1}) \OP \gg.
\end{equation*}
So the point is to prove that $\hh$ is defined if and only if $\hh'$ is, in 
which case they are equal. Now, if $\hh$ is defined, the assumption that 
$\gg_\ii \OP \gg_{\ii+1}$ is defined plus~\eqref{E:Germ3} imply that $\hh'$ 
exists and equals~$\hh$. Conversely, if $\hh'$ is defined, \eqref{E:LeftAss} 
implies that $\gg_{\ii+1} \OP \gg$ is defined, and then \eqref{E:Germ3} 
implies that $\hh$ exists and equals~$\hh'$.

Assume now that $\gg, \gg'$ lie in~$\SSS$ and $\iota\gg = \iota\gg'$ holds, 
that is, the length one paths $\seq\gg$ and $\seq{\gg'}$ are 
$\equiv$-equivalent. The above claim gives $\gg = \Pi(\seq\gg) = 
\Pi(\seq{\gg'}) =\nobreak \gg'$, so $\iota$ is injective.

Next, assume that $\ff, \gg$ belong to~$\SSS$ and $\ff \OP \gg$ is defined. We 
have $\seqq\ff\gg \equiv \seq{\ff \OP \gg}$, which means that the product of 
$\iota\ff$ and~$\iota\gg$ in~$\Cat(\SSSg)$ is~$\iota(\ff \OP \gg)$.

Finally, assume that $\gg$ belongs to~$\SSS$ and $\iota\gg'$ is a 
right-divisor of~$\iota\gg$ in~$\Cat(\SSSg)$. This means that there exist 
elements $\ff_1 \wdots \ff_\pp, \gg_1 \wdots \gg_\qq$ of~$\SSS$ such that 
$\iota\gg$ is the $\equiv$-class 
of~$\seqqqqqq{\ff_1}\etc{\ff_\pp}{\gg_1}\etc{\gg_\qq}$ and $\iota\gg'$ is the 
$\equiv$-class of $\seqqq{\gg_1}\etc{\gg_\qq}$. By the claim above, the first 
relation implies that $\Pi(\seqqqqqq{\ff_1}\etc{\ff_\pp}{\gg_1}\etc{\gg_\qq})$ 
exists (and equals~$\gg$). By construction, this implies that 
$\Pi(\seqqq{\gg_1}\etc{\gg_\qq})$ exists as well, hence that $\gg'$ belongs 
to~$\SSS$. So $\iota\SSS$ is closed under right-divisor in~$\Cat(\SSSg)$.
\end{proof}

In the context of Proposition~\ref{P:Embed}, we shall from now on 
identify~$\SSS$ with its image in the category~$\Cat(\SSSg)$, that is, drop 
the 
canonical injection~$\iota$. Before going on, we establish a few consequences 
of the existence of the above function~$\Pi$. If $\SSSg$ is a germ, an 
element~$\ee$ of~$\SSS$ is naturally called \emph{invertible} if there 
exists~$\ee'$ in~$\SSS$ satisfying $\ee \OP \ee' = \id\xx$ and $\ee' \OP \ee = 
\id\yy$, with $\xx$ the source of~$\ee$ and $\yy$ its target. We denote 
by~$\SSSi$ the family of all invertible elements of~$\SSS$. Also we 
introduce a local version of left-divisibility (a symmetric notion of local 
right-divisibility will also be considered below).

\begin{defi}
\label{D:Cancel}
\ITEM1 Assume that $\SSSg$ is a germ. For $\ff, \gg$ in~$\SSS$, we say that 
$\ff \diveS \gg$ (\resp.\ $\ff \eqirS \gg$) holds if $\gg = \ff \OP \gg'$ 
holds for some~$\gg'$ in~$\SSS$ (\resp.\ in~$\SSSi$).

\ITEM2 A germ~$\SSSg$ is called \emph{left-cancellative} if there exist no 
triple $\ff, \gg, \gg'$ in~$\SSS$ satisfying $\gg \not= \gg'$ and $\ff \OP \gg 
= \ff \OP \gg'$.
\end{defi}

\begin{lemm}
\label{L:LocalDiv}
Assume that $\SSSg$ is a left-associative germ. 

\ITEM1 For~$\ff, \gg$ in~$\SSS$, the relation $\ff \dive \gg$ holds 
in~$\Cat(\SSSg)$ if and only if $\ff \diveS \gg$~holds.

\ITEM2 The relation~$\diveS$ is transitive. If $\hh \OP \ff$ and $\hh \OP \gg$ 
are defined, then $\ff \diveS \gg$ implies $\hh \OP \ff \diveS \hh \OP \gg$, 
and $\ff \eqirS \gg$ implies $\hh \OP \ff \eqirS \hh \OP \gg$.

\ITEM3 If, in addition, $\SSSg$ is left-cancellative, then $\ff \eqirS \gg$ is 
equivalent to the conjunction of~$\ff \diveS \gg$ and~$\gg \diveS \ff$ and, if 
$\hh \OP \ff$ and $\hh \OP \gg$ are defined, then $\ff \diveS \gg$ is 
equivalent to $\hh \OP \ff \diveS \hh \OP \gg$, and $\ff \eqirS \gg$ is 
equivalent to $\hh \OP \ff \eqirS \hh \OP \gg$.
\end{lemm}

\begin{proof}
\ITEM1 Assume $\ff, \gg \in \SSS$ and $\gg = \ff \gg'$ in~$\Cat(\SSSg)$. As 
$\gg$ lies in~$\SSS$ and $\SSS$ is closed under right-divisor 
in~$\Cat(\SSSg)$, 
the 
element~$\gg'$ lies in~$\SSS$. So we have $\seq\gg \equiv \seqq\ff{\gg'}$, 
whence, applying the function~$\Pi$ of~\eqref{E:Embed}, $\gg = \Pi(\seq\gg) = 
\Pi(\seqq\ff{\gg'}) = \ff \OP \gg'$. Therefore $\ff \diveS \gg$ is satisfied. 
The converse implication is straightforward.

\ITEM2 As the relation~$\dive$ in~$\Cat(\SSSg)$ is transitive, \ITEM1 implies 
that $\diveS$ is transitive as well. Next, assume that $\hh \OP \ff$ and $\hh 
\OP \gg$ are defined, and $\gg = \ff \OP \gg'$ holds. By~\eqref{E:Germ3}, we 
deduce $\hh \OP \gg = \hh \OP (\ff \OP \gg') = (\hh \OP \ff) \OP \gg'$, whence 
$\hh \OP \ff \diveS \hh \OP \gg$.  Considering the special case when $\gg'$ 
belongs to~$\SSSi$, we deduce that $\ff \eqirS \gg$ implies $\hh \OP \ff 
\eqirS \hh \OP \gg$.

\ITEM3 Assume $\ff = \gg \OP \ee$ and $\gg = \ff \OP \ee'$. We deduce $\ff = 
(\ff \OP \ee') \OP \ee$, whence
$\ff =\nobreak \ff \OP (\ee' \OP \ee)$ by left-associativity. By 
left-cancellativity, we deduce $\ee' \OP \ee = \id\yy$, where $\yy$ is the 
target of~$\ff$. So $\ee$ and $\ee'$ are invertible, and $\ff \eqirS \gg$ 
holds.

Assume now that $\hh \OP \ff$ and $\hh \OP \gg$ are defined and $\hh \OP \ff 
\diveS \hh \OP \gg$ holds. So we have $\hh \OP \gg = (\hh \OP \ff) \OP \gg'$ 
for some~$\gg'$. By left-associativity, $\ff \OP \gg'$ must be defined and we 
obtain $\hh \OP \gg = \hh \OP (\ff \OP \gg')$, whence $\gg = \ff \OP \gg'$ by 
left-cancellativity, and $\ff \diveS \gg$.

 Finally, $\ff \eqirS \gg$ implies $\hh \OP \ff \eqirS \hh \OP \gg$ by~\ITEM2. 
Conversely, $\hh \OP \ff \eqirS \hh \OP \gg$ implies both $\hh \OP \ff \diveS 
\hh \OP \gg$ and $\hh \OP \gg \diveS \hh \OP \ff$, hence $\ff \diveS \gg$ and 
$\gg \diveS \ff$ by the result above, whence in turn $\ff \eqirS \gg$. 
\end{proof}

%%%%

\subsection{Garside germs}
\label{SS:GarGerm}

Our goal will be to characterize those germs that give 
rise to Garside families. To state the results easily, we introduce a 
terminology.

\begin{defi}
\label{D:GarGerm}
A germ~$\SSSg$ is said to be a \emph{Garside germ} if there exists a 
left-cancellative category~$\CCC$ such that $\SSS$ is a full Garside family 
of~$\CCC$.
\end{defi}

By Proposition~\ref{P:GarGerm}, if $\SSS$ is a full Garside family 
in some 
left-cancellative category~$\CCC$, the latter must isomorphic 
to~$\Cat(\SSSg)$. 
So, a germ~$\SSSg$ is a Garside germ if and only if the category~$\Cat(\SSSg)$ 
is left-cancellative and $\SSS$ is a full Garside family in~$\Cat(\SSSg)$. In 
other words, in Definition~\ref{D:GarGerm}, we can assume that the 
category~$\CCC$ is~$\Cat(\SSSg)$.

We shall now state and  begin to establish  (the argument will be completed in 
Section~\ref{S:RecGerm} only) simple conditions that characterize 
Garside germs.

\begin{defi}
\label{D:Patch}
Assume that $\SSSg$ is a germ. For~$\seqq{\gg_1}{\gg_2}$ in~$\Seq\SSS2$, we put
\begin{gather}
\label{E:I}
\IIIS(\gg_1, \gg_2) = \{\hh \in \SSS \mid \exists \gg \in \SSS \ (\hh = 
\gg_1 \OP \gg \mbox{ and } \gg \diveS \gg_2 ) \},\\
\label{E:J}
\JJJS(\gg_1, \gg_2) = \{\gg \in \SSS \mid \gg_1 \OP \gg \mbox{ is 
defined and } \gg \diveS \gg_2 \}.
\end{gather}
A map from~$\Seq\SSS2$ to~$\SSS$ is called an \emph{$\III$-function} (\resp.\ 
a 
\emph{$\JJJ$-function}) for~$\SSSg$ if, for 
every~$\seqq{\gg_1}{\gg_2}$ in~$\Seq\SSS2$, 
the value at~$(\gg_1, \gg_2)$ lies in~$\IIIS(\gg_1, \gg_2)$ (\resp.\ 
in~$\JJJS(\gg_1, \gg_2)$).
\end{defi}

We recall that, in~\eqref{E:I}, writing $\hh = \gg_1 \OP \gg$ implies that 
$\gg_1 \OP \gg$ is defined. An element of~$\JJJS(\gg_1, \gg_2)$ is a 
fragment of~$\gg_2$ that can be added to~$\gg_1$ legally, that is, without 
going out of~$\SSS$.  Note that $\gg_1$ always belongs to~$\IIIS(\gg_1, 
\gg_2)$ and that, if $\yy$ is the target of~$\gg_1$, then  $\id\yy$ always 
belongs to~$\JJJS(\gg_1, \gg_2)$. So, in particular, $\IIIS(\gg_1, 
\gg_2)$ and~$\JJJS(\gg_1, \gg_2)$ are never empty. 

The connection between~$\IIIS(\gg_1, \gg_2)$ and~$\JJJS(\gg_1, 
\gg_2)$ is clear: with obvious notation, we have $\IIIS(\gg_1, \gg_2) = 
\gg_1 \OP \JJJS(\gg_1, \gg_2)$. However, it turns out that, depending on 
the situation, using $\III$ or~$\JJJ$ is more convenient, and it is 
useful to 
introduce both notions. Note that, if $\hh$ belongs to~$\IIIS(\gg_1, 
\gg_2)$ and $\SSSg$ embeds in~$\Cat(\SSSg)$, then, in~$\Cat(\SSSg)$, we have 
$\hh 
\in \SSS$ and  $\gg_1 \dive \hh \dive \gg_1 \gg_2$.  However, the latter 
relations need not imply $\hh \in \IIIS(\gg_1, \gg_2)$  \emph{a priori} 
since $\gg_1 \gg \dive \gg_1 \gg_2$ is not known to imply $\gg \dive \gg_2$ as 
long as $\Cat(\SSSg)$ has not been proved to be left-cancellative.  

We shall characterize Garside germs by the existence of $\III$- or 
$\JJJ$-functions satisfying some algebraic laws reminiscent of the $\HHH$-law 
of~\eqref{E:HLaw}.

\begin{defi}
\label{D:IJLaws}
If $\SSSg$ is a germ, a map~$\FF$ from~$\Seq\SSS2$ to~$\SSS$ is said to obey 
the \emph{$\III$-law} if, for every $\seqqq{\gg_1}{\gg_2}{\gg_3}$ 
in~$\Seq\SSS3$ such that $\gg_1 \OP \gg_2$ is defined, we have 
\begin{equation}
\label{E:ILaw}
\FF(\gg_1, \FF(\gg_2, \gg_3)) \eqir \FF(\gg_1 \OP \gg_2, \gg_3).
\end{equation}
The map $\FF$ is said to obey the \emph{$\JJJ$-law} if, under the same 
hypotheses, we have 
\begin{equation}
\label{E:JLaw}
\FF(\gg_1, \gg_2 \OP \FF(\gg_2, \gg_3)) \eqir \gg_2 \OP \FF(\gg_1 \OP 
\gg_2, \gg_3). 
\end{equation}
If, in \eqref{E:ILaw} or \eqref{E:JLaw}, $=$ replaces~$\eqir$, we speak of the 
\emph{sharp} $\III$- or $\JJJ$-law. 
\end{defi}

The result we shall prove below is as follows.

\begin{prop}
\label{P:RecGarGermI}
A germ~$\SSSg$ is a Garside germ if and only if it satisfies any one of the 
following equivalent conditions:
\begin{gather}
\label{E:RecGarGermI1}
\BOX{$\SSSg$ is left-associative, left-cancellative, and it admits an 
$\III$-function that satisfies the sharp $\III$-law;}\\
\label{E:RecGarGermI2}
\BOX{$\SSSg$ is left-associative, left-cancellative, and it admits a 
$\JJJ$-function that satisfies the sharp $\JJJ$-law.}
\end{gather}
\end{prop}

Note that all conditions in Proposition~\ref{P:RecGarGermI} are local in that 
they only involve the elements of~$\SSS$ and computations taking place 
inside~$\SSS$. In particular, if $\SSS$ is finite, the conditions are 
effectively checkable in finite time.

 In this section, we prove only that the conditions are necessary. The 
converse is deferred to Section~\ref{S:RecGerm} below.

\begin{lemm}
\label{L:EquivIJ}
Assume that $\SSSg$ is a germ that is left-associative and left-cancellative 
and $\II, \JJ : \Seq\SSS2 \to \SSS$ are connected by $\II(\gg_1, \gg_2) = 
\gg_1 \OP \JJ(\gg_1, \gg_2)$ for all~$\gg_1, \gg_2$. 

\ITEM1 The map~$\II$ is an $\III$-function for~$\SSSg$ if and only if 
$\JJ$ is a $\JJJ$-function for~$\SSSg$.

\ITEM2 In the situation of~\ITEM1, the map~$\II$ obeys the $\III$-law (\resp.\ 
the sharp $\III$-law) if and only if $\JJ$ obeys the $\JJJ$-law (\resp.\ the 
sharp $\JJJ$-law).
\end{lemm}

\begin{proof}
\ITEM1 First, the assumption that $\SSSg$ is left-cancellative implies 
that, for every~$\II$, there exists at most one associated~$\JJ$. Then the 
definitions of an $\III$- and a $\JJJ$-functions for~$\SSSg$ directly give the 
expected equivalence.

\ITEM2 Assume that $\II$ is an $\III$-function obeying the $\III$-law.  
Assume 
that~$\seqqq{\gg_1}{\gg_2}{\gg_3}$ lies in~$\Seq\SSS3$ and $\gg_1 \OP \gg_2$ 
is defined. By assumption, we have $\II(\gg_1, \II(\gg_2, \gg_3)) \eqir 
\II(\gg_1 \OP \gg_2, \gg_3)$, which translates into
\begin{equation}
\label{E:EquivIJ1}
\gg_1 \OP \JJ(\gg_1, \gg_2 \OP \JJ(\gg_2, \gg_3)) \eqir (\gg_1 \OP \gg_2) \OP 
\JJ(\gg_1 \OP \gg_2, \gg_3).
\end{equation}
Then the assumption that $\SSSg$ is left-associative implies that $\gg_2 \OP 
\JJ(\gg_1 \OP \gg_2, \gg_3)$ is defined and, therefore, \eqref{E:EquivIJ1} 
implies
\begin{equation}
\label{E:EquivIJ2}
\gg_1 \OP \JJ(\gg_1, \gg_2 \OP \JJ(\gg_2, \gg_3)) \eqir \gg_1 \OP (\gg_2 \OP 
\JJ(\gg_1 \OP \gg_2, \gg_3)).
\end{equation}
Finally,  by Lemma~\ref{L:LocalDiv}, we may left-cancel~$\gg_1$ 
in~\eqref{E:EquivIJ2}, and what remains is the expected instance of the 
$\JJJ$-law. So $\JJ$ obeys the $\JJJ$-law. 

The argument in the case when $\II$ obeys the sharp $\III$-law is similar: now 
\eqref{E:EquivIJ1} and \eqref{E:EquivIJ2} are equalities, and, applying the 
assumption that $\SSSg$ is left-cancellative, we directly deduce the expected 
instance of the sharp $\JJJ$-law. So $\JJ$ obeys the sharp $\JJJ$-law. 

Conversely, assume that $\JJ$ is a $\JJJ$-function for~$\SSSg$ that satisfies 
the $\JJJ$-law. Assume that~$\seqqq{\gg_1}{\gg_2}{\gg_3}$ lies in~$\Seq\SSS3$ 
and $\gg_1 \OP \gg_2$ is defined. By the $\JJJ$-law, we have $\JJ(\gg_1, \gg_2 
\OP \JJ(\gg_2, \gg_3)) \eqir \gg_2 \OP \JJ(\gg_1 \OP \gg_2, \gg_3)$. By 
definition of a $\JJJ$-function, the expression $\gg_1 \OP \JJ(\gg_1, \gg_2 
\OP \JJ(\gg_2, \gg_3))$ is defined, hence so is $\gg_1 \OP (\gg_2 \OP 
\JJ(\gg_1 \OP \gg_2, \gg_3))$, and we obtain~\eqref{E:EquivIJ2}. 
Applying~\eqref{E:Germ3}, which is legal as $\gg_1 \OP \gg_2$ is defined, we 
deduce~\eqref{E:EquivIJ1}, whence $\II(\gg_1, \II(\gg_2, \gg_3)) \eqir 
\II(\gg_1 \OP \gg_2, \gg_3)$, the expected instance of the $\III$-law.  So 
$\II$ obeys the $\III$-law.

Finally, if $\JJ$ obeys the sharp $\JJJ$-law, the argument is similar: 
\eqref{E:EquivIJ1} and \eqref{E:EquivIJ2} are equalities, and one obtains the 
expected instance of the sharp $\III$-law. So $\II$ obeys the sharp $\III$-law 
in this case. 
\end{proof}

The (sharp) $\III$- and $\JJJ$-laws are closely connected with the $\HHH$-law, 
which implies that the conditions of Proposition~\ref{P:RecGarGermI} are 
necessary. 

\begin{lemm}
\label{L:Necessary}
Every Garside germ satisfies~\eqref{E:RecGarGermI1} and~\eqref{E:RecGarGermI2}.
\end{lemm}

\begin{proof}
Assume that $\SSSg$ is a Garside germ. Let $\CCC = \Cat(\SSSg)$. By 
definition, 
$\SSS$ embeds in~$\CCC$, so, by Proposition~\ref{P:Embed}, $\SSSg$ must be 
left-associative. Next, by definition again, $\CCC$ is left-cancellative and, 
therefore, $\SSSg$ must be left-cancellative as, if $\ff, \gg, \gg'$ lie 
in~$\SSS$ and satisfy $\ff \OP \gg = \ff \OP \gg'$, then $\ff \gg = \ff \gg'$ 
holds in~$\CCC$, implying $\gg = \gg'$. 

Next, as $\SSS$ is a Garside family in~$\CCC$, it 
satisfies~\eqref{E:RecGarIII2}, so there exists~$\HH$ defined on~$\SSS\CCC$, 
hence in particular on~$\SSS^2$, satisfying the sharp $\HHH$-law. Then we 
define $\II : \Seq\SSS2 \to \SSS$ by $\II(\gg_1, \gg_2) = \HH(\gg_1 \gg_2)$.
 
First we claim that $\II$ is an $\III$-function for~$\SSSg$. Indeed, assume 
$\seqq{\gg_1}{\gg_2} \in \Seq\SSS2$. By definition, we have $\gg_1 \eqir 
\HH(\gg_1) \dive \HH(\gg_1 \gg_2) \dive \gg_1\gg_2$ in~$\CCC$, hence 
$\HH(\gg_1 \gg_2) = \gg_1 
\gg$ for some~$\gg$ satisfying $\gg \dive \gg_2$. As $\gg$ 
right-divides~$\HH(\gg_1 \gg_2)$, which lies in~$\SSS$, and, by 
Proposition~\ref{P:Embed}, $\SSS$ is closed under right-divisor, $\gg$ 
must lie in~$\SSS$. By Lemma~\ref{L:LocalDiv}, it follows that 
$\HH(\gg_1 \gg_2)$ lies in~$\IIIS(\gg_1, \gg_2)$.

Next, assume that $\seqqq{\gg_1}{\gg_2}{\gg_3}$ lies in~$\Seq\SSS3$ 
and $\gg = \gg_1 \OP \gg_2$ holds. Then the sharp $\HHH$-law gives 
$\HH(\gg_1\HH(\gg_2 
\gg_3)) = \HH(\gg_1(\gg_2 \gg_3)) = \HH(\gg \gg_3)$. This directly translates 
into $\II(\gg_1, \II(\gg_2, \gg_3)) = \II(\gg_1 \OP \gg_2, \gg_3)$, 
the 
expected instance of the sharp $\III$-law. So \eqref{E:RecGarGermI1} is 
satisfied. 

Finally, by Lemma~\ref{L:EquivIJ}, \eqref{E:RecGarGermI2} is satisfied as well.
\end{proof}

%%%%%%%%
\section{Recognizing Garside germs}
\label{S:RecGerm}

We shall now establish two intrinsic characterizations of Garside germs, 
beginning with the one already stated in Proposition~\ref{P:RecGarGermI}.

%%%%
\subsection{Using the $\JJJ$-law}
\label{SS:Rec}

The principle for establishing that the conditions of 
Proposition~\ref{P:RecGarGermI} imply that $\SSSg$ is a Garside germ is 
obvious, namely using the given $\III$- or $\JJJ$-function to construct a head 
function on~$\SSS^2$ and then using Proposition~\ref{P:RecGarIII}. However, 
the argument is more delicate, because we do not know \emph{a priori} that the 
category~$\Cat(\SSSg)$ is left-cancellative and, therefore, eligible for 
Proposition~\ref{P:RecGarIII}. So what we shall do is simultaneously 
constructing the head function and proving left-cancellativity. The main point 
for that is to be able to control not only the head~$\HH(\gg)$ of an 
element~$\gg$, but also its tail, defined to be the element~$\gg'$ satisfying 
$\gg = \HH(\gg) \gg'$. To perform the construction, using a $\JJJ$-function is 
more convenient than using an $\III$-function. Here is the key technical 
result. 
%To make notation simpler, we shall abuse notation below and use mixed  
%expressions like $\gg\pc\ww$ for $\seq\gg \pc \ww$ for $\gg$ an element  
%of~$\XXX$ and $\ww$ an $\XXX$-path.

\begin{lemm}
\label{L:Main}
Assume that $\SSSg$ is a left-associative, left-cancellative germ and $\JJ$ is 
a $\JJJ$-function for~$\SSSg$ that satisfies the sharp $\JJJ$-law. Define 
functions 
$$\KK: \Seq\SSS2 \to \SSS, \qquad \HHu : \SSS^* \to \SSS, \qquad \TTu : 
\SSS^* \to \SSS^*$$ 
by $\gg_2 = \JJ(\gg_1, \gg_2) \OP \KK(\gg_1, \gg_2)$, $\HHu(\ew_\xx) = 
\id\xx$, $\TTu(\ew_\xx) = \ew_\xx$ and, for $\gg$ in~$\SSS$ and $\ww$ 
in~$\SSS^*$,
\begin{equation}
\HHu(\gg\pc\ww) = \gg \OP \JJ(\gg, \HHu(\ww)) \quad\mbox{and}\quad 
\TTu(\gg\pc\ww) = \KK(\gg, \HHu(\ww)) \pc \TTu(\ww).
\end{equation}
Then, for each~$\ww$ in~$\SSS^*$, we have
\begin{equation}
\label{E:Decomp}
\ww \equiv \HHu(\ww) \pc \TTu(\ww),
\end{equation}
and $\ww \equiv \ww'$ implies $\HHu(\ww) = \HHu(\ww')$ and $\TTu(\ww) 
\equiv \TTu(\ww')$.
\end{lemm}

\begin{proof}
First, the definition of~$\KK$ makes sense and is unambiguous. Indeed, by 
definition, $\JJ(\gg_1, \gg_2) \diveS \gg_2$ holds, so there exists~$\gg$ 
in~$\SSS$ satisfying $\gg_2 = \JJ(\gg_1, \gg_2) \OP \gg$. Moreover, as $\SSSg$ 
 
is left-cancellative, the element~$\gg$ is unique.

As for proving~\eqref{E:Decomp}, we use induction on the length of~$\ww$. For 
$\ww = 
\ew_\xx$, we have $\ew_\xx \equiv \seq{\id\xx} \pc \ew_\xx$. Otherwise, for 
$\gg$ in~$\SSS$ and $\ww$ in~$\SSS^*$, we find
\begin{align*}
\gg \pc \ww 
&\equiv \gg \pc \HHu(\ww) \pc \TTu(\ww) 
&& \mbox{by induction hypothesis,}\\
&\equiv \gg \pc \JJ(\gg, \HHu(\ww)) \OP \KK(\gg, \HHu(\ww)) \pc 
\TTu(\ww)
&& \mbox{by definition of~$\KK$,}\\
&\equiv \gg \pc \JJ(\gg, \HHu(\ww)) \pc \KK(\gg, \HHu(\ww)) \pc 
\TTu(\ww)
&& \mbox{by definition of $\equiv$,}\\
&\equiv \gg \OP \JJ(\gg, \HHu(\ww)) \pc \KK(\gg, \HHu(\ww)) \pc 
\TTu(\ww)
&& \mbox{as $\JJ(\gg, \HHu(\ww))$ lies in~$\JJJS(\gg, 
\HHu(\ww))$,}\\
&= \HHu(\gg \pc \ww) \pc \TTu(\gg \pc \ww)
&& \mbox{by definition of $\HHu$ and~$\TTu$.}
\end{align*}

As for the compatibility of~$\HHu$ and~$\TTu$ with respect to~$\equiv$, 
owing to the inductive definitions of~$\equiv$, $\HHu$ and~$\TTu$, it is 
sufficient to establish the relations
\begin{equation}
\label{E:Comp}
\HHu(\gg \pc \ww) = \HHu(\gg_1 \pc \gg_2 \pc \ww) 
\quad\mbox{and}\quad
\TTu(\gg \pc \ww) \equiv \TTu(\gg_1 \pc \gg_2 \pc \ww)
\end{equation}
for $\gg_1 \OP \gg_2 = \gg$ and $\ww$ in~$\SSS^*$ such that $\gg \pc \ww$ is 
a path (see Figure~\ref{F:Main}). 
Now, applying  the sharp form of~\eqref{E:JLaw}  with $\gg_3 = \HHu(\ww)$ 
and then using the 
definition of~$\HHu(\gg_2 \pc \ww)$, we obtain 
\begin{equation}
\label{E:Comp2}
\gg_2 \OP \JJ(\gg, \HHu(\ww)) = \JJ(\gg_1, \HHu(\gg_2 \pc \ww)).
\end{equation}
Then the first relation of~\eqref{E:Comp} is satisfied since we can write 
\begin{align*}
\HHu(\gg \pc \ww)
&= (\gg_1 \OP \gg_2) \OP \JJ(\gg, \HHu(\ww))
&& \mbox{by definition of $\HHu$,}\\
&= \gg_1 \OP (\gg_2 \OP \JJ(\gg, \HHu(\ww)))
&& \mbox{by  left-associativity, }\\
&= \gg_1 \OP \JJ(\gg_1, \HHu(\gg_2 \pc \ww)) = \HHu(\gg_1 \pc \gg_2 
\pc \ww)
&& \mbox{by \eqref{E:Comp2} and the definition of $\HHu$}.
\end{align*}
 We turn to the second relation in~\eqref{E:Comp}. Applying the definition 
of~$\HHu$, we first find
\begin{align*}
\gg_2 \OP \JJ(\gg_2, \HHu(\ww))
&=\HHu(\gg_2 \pc \ww)\\
&= \JJ(\gg_1, \HHu(\gg_2 \pc \ww)) \OP \KK(\gg_1, \HHu(\gg_2 \pc \ww))
&& \mbox{by definition of $\KK$,}\\
&= (\gg_2 \OP \JJ(\gg, \HHu(\ww))) \OP \KK(\gg_1, \HHu(\gg_2 \pc \ww))
&& \mbox{by \eqref{E:Comp2},}\\
&= \gg_2 \OP (\JJ(\gg, \HHu(\ww)) \OP \KK(\gg_1, \HHu(\gg_2 \pc \ww)))
&& \mbox{by  left-associativity, }
\end{align*}
whence, as $\SSSg$ is a left-cancellative germ,
\begin{equation}
\label{E:Comp3}
\JJ(\gg_2, \HHu(\ww)) = \JJ(\gg, \HHu(\ww)) \OP \KK(\gg_1, \HHu(\gg_2 
\pc \ww)).
\end{equation}
We deduce
\begin{align*}
\JJ(\gg, \HHu(\ww)) &\OP \KK(\gg, \HHu(\ww))
=\HHu(\ww)
&& \mbox{by definition of $\KK$,}\\
&=\JJ(\gg_2, \HHu(\ww)) \OP \KK(\gg_2, \HHu(\ww))
&& \mbox{by definition of $\KK$,}\\
&=(\JJ(\gg, \HHu(\ww)) \OP \KK(\gg_1, \HHu(\gg_2 \pc \ww))) \OP 
\KK(\gg_2, \HHu(\ww))
&& \mbox{by \eqref{E:Comp3},}\\
&=\JJ(\gg, \HHu(\ww)) \OP (\KK(\gg_1, \HHu(\gg_2 \pc \ww)) \OP 
\KK(\gg_2, \HHu(\ww)))
&& \mbox{by  left-associativity.}
\end{align*}
As $\SSSg$ is a left-cancellative germ, we may left-cancel~$\JJ(\gg, 
\HHu(\ww))$, and we obtain
$$\KK(\gg, \HHu(\ww)) = \KK(\gg_1, \HHu(\gg_2 \pc \ww)) \OP 
\KK(\gg_2, 
\HHu(\ww)),$$
whence $\KK(\gg, \HHu(\ww)) \pc \TTu(\ww) \equiv \KK(\gg_1, \HHu(\gg_2 
\pc 
\ww)) \pc \KK(\gg_2, \HHu(\ww)) \pc \TTu(\ww)$. Owing to the definition 
of~$\TTu$, this is exactly the second relation 
in~\eqref{E:Comp}.
\end{proof}

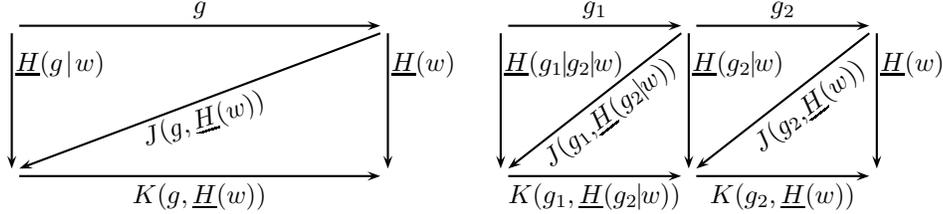
\begin{figure}[htb]
\begin{picture}(120,26)(0,-4)
\pcline{->}(49,19)(1,1)
\put(17,4){\rotatebox{18}{\hbox{$\JJ(\gg, \HHu(\ww))$}}}
\pcline{->}(0,19)(0,1)
\put(0.5,14){$\HHu(\gg \pc \ww)$}
\pcline{->}(1,0)(49,0)
\tbput{$\KK(\gg, \HHu(\ww))$}
\pcline{->}(50,19)(50,1)
\put(50.5,14){$\HHu(\ww)$}
\pcline{->}(1,20)(49,20)
\taput{$\gg$}
\pcline{->}(66,0)(89,0)
\tbput{$\KK(\gg_1, \HHu(\gg_2 \!\pc\! \ww))$}
\pcline{->}(91,0)(114,0)
\tbput{$\KK(\gg_2, \HHu(\ww))$}
\pcline{->}(66,20)(89,20)
\taput{$\gg_1$}
\pcline{->}(91,20)(114,20)
\taput{$\gg_2$}
\pcline{->}(89,19)(66,1)
\put(70,1){\rotatebox{38}{\hbox{$\JJ(\gg_1,\! \HHu(\gg_2 \!\pc\! 
\ww))$}}}
\pcline{->}(114,19)(91,1)
\put(98,3){\rotatebox{38}{\hbox{$\JJ(\gg_2,\! \HHu(\ww))$}}}
\pcline{->}(65,19)(65,1)
\put(65.5,14){$\HHu(\gg_1 \!\pc\! \gg_2 \!\pc\! \ww)$}
\pcline{->}(90,19)(90,1)
\put(90.5,14){$\HHu(\gg_2 \!\pc\! \ww)$}
\pcline{->}(115,19)(115,1)
\put(115.5,14){$\HHu(\ww)$}
\end{picture}
\caption[]{\sf\smaller Proof of Lemma~\ref{L:Main}: attention! as long as 
the ambient category is not proved to be left-cancellative, the above diagrams 
should be taken with care.}
\label{F:Main}
\end{figure}

We can now complete the argument easily.

\begin{proof}[Proof of Proposition~\ref{P:RecGarGermI}]
 Owing to Lemmas~\ref{L:Necessary} and~\ref{L:EquivIJ}, it suffices to prove 
now that \eqref{E:RecGarGermI2} implies that $\SSSg$ is a Garside germ.

So assume that $\SSSg$ is a germ that is left-associative and 
left-cancellative, and $\JJ$ is a $\JJJ$-function on~$\SSSg$ that satisfies 
the 
sharp $\JJJ$-law. Let $\CCC = \Cat(\SSSg)$. As $\SSSg$ is left-associative, 
Proposition~\ref{P:Embed} implies that $\SSS$ embeds in~$\CCC$ and is a full 
family in~$\CCC$.

Now we appeal to the functions~$\KK$, $\HHu$, and~$\TTu$ of 
Lemma~\ref{L:Main}. First, for each~$\ww$ in~$\SSS^*$ and each~$\gg$ in~$\SSS$ 
such that the target of~$\gg$ is the source of~$\ww$, we have
\begin{equation}
\label{E:Axiom1}
\ww \equiv \JJ(\gg, \HHu(\ww)) \pc \TTu(\gg \pc \ww):
\end{equation}
indeed, we have $\gg \OP \JJ(\gg, \HHu(\ww)) = \HHu(\gg \pc \ww)$ 
and 
\begin{align*}
\ww 
&\equiv \HHu(\ww) \pc \TTu(\ww)
&& \mbox{by \eqref{E:Decomp},}\\
&= \JJ(\gg, \HHu(\ww)) \OP \KK(\gg, \HHu(\ww)) \pc \TTu(\ww)
&& \mbox{by definition of~$\KK$,}\\
&\equiv \JJ(\gg, \HHu(\ww)) \pc \KK(\gg, \HHu(\ww)) \pc \TTu(\ww)
&& \mbox{by definition of $\equiv$,}\\
&= \JJ(\gg, \HHu(\ww)) \pc \TTu(\gg \pc \ww).
&& \mbox{by definition of $\TTu(\gg \pc \ww)$.}
\end{align*}
Assume now $\gg \pc \ww \equiv \gg \pc \ww'$. First, Lemma~\ref{L:Main} 
implies $\HHu(\gg \pc \ww) = \HHu(\gg \pc \ww')$, that is, $\gg \OP 
\JJ(\gg, \HHu(\ww)) = \gg \OP \JJ(\gg, \HHu(\ww'))$ owing to the 
definition of~$\HHu$. As $\SSSg$ is a left-cancellative germ, we may 
left-cancel~$\gg$ and we deduce $\JJ(\gg, 
\HHu(\ww)) = \JJ(\gg, \HHu(\ww'))$. Then, applying~\eqref{E:Axiom1} 
twice and Lemma~\ref{L:Main} again, we find
$$\ww \equiv \JJ(\gg, \HHu(\ww)) \pc \TTu(\gg \pc \ww) \equiv \JJ(\gg, 
\HHu(\ww')) \pc \TTu(\gg \pc \ww') \equiv \ww',$$
which implies that $\CCC$ is left-cancellative. 

Next, Lemma~\ref{L:Main} shows that the function~$\HHu$ induces a well 
defined function of~$\CCC$ to~$\SSS$, say~$\HH$. Then \eqref{E:Decomp} implies 
that $\HH(\gg) \dive \gg$ holds for every~$\gg$ in~$\CCC$. On the other hand, 
assume that $\hh$ belongs to~$\SSS$, and that $\hh \dive \gg$ holds in~$\CCC$. 
This means that there exists~$\ww$ in~$\SSS^*$ such that $\hh \pc \ww$ 
represents~$\gg$. By construction, we have $\HHu(\hh \pc \ww) = \hh \OP 
\JJ(\hh, \HHu(\ww))$, which implies $\hh \dive \HH(\gg)$ in~$\CCC$. So 
$\HH(\gg)$ is an $\SSS$-head of~$\gg$ and, 
therefore, every element of~$\CCC$ admits an $\SSS$-head. 

Finally, by Proposition~\ref{P:Embed}, $\SSS$ is closed under right-divisor 
in~$\CCC$, which implies that $\SSSs$ is also closed under right-divisor: 
indeed, a right-divisor of an element of~$\CCCi$ must lie in~$\CCCi$, and, if 
$\ff$ right-divides~$\gg\ee$ with $\gg \in \SSS$ and $\ee \in \CCCi$, then 
$\ff\ee\inv$ right-divides~$\gg$, hence it belongs to~$\SSS$, and therefore 
$\ff$ belongs to~$\SSS\CCCi$, hence to~$\SSSs$. Therefore,  $\SSS$ 
satisfies~\eqref{E:RecGarII1} in~$\CCC$ hence, by 
Proposition~\ref{P:RecGarII}, it is a Garside family  in~$\CCC$, which, we 
recall, is~$\Cat(\SSSg)$. Hence $\SSSg$ is  a Garside germ.
\end{proof}

%%%%
\subsection{Maximum $\III$-functions}
\label{SS:Max}

Continuing the investigation of Garside germs, we establish alternative 
characterizations of the latter involving maximality conditions. A point 
of interest is that such characterizations are automatically satisfied in 
convenient Noetherian contexts, leading to simplified versions of the criteria 
similar to  the results of Subsection~\ref{SS:Special}.

In Proposition~\ref{P:RecGarGermI}, we characterized Garside germs by the 
existence of an $\III$- or a $\JJJ$-function that satisfies the (sharp) 
$\III$-law or $\JJJ$-law. We now consider $\III$- or $\JJJ$-functions that 
satisfy maximality conditions.

\begin{defi}
\label{D:SUMaximum}
An $\III$-function (\resp.\ a $\JJJ$-function)~$\FF$ is called \emph{maximum} 
if, for all~$\gg_1, \gg_2$, every element~$\hh$ of~$\IIIS(\gg_1, \gg_2)$ 
(\resp.\ of $\JJJS(\gg_1, \gg_2)$) satisfies $\hh \diveS \FF(\gg_1, 
\gg_2)$.
\end{defi}

\begin{prop}
\label{P:RecGarGermII}
 A germ~$\SSSg$ is a Garside germ if and only if it satisfies any one of the 
following equivalent conditions: 
\begin{gather}
\label{E:RecGarGermII1Bis}
\BOX{$\SSSg$ is left-associative, left-cancellative, and, for 
every~$\seqq{\gg_1}{\gg_2}$ in~$\Seq\SSS2$, the family~$\IIIS(\gg_1, \gg_2)$ 
admits a $\diveS$-greatest element;}\\
\label{E:RecGarGermII1}
\BOX{$\SSSg$ is left-associative, left-cancellative, and admits a maximum 
$\III$-function;}\\
\label{E:RecGarGermII2Bis}
\BOX{$\SSSg$ is left-associative, left-cancellative, and for 
every~$\seqq{\gg_1}{\gg_2}$ in~$\Seq\SSS2$, the family~$\JJJS(\gg_1, \gg_2)$ 
admits a $\diveS$-greatest element.}\\
\label{E:RecGarGermII2}
\BOX{$\SSSg$ is left-associative, left-cancellative, and admits a maximum 
$\JJJ$-function.}
\end{gather}
\end{prop}

The next lemma establishes the equivalence of the conditions involving maximum  functions and greatest elements and proves that the conditions of Proposition~\ref{P:RecGarGermII} are satisfied in every Garside germ (this is easy).

\begin{lemm}
\label{L:Necessary2}
\ITEM1 For every germ, \eqref{E:RecGarGermII1Bis} is equivalent 
to~\eqref{E:RecGarGermII1}, and \eqref{E:RecGarGermII2Bis} is equivalent 
to~\eqref{E:RecGarGermII2}.

\ITEM2 Every Garside germ 
satisfies~\eqref{E:RecGarGermII1Bis}--\eqref{E:RecGarGermII2}.
\end{lemm}

\begin{proof}
\ITEM1 By definition, an $\III$-function~$\FF$ on~$\Seq\SSS2$ is maximum 
if and only if, for every~$\seqq{\gg_1}{\gg_2}$ in~$\Seq\SSS2$, the 
value~$\FF(\gg_1, \gg_2)$ is a $\diveS$-greatest element in~$\IIIS(\gg_1, 
\gg_2)$. So \eqref{E:RecGarGermII1} directly 
implies~\eqref{E:RecGarGermII1Bis}. Conversely, if \eqref{E:RecGarGermII1Bis} 
is satisfied, we obtain a maximum $\III$-function by picking, possibly using 
the Axiom of Choice, a $\diveS$-greatest element in~$\IIIS(\gg_1, \gg_2)$ for 
each~$\seqq{\gg_1}{\gg_2}$. So \eqref{E:RecGarGermII1Bis} 
implies~\eqref{E:RecGarGermII1}. The argument is similar 
for~\eqref{E:RecGarGermII2Bis} and~\eqref{E:RecGarGermII2} mutatis mutandis.

\ITEM2 Assume that $\SSSg$ is a Garside germ. By Lemma~\ref{L:Necessary}, 
$\SSSg$ is 
left-associative and left-cancellative. For every~$\gg$ in~$\Cat(\SSSg)$, let 
$\HH(\gg)$ be an $\SSS$-head of~$\gg$: by Proposition~\ref{P:RecGarII}, such a head always exists as $\SSS$ is 
full in~$\Cat(\SSSg)$ and includes~$\Id\SSS$. Now define $\II, \JJ : \Seq\SSS2 
\to \SSS$ by $\II(\gg_1, \gg_2) = \gg_1 \OP \JJ(\gg_1, \gg_2) = \HH(\gg_1 
\gg_2)$. Then, as in the proof of Lemma~\ref{L:Necessary}, $\II$ is a 
$\III$-function and $\JJ$ is a $\JJJ$-function for~$\SSSg$. 

Moreover, assume that $\seqq{\gg_1}{\gg_2}$ lies in~$\Seq\SSS2$ and we have 
$\hh = \gg_1 \OP \gg$ with $\gg \dive \gg_2$. Then we have $\hh \in \SSS$ and 
$\hh \dive \gg_1 \gg_2$, whence $\hh \dive \HH(\gg_1 \gg_2)$, that is, $\hh 
\dive \II(\gg_1, \gg_2)$, since $\HH(\gg_1 \gg_2)$ is an $\SSS$-head of~$\gg_1 
\gg_2$. Hence $\II$ is a maximum $\III$-function for~$\SSSg$. The argument is 
similar for~$\JJ$.
\end{proof}
 
We shall prove the converse implications by using 
Proposition~\ref{P:RecGarGermI}. The main observation is that a maximum 
$\JJJ$-function necessarily satisfies the $\JJJ$-law.

\begin{lemm}
\label{L:MaximumJ}
Assume that $\SSSg$ is a germ that is left-associat\-ive and 
left-cancellat\-ive, and $\JJ$ is a maximum $\JJJ$-function for~$\SSSg$. Then 
$\JJ$ satisfies the $\JJJ$-law.
\end{lemm}

\begin{proof}
Assume $\seqqq{\gg_1}{\gg_2}{\gg_3} \in \Seq\SSS3$ and $\gg_1 \OP \gg_2$ is 
defined. 
Set $\hh =\nobreak \JJ(\gg_1, \gg_2 \OP \JJ(\gg_2, \gg_3))$ and $\hh' = \gg_2 
\OP \ff'$ with $\ff' = \JJ(\gg_1 \OP \gg_2, \gg_3)$.  Our aim is to prove $\hh 
\eqir \hh'$. 

First, $\gg_1 \OP \gg_2$ is defined and $\gg_2 \diveS \gg_2 \OP \JJ(\gg_2, 
\gg_3)$ is true, so, by maximality of $\JJ(\gg_1, \gg_2 \OP \JJ(\gg_2, 
\gg_3))$, we must have $\gg_2 \diveS \JJ(\gg_1, \gg_2 \OP \JJ(\gg_2, \gg_3))$, 
that is, $\gg_2 \diveS \hh$. Write $\hh = \gg_2 \OP \ff$. By assumption, $\hh$ 
belongs to $\JJJS(\gg_1, \gg_2 \OP \JJ(\gg_2, \gg_3))$, hence we have 
$\hh \diveS \gg_2 \OP \JJ(\gg_2, \gg_3)$, that is $\gg_2 \OP \ff \diveS \gg_2 
\OP \JJ(\gg_2, \gg_3)$, which implies $\ff \diveS \JJ(\gg_2, \gg_3)$ as 
$\SSSg$ 
is left-cancellative, whence in turn $\ff \diveS \gg_3$ since $\JJ(\gg_2, 
\gg_3)$ belongs to $\JJJS(\gg_2, \gg_3)$. Now, by assumption, $\gg_1 \OP 
\hh$, that is, $\gg_1 \OP (\gg_2 \OP \ff)$, is defined, and so is $\gg_1 \OP 
\gg_2$ by assumption. Hence $(\gg_1 \OP \gg_2) \OP \ff$ is defined as well, 
and $\ff \diveS \gg_3$ holds. By maximality of $\JJ(\gg_1 \OP \gg_2, \gg_3)$, 
we deduce $\ff \diveS \ff'$, whence 
$\hh \diveS \gg_2 \OP \ff' = \hh'$.

For the other direction, the definition of~$\ff'$ implies that $(\gg_1 \OP 
\gg_2) \OP \ff'$ is defined and $\ff' \diveS \gg_3$ holds. By 
left-associativity, the first relation implies that 
$\gg_1 \OP (\gg_2 \OP \ff')$, that is, $\gg_1 \OP \hh'$, is defined. On the 
other hand, $\gg_2 \OP \ff'$ is defined by assumption and $\ff' \diveS \gg_3$ 
holds, so the maximality of~$\JJ(\gg_2, \gg_3)$ implies $\hh' \diveS \gg_2 \OP 
\JJ(\gg_2, \gg_3)$. Then the maximality of $\JJ(\gg_1, \gg_2 \OP \JJ(\gg_2, 
\gg_3))$ implies $\hh' \diveS \JJ(\gg_1, \gg_2 \OP \JJ(\gg_2, \gg_3))$, that 
is, $\hh' \diveS \hh$. 
So $\hh 
\eqir \hh'$ is satisfied, the desired instance of the $\JJJ$-law.
\end{proof}

Not surprisingly, we have a similar property for a maximum $\II$-function with 
respect to the $\III$-law.

\begin{lemm}
\label{L:MaximumI}
Assume that $\SSSg$ is a germ that is left-associat\-ive and 
left-cancellat\-ive, and $\II$ is a maximum $\III$-function for~$\SSSg$. Then 
$\II$ satisfies the $\III$-law.
\end{lemm}

\begin{proof}
One could mimic the argument used for Lemma~\ref{L:MaximumJ}, but the 
exposition is less convenient, and we shall instead derive the result from 
Lemma~\ref{L:MaximumJ}.

So,  assume that $\II$ is a maximum $\III$-function for~$\SSSg$ and  let $\JJ 
: 
\Seq\SSS2 \to \SSS$ be defined by $\II(\gg_1, \gg_2) = \gg_1 \OP \JJ(\gg_1, 
\gg_2)$.  By Lemma~\ref{L:EquivIJ}, $\JJ$ is a $\JJJ$-function for~$\SSSg$.  
Moreover, the assumption that $\II$ is maximum implies that $\JJ$ is maximum 
too. Indeed, assume $\hh \in \JJJS(\gg_1, \gg_2)$. Then $\gg_1 \OP \hh$ 
is defined and belongs to~$\IIIS(\gg_1, \gg_2)$, hence $\gg_1 \OP \hh 
\diveS \II(\gg_1, \gg_2)$, that is, $\gg_1 \OP \hh \diveS \gg_1 \OP \JJ(\gg_1, 
\gg_2)$, whence $\hh \diveS \JJ(\gg_1, \gg_2)$ by Lemma~\ref{L:LocalDiv}. Then, by 
Lemma~\ref{L:MaximumJ}, $\JJ$ satisfies the $\JJJ$-law.  By 
Lemma~\ref{L:EquivIJ} again, this in turn implies that $\II$ satisfies the 
$\III$-law. 
\end{proof}

We shall now manage to go from a function obeying the $\III$-law to one 
obeying the sharp $\III$-law, that is, force equality instead of 
$\eqirS$-equivalence.

\begin{lemm}
\label{L:Eqir}
Assume that $\SSSg$ is a left-associative and left-cancellative germ, and 
$\II$ 
is a maximum $\III$-function for~$\SSSg$. Then every function $\II' :  
\Seq\SSS2 \to \SSS$ satisfying $\II'(\gg_1, \gg_2) \eqirS \II(\gg_1, \gg_2)$ 
for all~$\gg_1, \gg_2$ is a maximum $\III$-function for~$\SSSg$.
\end{lemm}

\begin{proof}
 Assume that $\II'(\gg_1, \gg_2) \eqirS \II(\gg_1, \gg_2)$ holds for 
all~$\gg_1, \gg_2$. First $\II'$ must be an $\III$-function for~$\SSSg$. 
Indeed, let $\seqq{\gg_1}{\gg_2}$ belong to~$\Seq\SSS2$.  We have $\gg_1 
\diveS  \II(\gg_1, \gg_2) \diveS \II'(\gg_1, \gg_2)$, whence $\gg_1 \diveS 
\II'(\gg_1,  \gg_2)$. Write $\II(\gg_1, \gg_2) = \gg_1 \OP \gg$ and 
$\II'(\gg_1, \gg_2) =  \gg_1 \OP \gg'$. By definition, we have $\gg \diveS 
\gg_2$, that is, $\gg_2 =  \gg \OP \ff$ for some~$\ff$, and, by assumption, 
$\gg' = \gg \OP \ee$ for  some invertible element~$\ee$ of~$\SSS$. We find 
$$\gg_2 = (\gg \OP (\ee \OP \ee\inv)) \OP \ff = ((\gg \OP \ee) \OP \ee\inv)  
\OP \ff = (\gg \OP \ee) \OP (\ee\inv \OP \ff),$$
whence $\gg' \diveS \gg_2$: the second equality comes from~\eqref{E:Germ3}, 
and the last one from the assumption that $\SSSg$ is left-associative. Hence  
$\II'(\gg_1, \gg_2)$ belongs to~$\IIIS(\gg_1, \gg_2)$. 

Now assume $\hh = \gg_1 \OP \gg$ with $\gg \dive \gg_2$. Then we have $\hh  
\diveS \II(\gg_1, \gg_2)$ by assumption, hence $\hh \diveS \II'(\gg_1, \gg_2)$ 
 by transitivity of~$\diveS$. So $\II'$ is a maximum $\III$-function  
for~$\SSSg$.
\end{proof}

\begin{lemm}
\label{L:Strong}
Assume that $\SSSg$ is a germ that is left-associat\-ive, left-cancellat\-ive, 
 
and admits a maximum $\III$-function. Then $\SSSg$ admits a maximum 
$\III$-function that satisfies the sharp $\III$-law.
\end{lemm}

\begin{proof}
Let $\II$ be a maximum $\III$-function for~$\SSSg$, and let~$\SSS'$ be an 
$\eqirS$-selector on~$\SSS$. For $\seqq{\gg_1}{\gg_2}$ in~$\Seq\SSS2$, define 
$\II'(\gg_1, \gg_2)$ to be the unique element of~$\SSS'$ that is  
$\eqir$-equivalent to~$\II(\gg_1, \gg_2)$. Then, by construction, $\II'$ is a 
function from~$\Seq\SSS2$ to~$\SSS$ satisfying $\II'(\gg_1, \gg_2) \eqirS  
\II(\gg_1, \gg_2)$ for every $\seqq{\gg_1}{\gg_2}$ in~$\Seq\SSS2$, hence, by 
Lemma~\ref{L:Eqir}, $\II'$ is a maximum $\III$-function for~$\SSS$. By  
Lemma~\ref{L:MaximumI}, $\II'$ satisfies the $\III$-law, that is, for every 
$\seqqq{\gg_1}{\gg_2}{\gg_3}$ in~$\Seq\SSS3$ such that $\gg_1 \OP \gg_2$ is 
defined, we have
$$\II'(\gg_1, \II'(\gg_2, \gg_3)) \eqir \II'(\gg_1 \OP \gg_2, \gg_3).$$
Now, by definition of a selector, two elements in the image of the 
function~$\II'$ must be equal whenever they are $\eqir$-equivalent. So $\II'$ 
satisfies the sharp $\III$-law.
\end{proof}

We can now complete the proof of Proposition~\ref{P:RecGarGermII}.

\begin{proof}[Proof of Proposition~\ref{P:RecGarGermII}]
Owing to Lemma~\ref{L:Necessary2}, it remains to prove that each of 
\eqref{E:RecGarGermII1} and \eqref{E:RecGarGermII2} implies that $\SSSg$ is a 
Garside germ.

Assume that $\SSSg$ is a germ that is left-associative, left-cancellative, and 
admits a maximum $\III$-function. Then, by Lemma~\ref{L:Strong}, $\SSSg$ 
admits 
an $\III$-function~$\II$ that satisfies the sharp $\III$-law. Therefore, by 
Proposition~\ref{P:RecGarGermI}, $\SSSg$ is a Garside germ. So 
\eqref{E:RecGarGermII1} implies that $\SSSg$ is a Garside germ.

Finally, assume that $\SSSg$ is a germ that is left-associative, 
left-cancellative, and admits a maximum $\JJJ$-function~$\JJ$. As already seen 
in the proof of Lemma~\ref{L:MaximumI}, the function~$\II$ defined 
on~$\Seq\SSS2$ by $\II(\gg_1, \gg_2) = \gg_1 \OP  \JJ(\gg_1, \gg_2)$ is a 
maximum $\III$-function for~$\SSSg$. So \eqref{E:RecGarGermII2} 
implies~\eqref{E:RecGarGermII1} and, therefore, it implies that $\SSSg$ is a 
Garside germ.
\end{proof}

%%%%
\subsection{Noetherian germs}
\label{SS:NoethGerm}

Noetherianity assumptions guarantee the existence of maximum (or minimal) 
elements with respect to left- or right-divisibility. In the context of germs, 
we shall use such assumptions to guarantee the existence of a maximum 
$\JJJ$-function under weak assumptions.

\begin{defi}
\label{D:RNoeth}
A germ~$\SSSg$ is said to be \emph{left-Noetherian} (\resp.\ 
\emph{right-Noether\-ian}) if every nonempty subfamily of~$\SSS$ has a 
least element with respect to the local left-divisibility relation~$\diveS$ 
(\resp.\ the local right-divisibility relation). The germ is called 
\emph{Noetherian} if it is both left- and right-Noetherian. 
\end{defi}

Adapting the proof of Lemma~\ref{L:Noeth}, one easily sees that a germ~$\SSSg$ that is left-associative and left-cancellative is right-Noetherian if and only if, using $\ff \divS$ for ``$\ff \diveS \gg$ and $\ff \noteqir_{\HS{-1.3}\SSS} \gg$'', there exists no infinite 
bounded $\divS$-increasing sequence in~$\SSS$, that is, there is no sequence $\ff_0, \ff_1, ...$ satisfying $\ff_0 \divS \ff_1 \divS ... 
\diveS \gg$ in~$\SSS$. 

 A subfamily~$\XXX$ of a category~$\CCC$ is said to \emph{admit common 
right-multiples} if any two elements of~$\XXX$ that share the same source 
admit a common right-multiple lying in~$\XXX$. The principle for deducing the 
existence of maximal elements from Noetherianity is as follows.

\begin{lemm}
\label{L:Maximum}
Assume that $\SSSg$ is a left-cancella\-tive germ that is right-Noether\-ian, 
$\gg$ belongs to~$\SSS$, and $\XXX$ is a nonempty subfamily of~$\SSS$ such 
that $\ff \diveS \gg$ holds for every~$\ff$ in~$\XXX$.  Then

\ITEM1 The family~$\XXX$ admits a $\divS$-maximal element.

\ITEM2 If $\XXX$ admits common right-multiples, $\XXX$ admits a $\diveS$-greatest element.
\end{lemm}

\begin{proof}
\ITEM1 Let $\ff$ be an arbitrary element of~$\XXX$. Starting from $\ff_0 = 
\ff$, we construct a $\divS$-increasing sequence~$\ff_0, \ff_1,...$ 
in~$\XXX$. As long as $\ff_\ii$ is not $\divS$-maximal in~$\XXX$, we can 
find~$\ff_{\ii+1}$ in~$\XXX$ satisfying $\ff_\ii \divS \ff_{\ii+1} \diveS 
\gg$. The assumption that $\CCC$ is right-Noetherian implies that the 
construction stops after a finite number~$\dd$ of steps. Then by construction, 
the element~$\ff_\dd$ is a $\divS$-maximal element of~$\XXX$. 

\ITEM2 By~\ITEM1, there exists~$\ff$ in~$\XXX$ that is $\divS$-maximal. Let 
$\hh$ be an arbitrary element of~$\XXX$. By assumption, there exists a common 
multiple~$\ff'$ of~$\ff$ and~$\hh$ that lies in~$\XXX$. Now, by assumption, 
$\ff$ is $\divS$-maximal in~$\XXX$, so $\ff \divS \ff'$ is impossible, 
and the only possibility is $\ff' \eqirS \ff$. But, then, $\hh \dive \ff'$ 
implies $\hh 
\diveS \ff$, that is, $\ff$ is a right-multiple of every element of~$\XXX$.
\end{proof}

We can now characterize right-Noetherian  Garside  germs. 

\begin{prop}
\label{P:RecRNoethGerm}
A right-Noether\-ian germ $\SSSg$ is a Garside germ if and only if $\SSSg$ is 
left-associative, left-cancellative, and, for every~$\seqq{\gg_1}{\gg_2}$ 
in~$\Seq\SSS2$, the family $\JJJS(\gg_1, \gg_2)$ admits common 
right-multiples. 
\end{prop}

\begin{proof}
Assume that $\SSSg$ is a Garside germ. Then $\SSSg$ is left-associative and 
left-cancellative by Lemma~\ref{L:Necessary}. Next, by 
Proposition~\ref{P:RecGarGermII}, $\SSSg$ admits a maximum 
$\JJJ$-function~$\JJ$. Then, for every~$\seqq{\gg_1}{\gg_2}$ 
in~$\Seq\SSS2$, the element~$\JJ(\gg_1, \gg_2)$ is a right-multiple of every 
element of~$\JJJS(\gg_1, \gg_2)$, hence a common right-multiple of any two of 
them. So $\JJJS(\gg_1, \gg_2)$ admits common right-multiples. 

Conversely, assume that $\SSSg$ is right-Noetherian  and  satisfies the 
conditions of the statement. Let $\seqq{\gg_1}{\gg_2}$ belong to~$\Seq\SSS2$. 
By 
assumption, the family~$\JJJS(\gg_1, \gg_2)$ admits common 
right-multiples, and it is a subfamily of the right-Noetherian family~$\SSS$. 
Hence, by Lemma~\ref{L:Maximum}, $\JJJS(\gg_1, 
\gg_2)$ admits a $\diveS$-greatest element. Hence, by 
Proposition~\ref{P:RecGarGermII}, $\SSSg$ is a Garside germ. 
\end{proof}

\begin{rema}
In the situation of Proposition~\ref{P:RecRNoethGerm}, the whole 
category~$\Cat(\SSSg)$ must be right-Noetherian. We shall not give the proof 
here.
\end{rema}

When we go to the more special case of a germ that admits local right-lcms, 
that is, in which any two elements that admit a common right-multiple (inside 
the germ) admit a right-lcm (in the germ), we obtain a new sufficient 
condition for recognizing a Garside germ.

\begin{prop}
\label{P:RecLLcmGerm}
A germ~$\SSSg$ that is left-associative, left-cancellative, 
right-Noeth\-erian, 
admits local right-lcms, and satisfies
\begin{equation}
\label{E:RecLLcmGerm}
\BOX{\VR(3,1)for all $\gg, \hh, \hh', \hh''$ in~$\SSS$, if $\gg \OP \hh$ and 
$\gg \OP \hh'$ are defined,

\hfill then $\gg \OP \hh''$ is defined for every right-lcm~$\hh''$ of~$\hh$ 
and~$\hh'$.}
\end{equation} 
is a Garside germ.
\end{prop}

\begin{proof}
Assume that $\seqq{\gg_1}{\gg_2}$ belongs to~$\Seq\SSS2$, and that $\hh$ and 
$\hh'$ lie in~$\JJJS(\gg_1, \gg_2)$. By assumption, we have $\hh \diveS 
\gg_2$ and $\hh' \diveS \gg_2$. As $\SSS$ admits local right-lcms, there must 
exist a right-lcm~$\hh''$ of~$\hh$ and~$\hh'$ that satisfies $\hh'' \diveS 
\gg_2$. If $\SSSg$ satisfies~\eqref{E:RecLLcmGerm}, the assumption that $\gg_1 
\OP \hh$ and $\gg_1 \OP \hh'$ are defined implies that $\gg_1 \OP \hh''$ is 
defined. But, then, $\hh''$ belongs to~$\JJJS(\gg_1, \gg_2)$ and, 
therefore, $\JJJS(\gg_1, \gg_2)$ admits common right-multiples. By 
Proposition~\ref{P:RecRNoethGerm}, it follows that $\SSSg$ is a Garside germ.
\end{proof}

It turns out that,  when right-lcms always exist,  the 
condition~\eqref{E:RecLLcmGerm} occurring in Proposition~\ref{P:RecLLcmGerm} 
follows from a slight strengthening of the left-associativity assumption. We 
shall naturally say that a germ~$\SSSg$ is \emph{right-associative} if the 
counterpart of~\eqref{E:LeftAss} is satisfied, that is, if $\ff \OP \gg$ is 
defined whenever $\ff \OP (\gg \OP \hh)$ is defined, and that $\SSSg$ is 
\emph{associative} if it is both left- and right-associative.

\begin{coro}
\label{C:RecLcmGerm}
A germ that is associative, left-cancellative, right-Noetherian, and admits 
right-lcms is a Garside germ.
\end{coro}

\begin{proof}
Assume that $\SSSg$ satisfies the hypotheses of the statement. We check that 
\eqref{E:RecLLcmGerm} is satisfied. So assume that $\gg \OP \hh$ and $\gg \OP 
\hh'$ are defined and $\hh''$ is a right-lcm of~$\hh$  and~$\hh'$.  Put $\ff = 
\gg \OP \hh$, $\ff' = \gg \OP \hh'$, and let $\fft$ be a right-lcm of~$\ff$ 
and~$\ff'$ (here we use the assumption that $\SSS$ admits right-lcms, and not 
only local right-lcms). 

First, we have $\gg \diveS \ff \diveS \fft$, whence $\gg \diveS \fft$, so 
there exists~$\hht$ satisfying $\fft = \gg \OP \hht$. Then, by 
Lemma~\ref{L:LocalDiv}, $\ff \diveS \fft$ implies $\hh \diveS \hht$ and $\ff' 
\diveS \fft$ implies $\hh' \diveS \hht$.
So $\hht$ is a common right-multiple of~$\hh$ and~$\hh'$, hence it is a 
right-multiple of their right-lcm~$\hh''$: we have $\hht = \hh'' \OP \ee$ for 
some~$\ee$. By assumption, $\gg \OP \hht$, which is $\gg \OP (\hh'' \OP \ee)$, 
is defined. By right-associativity, this implies that $\gg \OP \hh''$ is 
defined, so \eqref{E:RecLLcmGerm} is true. Then, $\SSSg$ is a Garside germ by 
Proposition~\ref{P:RecLLcmGerm}.
\end{proof}

%%%%%%%%
\section{Germs derived from a groupoid}
\label{S:Appli}

We conclude with an application of the previous constructions. Starting from a 
group(oid) together with a distinguished generating family, we shall derive a 
germ, possibly leading in turn to a new category and a new groupoid. The 
latter groupoid is a sort of unfolded version of the initial one. The seminal 
example corresponds to starting with a Coxeter group and arriving at the 
ordinary and dual braid monoid of the associated Artin-Tits group. 

%%%%
\subsection{Tight sequences}
\label{SS:Tight}

Our aim is to associate with every groupoid equipped with a convenient family 
of generators a certain germ, so that this germ is a Garside germ whenever the 
initial groupoid has convenient properties. In order to make the construction 
nontrivial, we shall have to consider sequences of elements in the initial 
groupoid that enjoy a certain length property called tightness.

\begin{defi}
\label{D:MonoidGen}
Assume that $\GGG$ is a groupoid. We say that a subfamily~$\Sigma$ of~$\GGG$ 
\emph{positively generates}~$\GGG$ if every element of~$\GGG$ admits an 
expression that is a $\Sigma$-path (no letter in~$\Sigma\inv$). Then, 
for~$\gg$ in~$\GGG{\setminus}\{1\}$, the \emph{$\Sigma$-length}~$\LT\gg\Sigma$ 
is defined to be the minimal number~$\ell$ such that $\gg$ admits an 
expression by a $\Sigma$-path of length~$\ell$; we complete with 
$\LT{\id\xx}\Sigma = 0$ for each object~$\xx$. 
\end{defi}

Note that, if $\Sigma$ is any family of generators for a groupoid~$\GGG$, then 
$\Sigma \cup \Sigma\inv$ positively generates~$\GGG$. 
Whenever $\Sigma$ positively generates a groupoid~$\GGG$, the $\Sigma$-length 
satisfies the triangular inequality $\LT{\ff\gg}\Sigma \le \LT\ff\Sigma + 
\LT\gg\Sigma$ and, more generally, for every path $(\gg_1 \wdots \gg_\rr)$ 
in~$\GGG$
\begin{equation}
\label{E:Triang}
\LT{\gg_1 \pdots \gg_\rr}\Sigma \le \LT{\gg_1}\Sigma + \pdots + 
\LT{\gg_\rr}\Sigma.
\end{equation}

\begin{defi}
\label{D:Tight}
Assume that $\GGG$ is a groupoid and $\Sigma$ positively generates~$\GGG$. A 
$\GGG$-path $(\gg_1 \wdots \gg_\rr)$ is called \emph{$\Sigma$-tight} if 
$\LT{\gg_1\pdots \gg_\rr}\Sigma = \LT{\gg_1}\Sigma + \pdots + 
\LT{\gg_\rr}\Sigma$ is satisfied.
\end{defi}

\begin{lemm}
\label{L:Tight}
Assume that $\GGG$ is a groupoid and $\Sigma$ positively generates~$\GGG$. 
Then $(\gg_1 \wdots \gg_\rr)$ is $\Sigma$-tight if and only if $(\gg_1 \wdots 
\gg_{\rr-1})$ and $(\gg_1\pdots \gg_{\rr-1}, \gg_\rr)$ are $\Sigma$-tight, if 
and only if $(\gg_2 \wdots \gg_\rr)$ and $(\gg_1, \gg_2\pdots \gg_\rr)$ are 
$\Sigma$-tight.
\end{lemm}

\begin{proof}
To make reading easier, we consider the case of three entries. Assume that 
$(\ff, \gg, \hh)$ is $\Sigma$-tight. By~\eqref{E:Triang}, we have
$\LT{\ff\gg\hh}\Sigma \le \LT{\ff\gg}\Sigma + \LT\hh\Sigma$, whence
$\LT{\ff\gg}\Sigma \ge \LT{\ff\gg\hh}\Sigma - \LT\hh\Sigma= \LT\ff\Sigma + 
\LT\gg\Sigma$.
On the other hand, by~\eqref{E:Triang}, we have $\LT{\ff\gg}\Sigma \le 
\LT\ff\Sigma + \LT\gg\Sigma$. We deduce $\LT{\ff\gg}\Sigma =\nobreak 
\LT\ff\Sigma +\nobreak \LT\gg\Sigma$, and $(\ff, \gg)$ is $\Sigma$-tight. Next 
we have $\LT{(\ff\gg)\hh}\Sigma = \LT\ff\Sigma + \LT\gg\Sigma + \LT\hh\Sigma$, 
whence $\LT{(\ff\gg)\hh}\Sigma = \LT{\ff\gg}\Sigma + \LT\hh\Sigma$ since, as 
seen above, $(\ff, \gg)$ is $\Sigma$-tight. Hence $(\ff\gg,\hh)$ is 
$\Sigma$-tight. 

Conversely, assume that $(\ff, \gg)$ and $(\ff\gg, \hh)$ are $\Sigma$-tight. 
Then we directly obtain $\LT{\ff\gg\hh}\Sigma = \LT{\ff\gg}\Sigma + 
\LT\hh\Sigma = \LT\ff\Sigma + \LT\gg\Sigma + \LT\hh\Sigma$, and $(\ff, \gg, 
\hh)$ is $\Sigma$-tight.

The argument is similar when gathering final entries instead of initial ones.
\end{proof}

Considering the tightness condition naturally leads to introducing two partial 
orderings on the underlying groupoid. 

\begin{defi}
\label{D:Prefix}
Assume that $\GGG$ is a groupoid and $\Sigma$ positively generates~$\GGG$. For 
$\ff, \gg$ in~$\GGG$ with the same source, we say that $\ff$ is a 
\emph{$\Sigma$-prefix} of~$\gg$, written $\ff \Pref\Sigma \gg$, if $(\ff, 
\ff\inv\gg)$ is $\Sigma$-tight. Symmetrically, we say that $\ff$ is a 
\emph{$\Sigma$-suffix} of~$\gg$, if $\ff, \gg$ have the same target and 
$(\gg\ff\inv, \ff)$ is $\Sigma$-tight.
\end{defi}

\begin{lemm}
Assume that $\GGG$ is a groupoid and $\Sigma$ positively generates~$\GGG$. 
Then being a $\Sigma$-prefix and being a $\Sigma$-suffix are partial 
orders on~$\GGG$ and $\id\xx \Pref\Sigma \gg$ holds for every~$\gg$ with 
source~$\xx$.
\end{lemm}

\begin{proof}
As $\LT{\id\yy}\Sigma$ is zero, every sequence $(\gg, \id\yy)$ is 
$\Sigma$-tight for every~$\gg$ with target~$\yy$, so $\gg \Pref\Sigma \gg$ 
always holds, and $\Pref\Sigma$ is reflexive. Next, as the $\Sigma$-length has 
nonnegative values, $\ff \Pref\Sigma \gg$ always implies $\LT\ff\Sigma \le 
\LT\gg\Sigma$. Now, assume $\ff \Pref\Sigma \gg$ and $\gg \Pref\Sigma \ff$. By 
the previous remark, we must have $\LT\ff\Sigma = \LT\gg\Sigma$, whence 
$\LT{\ff\inv\gg}\Sigma = 0$. Hence, $\ff\inv \gg$ is an identity-element, that 
is, $\ff = \gg$ holds. So $\Pref\Sigma$ is antisymmetric. Finally, assume $\ff 
\Pref\Sigma \gg \Pref\Sigma \hh$. Then $(\ff, \ff\inv\gg)$ and $(\gg, 
\gg\inv\hh)$, which is $(\ff\ff\inv\gg, \gg\inv\hh)$, are $\Sigma$-tight. By 
Lemma~\ref{L:Tight}, we deduce that $(\ff, \ff\inv\gg, \gg\inv\hh)$ is 
$\Sigma$-tight, and then that $(\ff, (\ff\inv\gg)(\gg\inv\hh))$, which is 
$(\ff, \ff\inv\hh)$ is $\Sigma$-tight. Hence $\ff \Pref\Sigma \hh$ holds, and 
$\Pref\Sigma$ is transitive. So $\Pref\Sigma$ is a partial order on~$\GGG$. 
Finally, as $\LT{\id\xx}\Sigma$ is zero, every path $(\id\xx, \gg)$ with~$\xx$ 
the source of~$\gg$ is $\Sigma$-tight, and $\id\xx \Pref\Sigma \gg$ holds.

The verifications for $\Sigma$-suffixes are entirely similar.
\end{proof}

By definition, $(\ff, \gg)$ is $\Sigma$-tight if and only if $\ff$ is a 
$\Sigma$-prefix of~$\ff\gg$. For subsequent use, we note the following weak 
compatibility condition of the partial order~$\Pref\Sigma$ with the product.

\begin{lemm}
\label{L:Compat}
Assume that $\GGG$ is a groupoid, $\Sigma$ positively generates~$\GGG$, and 
$\ff, \gg$ are elements of~$\GGG$. If $(\ff, \gg)$ is $\Sigma$-tight and 
$\gg'$ is a $\Sigma$-prefix of~$\gg$, then $\ff \Pref\Sigma \ff \gg' 
\Pref\Sigma \ff\gg$ holds.
\end{lemm}

\begin{proof}
Assume $\gg' \Pref\Sigma \gg$. Then, by definition, $(\gg', \gg'{}\inv\gg)$ is 
$\Sigma$-tight. On the other hand, by assumption, $(\ff, \gg)$, which is 
$(\ff, \gg'(\gg'{}\inv\gg))$, is $\Sigma$-tight. By Lemma~\ref{L:Tight}, it 
follows that $(\ff, \gg', \gg'{}\inv\gg)$ is $\Sigma$-tight. First we deduce 
that $(\ff, \gg')$ is $\Sigma$-tight, that is $\ff \Pref\Sigma \ff\gg'$ holds. 
Next we deduce that $(\ff\gg', \gg'{}\inv\gg)$ is $\Sigma$-tight as well. As 
$\gg'{}\inv\gg$ is also $(\ff\gg')\inv(\ff\gg)$, the latter relation is 
equivalent to $\ff\gg' \Pref\Sigma \ff\gg$.
\end{proof}

%%%%

\subsection{Derived germ}
\label{SS:Derived}

Here is now the basic scheme for constructing a germ. If $\HHH$ is a subfamily 
of a category, we denote by~$\Id\HHH$ the family of all 
identity-elements~$\id\xx$ for~$\xx$ source or target of an element of~$\HHH$.

\begin{defi}
\label{D:Derived}
Assume that $\GGG$ is a groupoid and $\Sigma$ positively generates~$\GGG$. 
For~$\HHH$ included in~$\GGG$, we denote by~$\Der\HHH\Sigma$ the structure 
$(\HHH, \Id\HHH, \OP)$, where $\OP$ is the partial operation on~$\HHH$ such 
that $\hh = \ff \OP \gg$ holds if and only if 
\begin{equation}
\label{E:Product}
\hh = \ff\gg \mbox{\ holds and $(\ff, \gg)$ is $\Sigma$-tight}.
\end{equation}
The structure~$\Der\HHH\Sigma$ is called the \emph{germ derived from~$\HHH$ 
and~$\Sigma$}.
\end{defi}

So we consider the operation that is induced on~$\HHH$ by the ambient product 
of~$\GGG$, but with the additional restriction that the products that are not 
$\Sigma$-tight are discarded. Speaking of germs here is legal, as we 
immediately see.

\begin{lemm}
\label{L:Derived}
Assume that $\GGG$ is a groupoid, $\Sigma$ positively generates~$\GGG$, and 
$\HHH$ is a subfamily of~$\GGG$ that includes~$\Id{\HHH}$. Then 
$\Der\HHH\Sigma$ is a cancellative germ that contains no nontrivial invertible 
element. 
\end{lemm}

\begin{proof}
The verifications are easy. First \eqref{E:Germ1} is satisfied by definition 
of~$\OP$, and so is \eqref{E:Germ2} since we assume that $\Id{\HHH}$ is 
included in~$\HHH$. Next, assume that $\ff, \gg, \hh$ belongs to~$\HHH$ and 
$\ff \OP \gg$, $\gg \OP \hh$ and $(\ff \OP \gg) \OP \hh$ are defined. This 
means that $\ff\gg$, $\gg\hh$, and $(\ff\gg)\hh$ belong to~$\HHH$ and that the 
pairs $(\ff, \gg)$, $(\gg, \hh)$, and $(\ff\gg, \hh)$ are $\Sigma$-tight 
in~$\GGG$. By Lemma~\ref{L:Tight}, $(\ff, \gg, \hh)$, and then $(\ff, 
\gg\hh)$, which is $(\ff, \gg \OP \hh)$, are $\Sigma$-tight. As $\ff(\gg\hh)$ 
belongs to~$\HHH$, we deduce that $\ff \OP \gg\hh$, that is, $\ff \OP (\gg \OP 
\hh)$, is defined, and it is equal to~$\ff(\gg\hh)$. The argument is symmetric 
in the other direction and, therefore, \eqref{E:Germ3} is satisfied. So 
$\Der\HHH\Sigma$ is a germ.

Assume now that $\ff, \gg, \gg'$ belong to~$\HHH$ and $\ff \OP \gg = \ff \OP 
\gg'$ holds. This implies $\ff \gg = \ff \gg'$ in~$\GGG$, whence $\gg = \gg'$. 
So the germ $\Der\HHH\Sigma$ is left-cancellative, hence cancellative by a 
symmetric argument.

Finally, assume $\ee \OP \ee' = \id\xx$ with $\ee, \ee'$ in~$\HHH$. Then we 
must have $\LT\ee\Sigma + \LT{\ee'}\Sigma = \LT{\id\xx}\Sigma = 0$. The only 
possibility is $\LT\ee\Sigma = \LT{\ee'}\Sigma = 0$, whence $\ee = \ee' = 
\id\xx$.
\end{proof}

We now consider Noetherianity conditions. Standard results assert that a 
germ~$\SSSg$ is right-Noetherian if and only if there exists a 
function~$\lambda : \SSS \to \Ord$ such that $\lambda(\ff) < \lambda(\gg)$ 
holds whenever $\ff$ is a proper right-divisor of~$\gg$ in~$\SSS$. We also 
consider \emph{left-Noetherianity}, defined as the well-foundedness of the 
left-divisibility relation~$\diveS$, and characterized by the existence of a 
function~$\lambda : \SSS \to \Ord$ such that $\lambda(\ff) < \lambda(\gg)$ 
holds whenever $\ff$ is a proper left-divisor of~$\gg$ in~$\SSS$.

\begin{lemm}
\label{L:DerivedNoeth}
Assume that $\GGG$ is a groupoid, $\Sigma$ positively generates~$\GGG$, and 
$\HHH$ is a subfamily of~$\GGG$ that includes~$\Id{\HHH}$. Then the derived 
germ $\Der\HHH\Sigma$ is both left- and right-Noetherian. 
\end{lemm}

\begin{proof}
Assume that $\ff, \gg$ lie in~$\HHH$ and we have $\gg = \ff \OP \gg'$ for some 
non-invertible~$\gg'$ that lies in~$\HHH$. By definition of~$\OP$, this 
implies $\LT\gg\Sigma = \LT\ff\Sigma + \LT{\gg'}\Sigma$, whence $\LT\ff\Sigma 
< \LT\gg\Sigma$ as, by definition of the $\Sigma$-length, the 
non-invertibility of~$\gg'$ implies $\LT{\gg'}\Sigma \ge 1$. So the 
$\Sigma$-length witnesses both for the left- and the right-Noetherianity 
of~$(\HHH, \OP)$. 
\end{proof}

Owing to the results of Sections~\ref{S:Germs} and~\ref{S:RecGerm}, the only 
situation when a germ leads to interesting results is when it is 
left-associative. These properties are not automatic for a derived 
germ~$\Der\HHH\Sigma$, but they turn out to be connected with closure under 
$\Sigma$-suffix and $\Sigma$-prefix, where we naturally say that $\HHH$ is 
closed under $\Sigma$-suffix (\resp.\ $\Sigma$-prefix) if every 
$\Sigma$-suffix (\resp.\ $\Sigma$-prefix) of an element of~$\HHH$ lies 
in~$\HHH$.

\begin{lemm}
\label{L:Assoc}
Assume that $\GGG$ is a groupoid, $\Sigma$ positively generates~$\GGG$, and 
$\HHH$ is a subfamily of~$\GGG$ that is closed under $\Sigma$-suffix (\resp.\ 
$\Sigma$-prefix). Then the germ~$\Der\HHH\Sigma$ is left-associative (\resp.\ 
right-associative) and an element~$\ff$ of~$\HHH$ is a local left-divisor 
(\resp.\  right-divisor) of an element~$\gg$ in~$\Der\HHH\Sigma$ if and only 
if $\ff$ is a $\Sigma$-prefix (\resp.\ $\Sigma$-suffix) of~$\gg$.
\end{lemm}

\begin{proof}
Assume that $\HHH$ is closed under $\Sigma$-suffix and $\ff, \gg, \hh$ are 
elements of~$\HHH$ such that $\ff \OP \gg$ and $(\ff \OP \gg) \OP \hh$ are 
defined. Then $\ff\gg$ and $(\ff\gg)\hh$ lie in~$\HHH$ and the pairs $(\ff, 
\gg)$ and $(\ff\gg, \hh)$ are $\Sigma$-tight. By Lemma~\ref{L:Tight}, $(\ff, 
\gg, \hh)$, and then $(\ff, \gg\hh)$ are $\Sigma$-tight. Hence $\gg\hh$ is a 
$\Sigma$-suffix of~$\ff\gg\hh$ in~$\GGG$ and, therefore, by assumption, 
$\gg\hh$ belongs to~$\HHH$. By Lemma~\ref{L:Tight} again, the fact that $(\ff, 
\gg, \hh)$ is $\Sigma$-tight implies that $(\gg, \hh)$ is $\Sigma$-tight, and 
we deduce $\gg\hh = \gg \OP \hh$. Thus the germ~$\Der\HHH\Sigma$ is 
left-associative. 

Assume now that $\ff, \gg$ lie in~$\HHH$ and $\ff$ is a left-divisor of~$\gg$ 
in the germ~$\Der\HHH\Sigma$. This means that $\ff \OP \gg' = \gg$ holds for 
some~$\gg'$ lying in~$\HHH$. Necessarily $\gg'$ is $\ff\inv \gg$, so $(\ff, 
\ff\inv\gg)$ has to be $\Sigma$-tight, which means that $\ff$ is a 
$\Sigma$-prefix of~$\gg$. Conversely, assume that $\ff, \gg$ lie in~$\HHH$ and 
$\ff$ is a $\Sigma$-prefix of~$\gg$. Then $(\ff, \ff\inv \gg)$ is 
$\Sigma$-tight, so $\ff\inv \gg$ is a $\Sigma$-suffix of~$\gg$. The assumption 
that $\HHH$ is closed under $\Sigma$-suffix implies that $\ff\inv \gg$ lies 
in~$\HHH$, and, then, $\ff \OP \ff\inv\gg = \gg$ holds, whence $\ff 
\divve{\Der\HHH\Sigma} \gg$. 

The arguments for right-associativity and right-divisibility 
in~$\Der\HHH\Sigma$ are entirely symmetric, using now the assumption that 
$\HHH$ is closed under $\Sigma$-prefix.
\end{proof}

We now wonder whether $\Der\HHH\Sigma$ is a Garside germ. As $\Der\HHH\Sigma$ 
is Noetherian, it is eligible for the criteria of Section~\ref{SS:NoethGerm}, and 
we are led to looking for the satisfaction of the associated conditions. The 
latter involve the left-divisibility relation of the germ and, therefore, by 
Lemma~\ref{L:Assoc}, they can be formulated inside the base groupoid in terms 
of $\Sigma$-prefixes.

\begin{prop}
\label{P:DerGarI}
Assume that $\GGG$ is a groupoid, $\Sigma$ positively generates~$\GGG$, $\HHH$ 
is a subfamily of~$\GGG$ that is closed under $\Sigma$-suffix, and
\begin{gather}
\label{E:DerGarI1}
\BOX{\VR(3,1.5) If $\gg, \gg'$ lie in~$\HHH$ and admit a common upper 
bound for~$\Pref\Sigma$\\
\null\hfill then they admit a least common upper bound for~$\Pref\Sigma$ 
in~$\HHH$,}\\
\label{E:DerGarI2}
\BOX{\VR(3,1.5) If $\gg, \gg', \gg'', \ff$ lie in~$\HHH$, $\ff \OP \gg$ 
and $\ff \OP \gg'$ are defined and lie in~$\HHH$, \\
\null\hfill \VR(3,1.5)and $\gg''$ is a least common upper bound of~$\gg$ 
and~$\gg'$ for $\Pref\Sigma$, \\
\null\hfill then $\ff \OP \gg''$ is defined.}
\end{gather}
Then $\Der\HHH\Sigma$ is a Garside germ.
\end{prop}

\begin{proof}
By Lemmas~\ref{L:Derived} and~\ref{L:Assoc}, the germ~$\Der\HHH\Sigma$ is 
left-associative, cancellative, Noetherian, and it admits no nontrivial 
invertible element. Hence, by Proposition~\ref{P:RecLLcmGerm}, 
$\Der\HHH\Sigma$ is a Garside germ if it satisfies 
\eqref{E:RecLLcmGerm}. Now, by Lemma~\ref{L:Assoc}, for~$\gg, \hh$ in~$\HHH$, 
the relation $\gg \divve{\Der\HHH\Sigma} \hh$ is equivalent to $\gg 
\Pref\Sigma \hh$ and, therefore, $\hh''$ is a right-lcm of~$\hh$ and~$\hh'$ 
in~$\Der\HHH\Sigma$ if and only it it is a least common upper bound 
of~$\hh$ and~$\hh'$ for~$\Pref\Sigma$. So \eqref{E:DerGarI1} means that $\Der\HHH\Sigma$ admits local right-lcms, whereas \eqref{E:DerGarI2} is a direct reformulation 
of~\eqref{E:RecLLcmGerm}.
\end{proof}

 In the context of Proposition~\ref{P:DerGarI}, as by assumption $\Sigma$ 
generates~$\GGG$, we can weaken~\eqref{E:DerGarI1} and~\eqref{E:DerGarI2} by 
restricting to the case when the elements~$\gg$ and~$\gg'$ lie in~$\Sigma$, 
but one then has to assume that the germ is associative on both sides, that 
is, $\HHH$ is also closed under $\Sigma$-prefix.

\begin{prop}
\label{P:DerGarIBis}
Assume  that $\GGG$  is a  groupoid, $\Sigma$  positively generates~$\GGG$, 
$\HHH$ is a subfamily of~$\GGG$ that is closed under $\Sigma$-suffix and 
$\Sigma$-prefix, and
\begin{gather}
\label{E:DerGarIBis1}
\BOX{\VR(3,1.5) If $\gg, \gg'$ lie in~$\Sigma$ and admit a common upper 
bound for~$\Pref\Sigma$\\
\null\hfill then they admit a least common upper bound for~$\Pref\Sigma$ 
in~$\HHH$,}\\
\label{E:DerGarIBis2}
\BOX{\VR(3,1.5) If $\gg, \gg'$ lie in~$\Sigma$, $\ff$ lies in~$\HHH$, $\ff 
\OP \gg$ and $\ff \OP \gg'$ are defined and lie in~$\HHH$, \\
\null\hfill \VR(3,1.5) and $\gg''$ is a least common upper bound of~$\gg$ 
and~$\gg'$ for $\Pref\Sigma$, \\
\null\hfill then $\ff \OP \gg''$ is defined.}
\end{gather}
Then $\Der\HHH\Sigma$ is a Garside germ.
\end{prop}

\begin{proof}
We first establish using induction on~$\ell$ that all elements~$\gg, \gg'$ 
of~$\HHH$ that admit a common upper bound~$\gg\hh$ for~$\Pref\Sigma$ 
satisfying $\LT{\gg\hh}\Sigma \le \ell$ admit a least common upper bound 
for~$\Pref\Sigma$. For $\ell = 0$ the result is trivial and, for $\ell 
\ge 1$, we argue using induction on $\LT\gg\Sigma + \LT{\gg'}\Sigma$.  
First, the result is trivial if $\LT\gg\Sigma$ or $\LT{\gg'}\Sigma$ is zero. 
Next, if both $\gg$ and~$\gg'$ belong to~$\Sigma$, the result is true 
by~\eqref{E:DerGarIBis1}. Otherwise, assuming $\gg \notin \Sigma$, we write 
$\gg = \gg_1 \OP \gg_2$. Since $\seqq{\gg_1}{\gg_2}$ as well as 
$\seqq{\gg_1\gg_2}\hh$ are $\Sigma$-tight by assumption, 
$\seqqq{\gg_1}{\gg_2}\hh$ and, then, $\seqq{\gg_1}{\gg_2\hh}$ are 
$\Sigma$-tight by Lemma~\ref{L:Tight}. As, moreover, $\HHH$ is closed under 
$\Sigma$-suffix, we have $\gg_1 \Pref\Sigma \gg\hh$ and thus $\gg\hh$ is 
a common upper bound of~$\gg_1$ and~$\gg'$ for~$\Pref\Sigma$, as 
$\Der\HHH\Sigma$ is cancellative by Lemma~\ref{L:Derived}. As we have 
$\LT{\gg_1}\Sigma + \LT{\gg'}\Sigma < \LT\gg\Sigma + \LT{\gg'}\Sigma$, the 
induction hypothesis implies that $\gg_1$ and~$\gg'$ admit a least common 
upper bound for~$\Pref\Sigma$, say~$\gg_1\hh_1$. Then $\gg \OP \hh$ 
is a common upper bound of~$\gg_1 \OP (\gg_2 \OP \hh)$ and $\gg_1 \OP \hh_1$ for~$\Pref\Sigma$, hence $\gg_2 \OP \hh$ is a common upper bound 
of~$\gg_2 \OP \hh$ and $\hh_1$ for~$\Pref\Sigma$. By construction, 
we have $\LT{\gg_2\hh}\Sigma < \LT{\gg\hh}\Sigma$, so the induction hypothesis 
implies that $\gg_2$ and~$\hh_1$ admit a least common upper bound 
for~$\Pref\Sigma$, say~$\gg_2 \hh_2$. By Lemma~\ref{L:Compat}, we have 
$\gg\hh_2 \Pref\Sigma \gg\hh$, and $\gg\hh_2$ is a least common upper bound 
of~$\gg$ and~$\gg'$ for~$\Pref\Sigma$. So \eqref{E:DerGarIBis1} 
implies~\eqref{E:DerGarI1}.

We now establish similarly using induction on~$\ell$ that, if $\ff, \gg, \gg'$ 
lie in~$\HHH$, if $\ff \OP \gg$, $\ff \OP \gg'$ are defined and lie in~$\HHH$, 
and $\gg, \gg'$ admit a least  common upper bound~$\gg''$ for~$\Pref\Sigma$ 
satisfying  $\LT{\gg''}\Sigma \le \ell$, then $\ff \OP \gg''$ is defined. For 
$\ell = 0$, the result is trivial and, for $\ell \ge 1$, we argue using 
induction on $\LT\gg\Sigma + \LT{\gg'}\Sigma$. As above, the result is 
trivial if $\LT\gg\Sigma$ or $\LT{\gg'}\Sigma$ is zero. Next, if both 
$\gg$ and~$\gg'$ belong to~$\Sigma$, the result is true 
by~\eqref{E:DerGarIBis2}. Otherwise, assuming $\gg \notin \Sigma$, we write 
$\gg = \gg_1 \OP \gg_2$. The induction hypothesis implies that, if $\gg_1 
\hh_1$ is the least common upper bound of~$\gg_1$ and~$\gg'$ 
for~$\Pref\Sigma$, then $\ff \OP (\gg_1\hh_1)$ is defined. Next, 
writing $\gg'' = \gg \hh$, the assumption that $\HHH$ is closed under 
$\Sigma$-suffix implies that $\gg_2 \hh$ is the least common upper bound 
of~$\gg_2$ and~$\hh_1$ for~$\Pref\Sigma$. By construction, we have 
$\LT{\gg_2\hh}\Sigma < \LT{\gg''}\Sigma$ and the assumption that $\HHH$ is 
closed under $\Sigma$-prefix implies that $\ff \OP \gg_1$ lies in~$\HHH$, so 
the induction hypothesis implies that $\ff\gg_1\gg_2\hh$, that is, $\ff\gg''$ 
lies in~$\HHH$. So \eqref{E:DerGarIBis2} implies~\eqref{E:DerGarI2}, and we 
can apply Proposition~\ref{P:DerGarI}.
\end{proof}

On the other hand,  if the $\Sigma$-prefix relation~$\Pref\Sigma$ defines an 
upper-semi-lattice on the considered subfamily~$\HHH$, that is, any two 
elements of~$\HHH$ admit a least common upper bound 
for~$\Pref\Sigma$, we obtain a simpler criterion.

\begin{prop}
\label{P:DerGarII}
Assume that $\GGG$ is a groupoid, $\Sigma$ positively generates~$\GGG$, and 
$\HHH$ is a subfamily of~$\GGG$ that is closed under $\Sigma$-prefix and 
$\Sigma$-suffix and any two elements of~$\HHH$ admit a $\Pref\Sigma$-least 
upper bound. Then $\Der\HHH\Sigma$ is a Garside germ.
\end{prop}

\begin{proof}
By Lemma~\ref{L:Assoc}, the germ~$\Der\HHH\Sigma$ is (left- and right-) 
associative, and the existence of least common upper bounds 
for~$\Pref\Sigma$ in~$\GGG$ implies the existence of right-lcms 
in~$\Der\HHH\Sigma$. Moreover, $\Der\HHH\Sigma$ is right-Noetherian by 
Lemma~\ref{L:DerivedNoeth}. Then the latter is a Garside germ by Corollary~\ref{C:RecLcmGerm}.
\end{proof}

When we consider a germ derived from the whole initial groupoid, the 
conditions about closure under prefix and suffix becomes trivial, so it only 
remains the condition about lcms.

\begin{coro}
\label{C:DerGarII}
Assume that $\GGG$ is a groupoid, $\Sigma$ positively generates~$\GGG$, and 
any two elements of~$\GGG$ admit a $\Pref\Sigma$-least common upper 
bound. Then 
$\Der\GGG\Sigma$ is a Garside germ.
\end{coro}

So the main condition for obtaining a Garside germ along the above lines is to 
find a positively generating subfamily~$\Sigma$ of~$\GGG$ such that the 
partial order~$\Pref\Sigma$ admits (local) least common upper bounds.

%%%%
\subsection{The ordinary Artin--Tits monoids}
\label{SS:Artin}

A  first  important  example  of  the  construction  described above is the
construction  of the  Artin--Tits monoids  starting from  arbitrary Coxeter
groups.  We take for $(\GG,  \Sigma)$ a Coxeter system,  and keep the whole
of~$\GG$,  that is, we choose $\HH  = \GG$. Then the monoid generated by the 
germ~$\Der\GG\Sigma$
is  the usual Artin--Tits monoid  associated with~$\GG$, see~\cite{Mic}. We
will use Proposition~\ref{P:DerGarIBis} to show that we have a Garside germ.

First, we recall some well known consequences  of the exchange lemma
(see for example \cite[No.\ 1.4 lemme 3]{Bbk}), which states:

\begin{lemm} 
\label{exchange}
If  $w$ is a $\Sigma$-word of minimal length representing an element~$\gg$ 
of~$\GG$ and $\hh$ is an element of~$\Sigma$ satisfying $\LT{\gg\hh}\Sigma \le 
 \LT\gg\Sigma$, then $\gg\hh$ is represented by some proper subword~$w'$ 
of~$w$.
\end{lemm}

The first consequence (see \cite[No.\ 1.8 Corollaire 1]{Bbk}) is

\begin{prop}  
\label{cor1}  
Assume that $(\GG, \Sigma)$ is a Coxeter system and $\II$ is included 
in~$\Sigma$. Let $\GG_I$ be the
subgroup  of $\GG$ generated by $I$.  Then all minimal $\Sigma$-words 
representing elements of $\GG_I$ are $I$-words.
\end{prop}

The second is

\begin{prop}\label{coset representative}
\label{I}
Assume that $(\GG, \Sigma)$ is a Coxeter system and $\II$ is included 
in~$\Sigma$. Let $\GG_I$ be the subgroup of $\GG$
generated by $I$. Then, for $\ff$ in~$\GG$, the following are equivalent:
\begin{gather}
\label{I1}
\BOX{$\ff$ has no non-trivial $\Sigma$-suffix in $\GG_I$.}\\
\label{I2}
\BOX{\item $\LT{\ff \gg}\Sigma=\LT\ff\Sigma+\LT \gg\Sigma$ holds for all 
$\gg\in
\GG_I$.}\\
\label{I3}
\BOX{\item $\ff$ has minimal $\Sigma$-length in its coset~$\ff\GG_\II$.}
\end{gather}
Further, if $\ff$ satisfies the conditions above, it is the unique element of $\ff \GG_I$ of minimal $\Sigma$-length and, for every~$\gg$ in~$\GG_\II$, every $\Sigma$-suffix of~$\ff\gg$ in~$\GG_\II$ is a $\Sigma$-suffix of~$\gg$.
\end{prop}

The analogous result (reversing left and right) applies to $\GG_I\ff$.

\begin{proof}
Equation \eqref{I2} implies
that $\ff$ has minimal $\Sigma$-length in its coset $\ff \GG_I$.
Conversely, assume that $\ff$ satisfies \eqref{I3} and let $\gg$ be an
element of minimal $\Sigma$-length in $\GG_I$ such that \eqref{I2}
does not hold. Then, if $w$ is a minimal word representing~$\ff$
and $u\aa$ is a minimal word representing $\gg$ with
$\aa$ in~$I$ and $u$ an $I$-word,
by minimality of $\gg$ the word $wu$ is minimal. Since
the word $wua$ is not minimal, Lemma \ref{exchange}
implies that there is a subword $w'$ of $wu$ representing also $\ff\gg$.
Since $ua$ is a minimal word, the word $w'$
must have the form $w''u$ with $w''$ a subword of $w$.
This contradicts the minimality of $\ff$
in $\ff \GG_I$. We have shown the equivalence of \eqref{I2}
and \eqref{I3}.

Conditions \eqref{I1} and \eqref{I3}
are  equivalent. Indeed,  if $\ff$  has a  non-trivial $\Sigma$-suffix 
$\hh \in \GG_I$, then we have
$\ff=\gg\hh$ with $\LT\gg\Sigma=\LT\ff\Sigma-\LT\hh \Sigma$ so that $\ff$
is not an element of minimal $\Sigma$-length in $\ff\GG_I$.
Hence \eqref{I3} implies \eqref{I1}.  Conversely,  if  $\ff'$ 
is an element of minimal $\Sigma$-length in $\ff\GG_I$, 
we  have $\ff=\ff'\hh$  with $\hh\in\GG_I$ and
$\LT\ff\Sigma=\LT{\ff'}+\LT \hh \Sigma$ by \eqref{I2} applied to $\ff'$ (we
use that \eqref{I3} implies \eqref{I2}). 
Then $\hh$  is a  nontrivial $\Sigma$-suffix  of $\ff$ in $\GG_I$,
contradicting \eqref{I1}.

Now \eqref{I2} shows that all elements of $\ff\GG_I$ have a
$\Sigma$-length strictly larger than that of $\ff$, whence the unicity.
Moreover, if an element~$\hh$ of~$\GG_I$ is a $\Sigma$-suffix of $\ff\gg$ with $\gg\in\GG_I$, then
$\LT {\ff\gg}\Sigma=\LT{\ff\gg\hh\inv}\Sigma+\LT \hh\Sigma$ and by \eqref{I2}
we have $\LT{\ff\gg}\Sigma=\LT\ff\Sigma+\LT{\gg}\Sigma$ and
$\LT{\ff\gg\hh\inv}\Sigma= \LT\ff\Sigma+\LT{\gg\hh\inv}\Sigma$ which gives 
$\LT{\gg}\Sigma= \LT{\gg\hh\inv}\Sigma+\LT\hh\Sigma$, so that $\hh$ is a
$\Sigma$-suffix of $\gg$.
\end{proof}

\begin{prop}
For every Coxeter system~$(\GG, \Sigma)$, the germ~$\Der\GG\Sigma$ is a 
Garside germ, and the corresponding monoid is the braid monoid associated 
with~$(\GG, \Sigma)$.
\end{prop}

\begin{proof}
We prove that $(\GG, \Sigma)$ is eligible for Proposition~\ref{P:DerGarIBis}. 
We  first  look  at  \eqref{E:DerGarIBis1};  let  $a,b\in\Sigma$  which  have  
a common upper bound for~$\Pref\Sigma$. We let  $I=\{a,b\}$ and let 
 $\ff$ be the upper
bound.  Write  $\ff=\gg\hh$  where  $\hh$  is of minimal $\Sigma$-length in
$\GG_I\ff$ and $\gg\in\GG_I$ with $\LT\gg\Sigma+\LT\hh\Sigma=\LT\ff\Sigma$.
Then, since $a$ and $b$  are $\Sigma$-prefixes of $f$, by Proposition~\ref{coset representative} they are $\Sigma$-prefixes of $\gg$ thus $\gg$ is a common upper bound of $a$ and $b$ for~$\Pref\Sigma$ in $\GG_I$.
Since every element of $\GG_I$ is equal to a product $aba\ldots $ or $bab\ldots$ with a number of factors at most the order of $ab$ (if finite), $ab$ has finite order and we have $\gg=\Delta_{a,b}$. Thus we found that $\Delta_{a,b}$ is a least common upper bound of $a$ and $b$ for~$\Pref\Sigma$.

Next,  we show \eqref{E:DerGarIBis2}, thus we assume this time that $a,b$ in~$\Sigma$ and$\ff$ in~$\GG$ satisfy $\LT\ff\Sigma+1=\LT{\ff  a}\Sigma=\LT{\ff b}\Sigma$, and we  assume  that  $a,b$  have a common upper bound for~$\Pref\Sigma$, which we have seen is~$\Delta_{a,b}$. We have to show $\LT\ff\Sigma+\LT{\Delta_{a,b}}\Sigma =  \LT{\ff  \Delta_{a,b}}\Sigma$; but this is exactly the fact that \eqref{I1} implies \eqref{I2}.

That the corresponding monoid is the braid monoid of~$(\GG, \Sigma)$ results 
from the presentation of the monoid associated with~$\Der\GG\Sigma$.
\end{proof}

%%%%
\subsection{The dual monoid}

Another   important  example  (which  was  part  of  motivating  the  above
developments)  is  the  dual  monoid  for  spherical Artin groups, or, more
generally,  for  the  braid  groups  associated  to  well-generated complex
reflection groups.

This  time  we  take  for  $(\GG,  \Sigma)$ a well-generated finite complex
reflection  group together with the set of all its reflections. We choose a
Coxeter  element $c$ in~$\GG$, that is, an $h$-regular element where $h$ is
the   highest   reflection   degree   (which   is  unique  since  $\GG$  is
well-generated)  and  we  take  for  $\HH$  the set~$\Div(\cc)$ of all left
$\Sigma$-prefixes  of~$\cc$. Then  the monoid  $\Der\HH\Sigma$ is  the dual
braid  monoid for~$\GG$ in the sense  of David Bessis \cite[8.1]{BesK};
proposition   \cite  [8.8]{BesK}   constructs  this   monoid  according  to
Proposition~\ref{P:DerGarII}.   Bessis  has  shown   \cite[8.2]{BesK}  that  
the  group
presented  by this germ is  the braid group of  $\GG$. The lattice property
for  the  case  of  the  dual  monoid  is  a  deep  result  of which only a
case-by-case  proof is known in general; see \cite[8.14]{BesK}. Previous to
this work, Bessis had given a construction for the real case in \cite{Bes},
using  case-by-case  arguments  for  the  lattice  property. There exists a
case-free  proof  for  finite  Coxeter  groups  due  to Brady and Watt, see
\cite{BW}.

The  same strategy can be applied to Artin groups of affine type. This time
we  take for $(\GG, \Sigma)$ a Coxeter group of affine type with the set of
all its reflections. We choose again a Coxeter element $c$, defined here as
the  product of all  simple reflections in  some chosen order,  and we take
again  for $\HH$ the set~$\Div(\cc)$. It has  been proved that, if $\GG$ is
of  type $\tilde G_2$ or $\tilde C_n$ or  $\tilde A_n$ and in the last case
the order of the simple reflections is such that two consecutive elements do
not commute, then this germ satisfies the assumptions of
Proposition~\ref{P:DerGarII}  and that the group  presented by this germ is
the  corresponding Artin  group. Moreover  these are  the only  cases where
$\Div(\cc)$ is a Garside germ. This last fact and the $\tilde G_2$-case are
unpublished  results of  Crisp and  McCammond; for  the $\tilde A$ case see
\cite{Di1} and for the $\tilde C$ case see~\cite{Di2}.

%%%%%%%%
\bibliographystyle{plain}
%\bibliography{biblio}

\end{document}